\documentclass{article}
\usepackage{amsmath,amssymb,amsfonts,amsthm}
\usepackage{bbm}
\usepackage{color}
\usepackage{enumitem}
\usepackage{authblk}

\newtheorem{theorem}{Theorem}
\newtheorem{lemma}{Lemma}
\newtheorem{proposition}{Proposition}
\newtheorem{definition}{Definition}

\newtheorem{remark}{Remark}

\allowdisplaybreaks[4]

\title{Error Analysis of Three-Layer Neural Network Trained with PGD for Deep Ritz Method}

\author[a]{Yuling Jiao}
\author[b]{Yanming Lai\thanks{Corresponding Author(ylaiam@connect.ust.hk)}}
\author[b]{Yang Wang}
\affil[a]{School of Mathematics and Statistics, and Hubei Key Laboratory of Computational Science, Wuhan University, Wuhan 430072, China}
\affil[b]{Department of Mathematics, The Hong Kong University of Science and Technology, Clear Water Bay, Kowloon, Hong Kong, China }

\begin{document}

\maketitle

\begin{abstract}

Machine learning is a rapidly advancing field with diverse applications across various domains. One prominent area of research is the utilization of deep learning techniques for solving partial differential equations(PDEs). In this work, we specifically focus on employing a three-layer tanh neural network within the framework of the deep Ritz method(DRM) to solve second-order elliptic equations with three different types of boundary conditions. We perform projected gradient descent(PDG) to train the three-layer network and we establish its global convergence. To the best of our knowledge, we are the first to provide a comprehensive error analysis of using overparameterized networks to solve PDE problems, as our analysis simultaneously includes estimates for approximation error, generalization error, and optimization error. We present error bound in terms of the sample size $n$ and our work provides guidance on how to set the network depth, width, step size, and number of iterations for the projected gradient descent algorithm. Importantly, our assumptions in this work are classical and we do not require any additional assumptions on the solution of the equation. This ensures the broad applicability and generality of our results.

\end{abstract}

\section{Introduction}

Machine learning is a field of artificial intelligence that focuses on developing algorithms and models capable of learning from data and making predictions or decisions. It involves the study of statistical techniques and computational algorithms to enable computers to automatically learn and improve from experience. Machine learning finds applications in various domains, such as image and speech recognition, natural language processing, recommendation systems, and autonomous vehicles. By leveraging large datasets and powerful computational resources, machine learning algorithms can uncover patterns, extract insights, and solve complex problems, driving advancements in technology and revolutionizing numerous industries.

Neural networks play a crucial role in machine learning methods, and their approximation capability is an important topic of concern for researchers. \cite{pinkus1999approximation} is a review of the two-layer network approximation results from the 1990s. In recent years, there have been numerous studies on the approximation results of deep neural networks\cite{yarotsky2017error, yarotsky2018optimal, yarotsky2021elementary,  yarotsky2020phase, siegel2020approximation, suzuki2018adaptivity, siegel2022sharp, shen2019deep, lu2021deep, shen2021deep, shen2021neural, zhang2022deep, zhang2024deep, jiao2023deep, guhring2020error, guhring2021approximation, yang2024optimal}. The pioneering work of \cite{yarotsky2017error} introduced a novel approach to approximate smooth functions using neural networks. By constructing neural networks that approximate Taylor expansions and partition of unity, this work provided insights into the approximation of smooth functions. \cite{suzuki2018adaptivity} focuses on approximation in Besov spaces. \cite{lu2021deep} study the approximation of smooth function classes. \cite{guhring2020error, guhring2021approximation} investigate approximation in Sobolev spaces. These works primarily utilize ReLU and sigmoid activation functions, which are commonly used in practice, to study the approximation properties of neural networks in common function spaces. There are also other works that explore neural networks with super-approximation capabilities. \cite{yarotsky2020phase} demonstrates that ReLU-periodic function networks can overcome the curse of dimensionality. \cite{yarotsky2021elementary} shows that sin-arcsin networks can overcome the curse of dimensionality. \cite{shen2021deep} demonstrates that ReLU-floor neural networks can overcome the curse of dimensionality. Shen et al. \cite{shen2021neural} shows that Floor-Exponential-Step networks can overcome the curse of dimensionality. \cite{zhang2022deep} leverage a triangular-wave function and the softsign function to overcome the curse of dimensionality. Jiao et al. \cite{jiao2023deep} demonstrate that ReLU-sine-exponential networks can overcome the curse of dimensionality. However, these results only address the curse of dimensionality at the approximation level and do not consider issues at the training level. Recently,  \cite{zhang2024deep} shows that ReLU neural networks can be approximated by commonly used activation functions. Therefore, the approximation results of ReLU networks can be translated into approximation results for other activation function networks.

The theoretical study of training problems in neural networks is a vast and rapidly evolving field. To gain a comprehensive understanding of this field, the readers are referred to the review article \cite{bartlett2021deep}. \cite{jacot2018neural} proposed the framework of Neural Tangent Kernel (NTK) analysis and proved the global convergence of infinitely wide neural networks. Since then, research based on NTK theory for studying optimization problems in shallow and deep neural networks has become a popular area of study. For further exploration, please refer to \cite{jacot2018neural,allen2019convergence,allen2019learning,allen2019convergence1, du2018gradient, du2019gradient, arora2019fine, zou2019improved, zou2020gradient, cao2020generalization, cao2019generalization, chen2021much, oymak2020toward, nguyen2020global, nguyen2021proof, liu2020linearity, liu2022loss}, and the references mentioned therein. These works demonstrate that training shallow and deep neural networks using gradient descent and stochastic gradient descent algorithms exhibits global convergence. In NTK analysis, the minimum eigenvalue of the kernel matrix plays a crucial role. However, research on the minimum eigenvalue of the kernel matrix is currently quite limited. \cite{oymak2020toward, nguyen2020global, banerjee2023neural, panigrahi2019effect} are outstanding works in this field. Given the importance of understanding the minimum eigenvalue's magnitude, further research in this area is highly needed. Apart from NTK theory, there are other theoretical approaches for studying neural network training. For example, mean field theory \cite{mei2018mean, chizat2018global, fang2021modeling, nguyen2023rigorous, sirignano2020mean} has been utilized in this context as well.

The generalization performance of neural networks is an important metric, and it is related to the complexity of the networks. \cite{bartlett2002rademacher, bartlett2019nearly, golowich2018size, neyshabur2015norm} have studied the VC dimension and pseudo-dimension of neural networks, which are indicators of function complexity. \cite{yang2024nearly} investigates upper bounds on the VC dimension and pseudo-dimension for function classes involving derivatives. Some works have shown that trained neural networks exhibit good generalization performance \cite{arora2019fine, allen2019learning, cao2019generalization, cao2020generalization}. However, these works still have some unresolved issues. In section 6 of \cite{arora2019fine}, the authors discuss the conditions under which the learning task can achieve good performance when the samples satisfy certain underlying functions. However, the function classes discussed in that section are highly limited and do not cover most real-world scenarios. It is mentioned that two-layer networks can generalize well, but there is no clear characterization of the parameter upper bounds for approximation in two-layer networks \cite{mhaskar1996neural, pinkus1999approximation}. Additionally, the function classes that two-layer networks can approximate are extremely limited. For example, it seems that there is currently no result indicating that two-layer networks can approximate functions in Sobolev spaces. The upper bounds in \cite{cao2019generalization} are of a similar nature to those in \cite{arora2019fine}, so they face the same issues. \cite{allen2019learning} assumes the existence of an underlying neural network that achieves low error on the data distribution, and the available data is significantly more than the minimum number of samples required to learn this underlying neural network. However, it is difficult to verify this assumption from the data. The results in \cite{cao2020generalization} depend on the PL coefficient, which in turn depends on the minimum eigenvalue of the NTK. However, as mentioned above, the relationship between the minimum eigenvalue and the number of samples is still not clearly explained.


Using deep learning methods to solve PDEs is a popular field of research. Neural network models have been applied to solve various types of PDEs \cite{han2018solving, weinan2021algorithms, sirignano2018dgm, lu2021deepxde, long2018pde, raissi2018deep}. Different loss functions have been proposed: \cite{weinan2018deep} introduced the deep Ritz method (DRM), \cite{raissi2019physics} proposed Physics-informed Neural Networks (PINNs), \cite{zang2020weak} introduced the weak adversarial network (WAN), and \cite{chen2023friedrichs} proposed Friedrichs learning. At the theoretical level, works such as \cite{weinan2022some, lu2021priori, hong2021priori, xu2020finite, lu2022priori} have provided error analysis for these methods based on the assumption that the solutions of the underlying equations lie in the so-called Barron spaces. Subsequently, works such as \cite{jiao2024error, duan2022convergence, duan2022deep, jiao2022rate, lu2022machine, shin2020convergence, shin2023error, jiao2023convergence, muller2022error} have discovered that this assumption is not necessary. Instead, relying on the classical theory of partial differential equation regularity is sufficient to establish error analysis for these methods, achieving consistency with classical numerical methods such as finite element methods. However, the aforementioned works only consider either the approximation error or the generalization error, or a combination of both. They fail to account for the optimization error that arises from training the neural network using optimization algorithms. Consequently, their analyses are incomplete. In order to provide a solid theoretical foundation for using machine learning to solve PDEs, it is essential to incorporate factors related to the training process into theoretical research. Indeed, this is a significant challenge because, as discussed earlier, the relationship between neural network optimization and generalization is not yet fully understood.


\subsection{Main Results}

Here, we present an informal version of the main results in this work. To keep it concise, we focus on providing the convergence rate of the Robin problem. For the formal and comprehensive version, please refer to Theorem \ref{convergence rate}.
\begin{theorem}[informal version]\label{Robin convergence rate}
Let $u_R$ be the weak solution of Robin problem \eqref{second order elliptic equation}\eqref{robin}. Let $n$ be the sample size. Let the overparametrization condition be
\begin{align*}
A=n^{\frac{415d^4(d+3)5^{d+2}}{288d^3+4}}.
\end{align*}
Let the step size
\begin{align*}
\eta\leq C(d,coe,\Omega)n^{-\frac{103d^3}{144d^3+2}}\frac{1}{A}
\end{align*}
and the iteration step $T=\frac{1}{\eta}$. Let $f_{W_T}$ be the three-layer neural network function trained by PGD after $T$ step. Then with probability at least $1-\frac{C(d,coe,\Omega)}{n}$,
\begin{align*}
&\|f_{W_T}-u_R\|_{H^1(\Omega)}\leq C(d,coe,\Omega)\max\{1,1/\beta\}n^{-\frac{1}{288d^3+4}}.
\end{align*}
\end{theorem}

\subsection{Our Contributions}

The following are the main contributions of our work: 
\begin{enumerate}

\item To the best of our knowledge, we are the first to provide a comprehensive error analysis of using overparameterized networks to solve PDE problems. Our work provides guidance on how to set the network depth, width, step size, and number of iterations for the projected gradient descent algorithm. Our assumption (\ref{Assumption}) in this work is classical, common, and weak, and we do not require any additional assumptions on the solution of the equation. Therefore, our results have strong generality.

\item In this work, we construct three-layer neural networks to approximate functions in Sobolev spaces, extending the results from \cite{de2021approximation} that were originally limited to $W^{s,\infty}(\Omega)$ spaces to general $W^{s,p}(\Omega)$ spaces. This result complements the research in \cite{guhring2021approximation}, which only considers approximation results for deep networks in Sobolev spaces.

\item When considering the generalization error of the neural network function class, it is necessary for the functions to lie within a pre-defined bounded ball in $C(\Omega)$. For PDE problems, truncation techniques that are useful for regression problems cannot be applied due to the consideration of functions in Sobolev spaces. To overcome this challenge, we leverage the properties of the projected gradient descent algorithm to ensure that the network parameters are initialized sufficiently close to the initialization, resulting in a bounded $C(\Omega)$ norm for the iteratively obtained neural network function sequence. Furthermore, we even find that the neural network sequence generated by projected gradient descent lies within a bounded ball in $C^{1}(\Omega)$. This allows us to apply complexity bounds for functions in the $C^{1}$ bounded ball to assist in estimating the generalization error.
\end{enumerate}

In this work, we have referenced the techniques used in \cite{drews2024universal, kohler2022analysis, drews2023analysis} to establish the global convergence of optimization algorithms. However, our work still differs in many aspects from theirs. First of all, our setups are not exactly the same. They use gradient descent to optimize deep neural networks, while we use projected gradient descent to optimize three-layer neural networks. Secondly, our problem setting is different. Their study focuses on regression problems and deals with a class of smooth functions, while our research focuses on PDE problems and deals with a class of Sobolev functions. Consequently, we need to construct neural networks that can approximate functions in Sobolev spaces. Thirdly, our approach to controlling the generalization error also differs from theirs. They utilize the smoothness of activation functions to control the covering number and, consequently, the generalization error. In contrast, we make full use of the properties of the projected gradient descent algorithm to demonstrate that the neural network sequences generated by the algorithm lie within a bounded ball in $C^{1}(\Omega)$. This enables us to apply complexity bounds for functions in the $C^{1}$ bounded ball to estimate the generalization error. Finally, our approach to decomposing the error is also different.

\subsection{Related Works}

Recently, Kohler et al.\cite{drews2024universal, kohler2022analysis, drews2023analysis} investigated the convergence of gradient descent for solving regression problems by overparameterized deep networks. In their work, they considered the approximation error, generalization error, and optimization error simultaneously, and their conclusions hold for smooth function classes, making their findings generalizable. They mitigated the negative impact of overparameterization on generalization error by exploiting the smoothness of activation functions, thereby bridging the gap between optimization error and generalization error under the overparameterized condition. More importantly, differing from the works based on the NTK theory, they presented an alternative approach to studying the global convergence of optimization algorithms. They utilized random initialization to ensure that the inner parameters of the neural network are sufficiently close to the best approximation elements and leveraged the convexity of the neural network function with respect to the outer parameters to establish the global convergence of gradient descent. Furthermore, when proving global convergence, they required the 2-norm of the outer parameters of the best approximation elements to be sufficiently small, and they achieved this by increasing the network width, leading to a proportional decrease in the 2-norm of the outer parameters. By not relying on the NTK framework, they were able to avoid explicitly analyzing the minimum eigenvalue of the NTK matrix.

\subsection{Organization of This Paper}
The paper is structured as follows: Section 2 provides a brief introduction to the Deep Ritz method. Section 3 covers the necessary background knowledge and tools employed in this work. In Section 4, we present the error decomposition, breaking down the total error into approximation error, generalization error, and optimization error. Subsequently, sections 5 to 7 delve into the investigation of these three types of errors individually. The main results are presented in Section 8, and finally, in Section 9, we conclude the paper.

\section{Deep Ritz Method}
	Let $\Omega$ be a convex
	bounded open set in $\mathbb{R}^d$ and assume that $\partial\Omega\in C^{\infty}$. Without loss of generality, we assume that $\Omega\subset(0,1)^d$. We consider the following second order  elliptic equation:
\begin{equation} \label{second order elliptic equation}
		-\triangle u +  w u = f  \text { in } \Omega
\end{equation}
with three kinds of boundary conditions:
\begin{align}
	u&=0\text { on } \partial \Omega\label{dirichlet}\\
	\frac{\partial u}{\partial n}&=g\text { on } \partial \Omega\label{neumann}\\
	u+\beta\frac{\partial u}{\partial n}&=g\text { on } \partial \Omega,\quad\beta\in\mathbb{R},\beta\neq0\label{robin}
\end{align}
which are called Dirichlet, Neumann and Robin boundary conditions, respectively. Note that for the Dirichlet problem, we only consider the homogeneous boundary condition here since the inhomogeneous case can be turned into the homogeneous case by translation.


We make the following assumption on the known terms in the equation:
\begin{equation}\label{Assumption}
f\in L^\infty(\Omega),\quad g\in H^{1/2}(\Omega),\quad w\in L^{\infty}(\Omega),\quad w\geq c_w
\end{equation}
where $c_w$ is some positive constant. In the following we abbreviate $$C\left(\|f\|_{L^\infty(\Omega)},\|g\|_{H^{1/2}(\Omega)},\|w\|_{L^{\infty}(\Omega)},c_w\right),$$  constants depending on the known terms in equation, as $C(coe)$ for simplicity.

For problem $(\ref{second order elliptic equation})(\ref{dirichlet})$, the variational problem is to find $u\in H_0^1(\Omega)$ such that
\begin{subequations}
\begin{equation} \label{variational dirichlet}
(\nabla u,\nabla v)+(wu,v)=(f,v),\quad\forall v\in H_{0}^1(\Omega).
\end{equation}
The corresponding minimization problem is
\begin{equation} \label{minimization dirichlet}
\min_{u\in H_0^1(\Omega)}\frac{1}{2}\int_{\Omega}\left(|\nabla u|^2+wu^2-2fu\right)dx.
\end{equation}
\end{subequations}
{
The variational problem (\ref{variational dirichlet}) is equivalent to the minimization problems (\ref{minimization dirichlet}). This is a well-known result. The following (\ref{variational neumann}) and (\ref{minimization neumann}), (\ref{variational robin}) and (\ref{minimization robin}) have the same kind of relationship. For a reference on this topic, please refer to \cite[Theorem 1.1.2]{ciarlet2002finite}.
}

For problem $(\ref{second order elliptic equation})(\ref{neumann})$, the variational problem is to find $u\in H^1(\Omega)$ such that
\begin{subequations}
\begin{equation} \label{variational neumann}
(\nabla u,\nabla v)+(wu,v)=(f,v)+(g,T_0v)|_{\partial\Omega},\quad\forall v\in H^1(\Omega).
\end{equation}
where $T_0$ is the zero order trace operator. The corresponding minimization problem is
\begin{equation} \label{minimization neumann}
\min_{u\in H^1(\Omega)}\int_{\Omega}\left(\frac{1}{2}|\nabla u|^2+\frac{1}{2}wu^2-fu\right)dx
-\int_{\partial\Omega}gT_0uds.
\end{equation}
\end{subequations}

For problem $(\ref{second order elliptic equation})(\ref{robin})$, the variational problem is to find $u\in H^1(\Omega)$ such that
\begin{subequations}
\begin{equation} \label{variational robin}
(\nabla u,\nabla v)+(wu,v)+\frac{1}{\beta}(T_0u,T_0v)|_{\partial\Omega}=(f,v)+\frac{1}{\beta}(g,T_0v)|_{\partial\Omega},\quad\forall v\in H^1(\Omega).
\end{equation}
The corresponding minimization problem is
\begin{equation} \label{minimization robin}
\min_{u\in H^1(\Omega)}\int_{\Omega}\left(\frac{1}{2}|\nabla u|^2+\frac{1}{2}wu^2-fu\right)dx
+\frac{1}{\beta}\int_{\partial\Omega}\left(\frac{1}{2}(T_0u)^2-gT_0u\right)ds.
\end{equation}
\end{subequations}
The next lemma says that when $g=0$ and $\beta\to0$, the solution of the Robin problem converges to the solution of the Dirichlet problem. 
\begin{lemma}[\cite{jiao2024error}, Lemma 3.4] \label{penalty convergence}
Let assumption $(\ref{Assumption})$ holds. Let $g=0$. Let $u_D$ be the solution of problem $(\ref{variational dirichlet})$(also $(\ref{minimization dirichlet})$) and $u_R$ the solution of problem $(\ref{variational robin})$(also $(\ref{minimization robin})$). There holds
\begin{equation*}
\|u_R-u_D\|_{H^1(\Omega)}\leq C(coe)\beta.
\end{equation*}
\end{lemma}
With this lemma, it suffice to consider the Robin problem since the Dirichlet problem can be handled through a limit process. Define ${L}_{R}$ as a functional on $H^1(\Omega)$:
\begin{equation*}
	{L}_{R}(u):=\int_{\Omega}\left(\frac{1}{2}|\nabla u|^2+\frac{1}{2}wu^2-fu\right)dx
	+\frac{1}{\beta}\int_{\partial\Omega}\left(\frac{1}{2}(T_0u)^2-gT_0u\right)ds.
\end{equation*}
Note that ${L}_{R}$ can be equivalently written as
\begin{align}
	{L}_{R}(u)=&|\Omega|\mathbb{E}_{X\sim U(\Omega)}\left(\frac{1}{2}|\nabla u(X)|^2+\frac{1}{2}w(X)u^2(X)-f(X)u(X)\right)\nonumber\\
	&+\frac{|\partial\Omega|}{\beta}\mathbb{E}_{Y\sim U(\partial\Omega)}\left(\frac{1}{2}(T_0u)^2(Y)-g(Y)T_0u(Y)\right),\label{continuous loss}
\end{align}
where $U(\Omega)$ and $U(\partial\Omega)$ are uniform distribution on $\Omega$ and $\partial\Omega$, respectively. We then introduce a discrete version of ${L}_R$:
\begin{align}
	\widehat{{L}}_{R}(u):=&\frac{|\Omega|}{n}\sum_{i=1}^{n}\left(\frac{1}{2}|\nabla u(X_i)|^2+\frac{1}{2}w(X_i)u^2(X_i)-f(X_i)u(X_i)\right)\nonumber\\
	&+\frac{|\partial\Omega|}{\beta m}\sum_{j=1}^{m}\left(\frac{1}{2}(T_0u)^2(Y_j)-g(Y_j)T_0u(Y_j)\right),\label{discrete loss}
\end{align}
where $\{X_i\}_{i=1}^{n}$ and $\{Y_j\}_{j=1}^{m}$ are i.i.d. random variables according to $U(\Omega)$ and $U(\partial\Omega)$ respectively. Similarly, for the Neumann problem we define
\begin{align}
	{L}_{N}(u):=&|\Omega|\mathbb{E}_{X\sim U(\Omega)}\left(\frac{1}{2}|\nabla u(X)|^2+\frac{1}{2}w(X)u^2(X)-f(X)u(X)\right)\nonumber\\
	&-|\partial\Omega|\mathbb{E}_{Y\sim U(\partial\Omega)}g(Y)T_0u(Y)\label{neumann continuous loss}
\end{align}
and
\begin{align}
	\widehat{{L}}_{N}(u):=&\frac{|\Omega|}{n}\sum_{i=1}^{n}\left(\frac{1}{2}|\nabla u(X_i)|^2+\frac{1}{2}w(X_i)u^2(X_i)-f(X_i)u(X_i)\right)\nonumber\\
	&-\frac{|\partial\Omega|}{m}\sum_{j=1}^{m}g(Y_j)T_0u(Y_j).\label{neumann discrete loss}
\end{align}
We consider the following two minimization problems:
\begin{equation} \label{optimization}
\min_{u\in\mathcal{P}}\widehat{{L}}_R(u),\quad\min_{u\in\mathcal{P}}\widehat{{L}}_N(u)
\end{equation}
The minimization is taken over some parametrized function class $\mathcal{P}$. In this work we choose $\mathcal{P}$ to be a neural network function class. 
Now we discuss in details that the neural network function class we choose. For some $m_1,m_2,A\in\mathbb{N}_{\geq1}$, let $W:=\{(W_s^l,b_s^l):s\in[A],l=1,2,3\}$ be the neural network parameters with $W_s^1\in\mathbb{R}^{m_1\times d},W_s^2\in\mathbb{R}^{m_2\times m_1},W_s^3\in\mathbb{R}^{1\times m_2},
b_s^1\in\mathbb{R}^{m_1},b_s^2\in\mathbb{R}^{m_2},b_s^3\in\mathbb{R}$ for $s\in[A]$ and define three-layer subnetworks $\{f_s^3\}_{s=1}^A$ by
\begin{align*}
f_s^0&=x;\\
f_s^1&=\sigma(f_{s,1}^{org})=\sigma(W_s^1f_s^0+b_s^1);\\
f_s^2&=\sigma(f_{s,2}^{org})=\sigma(W_s^2f_s^1+b_s^2);\\
f_s^3&=W_s^3f_s^2+b_s^3,
\end{align*}
where $\sigma:\mathbb{R}\to\mathbb{R}$, and define $f_W$ to be the sum of the subnetworks:
\begin{align*}
f_W=\sum_{s=1}^{A}f_s^{3}.
\end{align*}
$\{W_s^l\}$ are called weights and $\{b_s^l\}$ are called biases. $\sigma$ is called an activation function. The width of subnetworks $\{f_s^3\}_{s=1}^A$ is defined as $\max\{m_1,m_2\}$. Let ${W}^l=\{({W}_s^l,{b}_s^l),s\in[A]\}$ for $l=1,2,3$. Define $\|W^3\|:=\left[\sum_{s=1}^A\|W_{s}^3\|_F^2+|b_{s}^3|^2\right]^{1/2}$ and $\|W^3\|_1:=\left[\sum_{s=1}^A(\|W_{s}^3\|_{\infty}+|b_{s}^3|)\right]^{1/2}$. For some $B_{inn},B_{out}\in\mathbb{R}$, define the neural network function class $\mathcal{F}_{NN}(\{m_1,m_2,A\},B_{inn},B_{out})$ to be 
\begin{align*}
&\mathcal{F}_{NN}(\{m_1,m_2,A\},B_{inn},B_{out}):=\\
&\{f_W:\|W_s^l\|_F,\|b_s^l\|_2\leq B_{inn}(l=1,2;s\in[A]),\|W^3\|_1\leq B_{out}\}.
\end{align*}
In the following we abbreviate $\mathcal{F}_{NN}(\{m_1,m_2,A\},B_{inn},B_{out})$ as $\mathcal{F}_{NN}$ for simplicity. In this work we choose $m_1=5d,m_2=\binom{2d+1}{d+1}$ and let the activation function $\sigma$ be tanh. Let $B_{\sigma}=\max\{\|\sigma\|_{C(\Omega)},1\}$ and define $B_{\sigma'},B_{\sigma''}$ similarly. Other parameters will be specified in Theorem \ref{convergence rate}.

\begin{remark}
Indeed, as we will see later, the requirement of $\sigma=\tanh$ is only necessary when studying the approximation error. For the analysis of the generalization error and optimization error, it suffices for $\sigma$ to be $C^2$ continuous. But as highlighted by \cite{de2021approximation}, the approach we employ for handling the approximation error is also applicable to other common activation functions such as sigmoid, arctan, and more. Therefore, with slight adjustments, the analysis in this paper can be extended to derive convergence rates for neural networks activated by these functions as well.
\end{remark}

Many algorithms can be used to solve the minimization problem (\ref{optimization}), such as gradient descent, stochastic gradient descent, etc. In this paper we employ projected gradient descent(PGD) to solve (\ref{optimization}). Taking the Robin problem as an example, the update rule of PGD is as follows:
\begin{align*}
W_{t+1}&=\mathrm{proj}_{\mathcal{C}}(W_t-\eta\nabla_{W}\widehat{L}_R(W_t)),\quad t=0,1,\cdots,T-1,
\end{align*}
where $\eta\in\mathbb{R}_{>0}$ is step size, $T\in\mathbb{N}_{>0}$ is the number of iterations and $\mathcal{C}$ is the set to which we project the iterates sequence.
We initialize the outer layer parameters to zero: $W_{s}^3=0,b_s^3=0,s\in[A]$. The initialization of the inner layer parameters will be specified in Theorem \ref{convergence rate}.


\section{Preliminaries}

\subsection{Sobolev Spaces}
In this part we summarize some concepts and results related to Sobolev spaces that we will need for our analysis.

Let $d\in\mathbb{N}_{\geq1}$. Let $\Omega\subset\mathbb{R}^d$ be an open bounded domain with smooth boundary $\partial\Omega$, and without loss of generality we assume that $\Omega\subset[0,1]^d$. Let $\alpha=(\alpha_1,\cdots,\alpha_n)$ be an $n$-dimensional index with $|\alpha|:=\sum_{i=1}^{n}\alpha_i$ and $s$ be a nonnegative integer. We use the notation $D^{\alpha}=\frac{\partial^{|\alpha|}}{\partial x_1^{\alpha_1}\cdots\partial x_d^{\alpha_d}}$. The standard function spaces, including continuous function space, $L^p$ space, Sobolev spaces, are given below.
	\begin{align*}
	&C(\Omega):=\{\text{all the continuous functions defined on }\Omega\},\\
 &C^s(\Omega):=\{f:\Omega\to\mathbb{R}\ |\ D^{\alpha}f\in C(\Omega)\},\\
	&C(\bar\Omega):=\{\text{all the continuous functions defined on }\bar\Omega\},\quad \|f\|_{C(\bar\Omega)}:=\max_{x\in\bar{\Omega}}|f(x)|,\\
	&C^s(\bar\Omega):=\{f:\bar\Omega\to\mathbb{R}\ |\ D^{\alpha}f\in C(\bar\Omega)\},\quad \|f\|_{C^s(\bar\Omega)}:=\max_{x\in\bar{\Omega},|\alpha|\leq s}|D^{\alpha}f(x)|,\\
	&L^p(\Omega):=\left\{f:\Omega\to\mathbb{R}\ |\ \int_{\Omega}|f|^pdx<\infty\right\},\quad \|f\|_{L^p(\Omega)}:=\left[\int_{\Omega}|f|^p(x)dx\right]^{1/p},\quad \forall p\in[1,\infty),\\
	&L^{\infty}(\Omega):=\{f:\Omega\to\mathbb{R}\ |\ \exists C>0 \ s.t.\ |f|\leq C \ a.e.\},\quad\|f\|_{L^{\infty}(\Omega)}:=\inf\{C \ | \ |f|\leq C \ a.e.\}, \\
    & W^{s,p}(\Omega):=\{f:\Omega\to\mathbb{R}\ |\ D^{\alpha}f\in L^p(\Omega),|\alpha|\leq s\},\quad\|f\|_{W^{s,p}(\Omega)}:=\left(\sum_{|\alpha|\leq s}\|D^{\alpha}f\|_{L^p(\Omega)}^p\right)^{1/p}.
\end{align*}
If $s$ is a nonnegative real number, the fractional Sobolev space $W^{s,p}(\Omega)$ can be defined as follows: setting $\theta=s-\lfloor{s}\rfloor$ and
\begin{align*}
&W^{s,p}(\Omega):=\left\{f:\Omega\to\mathbb{R}\ |\ \int_{\Omega}\int_{\Omega}\frac{|D^{\alpha}f(x)-D^{\alpha}f(y)|^p}{|x-y|^{\theta p+d}}dxdy<\infty,\quad\forall |\alpha|=\lfloor{s}\rfloor\right\},\\
&\|f\|_{W^{s,p}(\Omega)}:=\left(\|f\|_{W^{\lfloor{s}\rfloor,p}(\Omega)}^p+\sum_{|\alpha|=\lfloor{s}\rfloor}\int_{\Omega}\int_{\Omega}\frac{|D^{\alpha}f(x)-D^{\alpha}f(y)|^p}{|x-y|^{\theta p+d}}dxdy\right)^{1/p}.
\end{align*}
Let $C_{0}^{\infty }({\Omega} )$ be the set of smooth functions with compact support in $\Omega$, and $W_0^{s,p}(\Omega)$ is the completion space of $C_0^{\infty}(\Omega)$ in $W^{s,p}(\Omega)$. For $s<0$, $W^{s,p}(\Omega)$ is the dual space of $W_{0}^{-s,q}(\Omega)$ with $q$ satisfying $\frac{1}{p}+\frac{1}{q}=1$. When $p=2$, $W^{s,p}(\Omega)$ is a Hilbert space and it is also denoted by $H^s(\Omega)$.

\begin{lemma}[\cite{de2021approximation}, Lemma A.7]\label{estimate for sobolev norm of composition}
Let $d_1, d_2, k \in \mathbb{N}_{\geq1}, \Omega_1 \subset \mathbb{R}^{d_1}, \Omega_2 \subset \mathbb{R}^{d_2}, f \in C^k\left(\Omega_1 ; \Omega_2\right)$ and $g \in C^k\left(\Omega_2 ; \mathbb{R}\right)$. Then it holds that
\begin{align*}
\|g \circ f\|_{W^{k, \infty}(\Omega_1)} \leq 16\left(e^2 k^4 d_2 d_1^2\right)^k\|g\|_{W^{k, \infty}(\Omega_2)} \max _{1 \leq i \leq d_2}\left\|(f)_i\right\|_{W^{k, \infty}(\Omega_1)}^k.
\end{align*}
\end{lemma}

\begin{lemma}[\cite{guhring2021approximation}, Lemma B.5]\label{estimate for sobolev norm of product}
Let $k \in \mathbb{N}_{\geq0}$, and assume that $f \in W^{k, \infty}(\Omega)$ and $g \in W^{k, p}(\Omega)$ with $1 \leq p \leq \infty$. If $k \geq 3$, additionally assume that $f \in C^k(\Omega)$ or $g \in C^k(\Omega)$. Then $f g \in W^{k, p}(\Omega)$ and 
\begin{align*}
\|f g\|_{W^{k, p}(\Omega)} \leq C(k,p,d) \sum_{i=0}^k\|f\|_{W^{i, \infty}(\Omega)}\|g\|_{W^{k-i, p}(\Omega)},
\end{align*}
and, consequently
\begin{align*}
\|f g\|_{W^{k, p}(\Omega)} \leq C(k,p,d)\|f\|_{W^{k, \infty}(\Omega)}\|g\|_{W^{k, p}(\Omega)}.
\end{align*}
\end{lemma}

\begin{lemma}\label{Bramble-Hilbert lemma}
Let $s \in \mathbb{N}, 1 \leq p \leq \infty$. Let $\Omega \subset \mathbb{R}^d$ be open and bounded, $x_0 \in \Omega$ and $r>0$ such that $\Omega$ is star-shaped with respect to $B:=B_{r,\|\cdot\|_{2}}\left(x_0\right)$, and $r>(1 / 2) r_{\max }^{\star}$. Then for any $f \in W^{s, p}(\Omega)$, there exists a polynomial $f^{(poly)}=\sum_{\widetilde{s}=0}^{{s}-1}\sum_{\beta\in P_{\widetilde{s},d}}c_{\beta}x^{\beta}$ such that
$$
\left|f-f^{(poly)}\right|_{W^{k, p}(\Omega)} \leq C(s, d, \gamma) h^{s-k}|f|_{W^{s, p}(\Omega)} \quad \text { for } k=0,1, \ldots, s,
$$
where $h=\operatorname{diam}(\Omega)$ and $\gamma=h/{r_{\max}^*}$. Moreover,
$|c_{\beta}|\leq C(s,d,R)r^{-d/p}\|f\|_{W^{s-1,p}(\Omega)}$ for all $\beta\in P_{\widetilde{s},d}$ with $0\leq|\widetilde{s}|\leq s-1$.
\end{lemma}
\begin{proof}
The existence of $f_{poly}$ is precisely the well-known Bramble-Hilbert lemma, and its proof can be found in \cite[Lemma 4.3.8]{brennermathematical}. The upper bound estimation of \(|c_{\beta}|\) can be found in \cite[Lemma B.9]{guhring2020error}.
\end{proof}

\begin{definition}[trace operator]
Let $m\in\mathbb{N}_{\geq1}$. Let $\Omega$ be a $C^{m}$ domain. The trace operator $T=(T_0,T_1,\cdots,T_{m-1})$ is defined by
\begin{align*}
T:H^m(\Omega)&\to H^{m-1/2}(\partial\Omega)\times H^{m-3/2}(\partial\Omega)\times\dots\times H^{1/2}(\partial\Omega)\\
v&\mapsto(T_0v,T_1v,\cdots,T_{m-1}v).
\end{align*}
\end{definition}

\begin{lemma}[trace theorem]\label{trace theorem}
Let $m\in\mathbb{N}_{\geq1}$. Let $\Omega$ be a $C^{m}$ domain. Then the trace operator is continuous and surjective.
\end{lemma}
\begin{proof}
See \cite[Theorem 7.33]{adams2003sobolev}. 
\end{proof}
The next three lemmas are classical regularity results for the PDEs we are concerned with in this work.
\begin{lemma}[\cite{grisvard2011elliptic}, Theorem 2.4.2.5] \label{uD regularity}
Let assumption $(\ref{Assumption})$ holds. Let $u_D$ be the weak solution of the Dirichlet problem \eqref{second order elliptic equation}\eqref{dirichlet} but replacing the right-hand side of \eqref{dirichlet} with some $H^{3/2}(\partial\Omega)$ function $g$. Then $u_D\in H^2(\Omega)$ and
\begin{align*}
\|u_D\|_{H^2(\Omega)}\leq C(\Omega,w)(\|f\|_{L^2(\Omega)}+\|g\|_{H^{3/2}(\partial\Omega)}).
\end{align*}
\end{lemma}

\begin{lemma}[\cite{grisvard2011elliptic}, Theorem 2.4.2.7] \label{uN regularity}
Let assumption $(\ref{Assumption})$ holds. Let $u_N$ be the solution of the Neumann problem \eqref{second order elliptic equation}\eqref{neumann}. Then $u_N\in H^2(\Omega)$ and
\begin{align*}
\|u_N\|_{H^2(\Omega)}\leq C(\Omega,w)(\|f\|_{L^2(\Omega)}+\|g\|_{H^{1/2}(\partial\Omega)}).
\end{align*}
\end{lemma}
\begin{lemma} \label{uR regularity}
Let assumption $(\ref{Assumption})$ holds. Let $u_R$ be the solution of the Robin problem \eqref{second order elliptic equation}\eqref{robin}. Then $u_R\in H^2(\Omega)$ and 
\begin{align*}
\|u_R\|_{H^2(\Omega)}\leq C(\Omega,w)(\|f\|_{L^{2}(\Omega)}+\|g\|_{H^{3/2}(\partial\Omega)}).
\end{align*}
\end{lemma}
\begin{proof}
The proof is given in the appendix. 
\end{proof}

\subsection{Convex Optimization}

\begin{definition}
Let $n\in\mathbb{N}_{\geq1}$. A subset $C$ of $\mathbb{R}^n$ is called convex if 
\begin{align*}
\alpha x+(1-\alpha)y\in C,\quad\forall x,y\in C,\forall\alpha\in[0,1].
\end{align*}
\end{definition}

\begin{definition}
Let $n\in\mathbb{N}_{\geq1}$. Let $C$ be a nonempty convex subset of $\mathbb{R}^n$. We say that a function $f:C\to\mathbb{R}$ is convex if
\begin{align*}
f(\alpha x+(1-\alpha)y)\leq\alpha f(x)+(1-\alpha)f(y),\quad\forall x,y\in C,\forall\alpha\in[0,1].
\end{align*}
\end{definition}

\begin{lemma}[\cite{bertsekas2009convex}, Proposition 1.1.7]\label{convex iff condition}
Let $n\in\mathbb{N}_{\geq1}$. Let $C$ be a nonempty convex subset of $\mathbb{R}^n$ and let $f:\mathbb{R}^n\to\mathbb{R}$ be differentiable over an open set that contains $C$. Then $f$ is convex over $C$ if and only if 
\begin{align*}
f(z)\geq f(x)+\nabla f(x)^T(z-x),\quad\forall x,z\in C.
\end{align*}
$f$ is strictly convex over $C$ if and only if the above inequality is strict whenever $x\neq z$.
\end{lemma}

\begin{lemma}[\cite{bertsekas2009convex}, Proposition 1.1.10]\label{convex hessian condition}
Let $n\in\mathbb{N}_{\geq1}$. Let $C$ be a nonempty convex subset of $\mathbb{R}^n$ and let $f:\mathbb{R}^n\to\mathbb{R}$ be twice continuously differentiable over an open set that contains $C$. Let $\nabla^2f(x)$ be the Hessian of $f$ at $x$. If $\nabla^2f(x)$ is positive semi-definite for all $x\in C$, then $f$ is convex over $C$; if $\nabla^2f(x)$ is positive definite for all $x\in C$, then $f$ is strictly convex over $C$.
\end{lemma}

\begin{definition}[projection]
Let $n\in\mathbb{N}_{\geq1}$. Let $C$ be a nonempty convex subset of $\mathbb{R}^n$ and let $z$ be a vector in $\mathbb{R}^n$. The vector that minimizes $\|z-x\|_2$ over $x\in C$ is called the projection of $z$ on $C$ and denoted as $\mathrm{proj}_C(z)$.

\end{definition}

\begin{lemma}[\cite{bertsekas2009convex}, Proposition 1.1.9]\label{projection operator property1}
Let $n\in\mathbb{N}_{\geq1}$. Let $C$ be a nonempty convex subset of $\mathbb{R}^n$ and let $z$ be a vector in $\mathbb{R}^n$. $\mathrm{proj}_C(z)$ is uniquely determined. Futhermore, a vector $x^*=\mathrm{proj}_C(z)$ if and only if 
\begin{align*}
(z-x^*)^T(x-x^*)\leq0,\quad\forall x\in C.
\end{align*}
\end{lemma}

\begin{lemma}\label{projection operator property2}
Let $n\in\mathbb{N}_{\geq1}$. Let $C$ be a nonempty convex subset of $\mathbb{R}^n$ and let $z,z'$ be two vectors in $\mathbb{R}^n$. There holds
\begin{align*}
\|\mathrm{proj}_C(z)-\mathrm{proj}_C(z')\|_2^2&\leq(z-z')^T(\mathrm{proj}_C(z)-\mathrm{proj}_C(z')),\\
\|\mathrm{proj}_C(z)-\mathrm{proj}_C(z')\|_2&\leq\|z-z'\|_2.
\end{align*}
\end{lemma}
\begin{proof}
We have
\begin{align*}
&\|\mathrm{proj}_C(z)-\mathrm{proj}_C(z')\|_2^2=(\mathrm{proj}_C(z)-\mathrm{proj}_C(z'))^T(\mathrm{proj}_C(z)-\mathrm{proj}_C(z'))\\
&=(\mathrm{proj}_C(z)-z)^T(\mathrm{proj}_C(z)-\mathrm{proj}_C(z'))+(z-z')^T(\mathrm{proj}_C(z)-\mathrm{proj}_C(z'))\\
&\quad\ +(z'-\mathrm{proj}_C(z'))^T(\mathrm{proj}_C(z)-\mathrm{proj}_C(z')).
\end{align*}
By Lemma \ref{projection operator property1}, the first and the third term on the right-hand side are nonpositive, hence we obtain the first inequality. By Cauchy-Schwarz inequality, the right-hand side of the first inequality is not greater than $\|z-z'\|_2\|\mathrm{proj}_C(z)-\mathrm{proj}_C(z')\|_2$. Then the second inequality follows.
\end{proof}

\begin{definition}
Let $n\in\mathbb{N}_{\geq1},L\in\mathbb{R}_{>0}$. Let $C$ be a nonempty convex subset of $\mathbb{R}^n$. Let $f:\mathbb{R}^n\to\mathbb{R}$ be differentiable over an open set that contains $C$. $f$ is called $L$-strongly smooth over $C$ if for any $x,y\in C$,
\begin{align*}
f\left(y\right)&\leq f(x)+\nabla f(x)^T\left(y-x\right)+\frac{L}{2}\left\|y-x\right\|_2^2.
\end{align*}
\end{definition}

\begin{lemma}\label{smoothness inequality}
Let $n\in\mathbb{N}_+$. Given $x,y\in\mathbb{R}^n$. For a differentiable function $f:\mathbb{R}^n\to\mathbb{R}$, if there exists some constant $L\in\mathbb{R}$ such that for any $a\in[0,1]$,
\begin{align*}
\|\nabla f[(1-a)x+ay]-\nabla f(x)\|_2\leq aL\|y-x\|_2,
\end{align*}
then
\begin{align*}
f\left(y\right)&\leq f(x)+\nabla f(x)^T\left(y-x\right)+\frac{L}{2}\left\|y-x\right\|_2^2.
\end{align*}
\end{lemma}
\begin{proof}
According to Cauchy-Schwarz inequality,
\begin{align*}
&f(y)-f(x)=\int_{0}^{1}\nabla f((1-a)x+ay)^T(y-x)da\\
&=\int_{0}^{1}[\nabla f((1-a)x+ay)-\nabla f(x)]^T(y-x)da+\int_{0}^{1}\nabla f(x)^T(y-x)da\\
&\leq\|y-x\|_2\int_{0}^{1}\|\nabla f((1-a)x+ay)-\nabla f(x)\|_2da+\nabla f(x)^T(y-x)\\
&\leq \frac{L}{2}\|y-x\|_2^2+\nabla f(x)^T(y-x).
\end{align*}
\end{proof}

\subsection{Function Classes Complexity and Concentration Inequality}

\begin{definition}[Rademacher complexity]
	Let $n\in\mathbb{N}_{\geq1}$. The Rademacher complexity of a set $A \subset \mathbb{R}^n$ is defined as
	\begin{equation*}
		\mathfrak{R}_n(A) = \mathbb{E}_{\{\varsigma_i\}_{i=1}^n}\left[\sup_{a\in A}\frac{1}{n}\sum_{i=1}^n \varsigma_i a_i\right],
	\end{equation*}
	where,   $\{\varsigma_i\}_{i=1}^n$ are $n$ i.i.d  Rademacher variables with $\mathbb{P}(\varsigma_i = 1) = \mathbb{P}(\varsigma_i = -1) = \frac{1}{2}.$
	The Rademacher complexity of  function class $\mathcal{G}$ associate with random sample $\{X_i\}_{i=1}^{n}$ is defined as
	\begin{equation*}
		\mathfrak{R}_n(\mathcal{G}) = \mathbb{E}_{\{X_i,\varsigma_i\}_{i=1}^{n}}\left[\sup_{g\in \mathcal{G}}\frac{1}{n}\sum_{i=1}^n \varsigma_i g(X_i)\right].
	\end{equation*}
\end{definition}

\begin{definition}[covering number]
An $\epsilon$-cover of a set $T$ in a metric space $(S, \tau)$
is a subset $T_c\subset S$ such  that for each $t\in T$, there exists a $t_c\in T_c$ such that $\tau(t, t_c) \leq\epsilon$. The $\epsilon$-covering number of $T$, denoted as $\mathcal{N}(\epsilon, T,\tau)$ is  defined to be the minimum cardinality among all $\epsilon$-cover of $T$ with respect to the metric $\tau$.
\end{definition}


The Rademacher complexity and the covering number of $\mathcal{G}$ share the following relation.
\begin{lemma}[\cite{jiao2024error}, Lemma 5.7]\label{to covering}
Let $n\in\mathbb{N}_{\geq1}$. For any function class $\mathcal{G}$ with $|g|\leq B_{\mathcal{G}}$ for all $g\in\mathcal{G}$,
\begin{align*}
\mathfrak{R}_n(\mathcal{G})\leq\inf_{0<\delta<B_{\mathcal{G}}/2}\left(4\delta+\frac{12}{\sqrt{n}}\int_{\delta}^{B_{\mathcal{G}}/2}\sqrt{\ln\mathcal{N}(\epsilon,\mathcal{G},\|\cdot\|_{\infty})}d\epsilon\right).
\end{align*}
\end{lemma}

\begin{lemma}[McDiarmid’s inequality]\label{McDiarmid’s inequality}
Let $n\in\mathbb{N}_{\geq1}$. Let $g$ be a function from $\Omega^n$ to $\mathbb{R}$. Suppose that function $g$ satisfies the bounded differences property: there exists constants $\{c_i\}_{i=1}^n$ such that for any $x_1,\cdots,x_n\in\Omega$,
\begin{align*}
\sup_{\widetilde{x}_i\in\Omega}|g(x_1,\cdots,\widetilde{x}_i,\cdots,x_n)-g(x_1,\cdots,x_i,\cdots,x_n)|\leq c_i,\quad i\in[n].
\end{align*}
Let $\{X_i\}_{i=1}^n$ be independent variables, then for any $\tau>0$,
\begin{align*}
|g(X_1,\cdots,X_n)-\mathbb{E}g(X_1,\cdots,X_n)|\leq\tau
\end{align*}
with probability at least $1-2e^{-\frac{2\tau^2}{\sum_{i=1}^{n}c_i^2}}$.
\end{lemma}
\begin{proof}
See \cite[Theorem 2.9.1]{vershynin2018high}.
\end{proof}

\subsection{Miscellaneous}
\begin{lemma}\label{power and logarithm}
For $0<p<\frac{1}{e}$,
\begin{align*}
\ln x\leq x^p,\quad x\in\left[\left(\frac{2}{p}\ln\frac{1}{p}+\frac{1}{p}\right)^{1/p},+\infty\right).
\end{align*}
For $p\geq\frac{1}{e}$,
\begin{align*}
\ln x\leq x^p,\quad x\in\left(0,+\infty\right).
\end{align*}
\end{lemma}
\begin{proof}
Define $g(x):=x^p-\ln x$. Then $g'(x)=px^{p-1}-\frac{1}{x}$, from which we conclude that $g(x)$ decreases on $\left(0,\left(\frac{1}{p}\right)^{1/p}\right)$ and increases on $\left(\left(\frac{1}{p}\right)^{1/p},+\infty\right)$. For $0<p<\frac{1}{e}$, assume 
$\frac{k}{e^k}\leq p\leq\frac{k-1}{e^{k-1}}$ with some $k\in\mathbb{N}_{\geq2}$. Since $p\leq\frac{k-1}{e^{k-1}}\leq\frac{e^{(k-1)/2}}{e^{k-1}}=\frac{1}{e^{(k-1)/2}}$, we have $k\leq2\ln\frac{1}{p}+1$. The result follows from the facts that 
\begin{align*}
g\left(\left(\frac{k}{p}\right)^{1/p}\right)=\frac{1}{p}\left({k}-\ln\frac{k}{p}\right)\geq0.
\end{align*}
and $\left(\frac{2}{p}\ln\frac{1}{p}+\frac{1}{p}\right)^{1/p}\geq\left(\frac{k}{p}\right)^{1/p}\geq\left(\frac{1}{p}\right)^{1/p}$.

For $p\geq\frac{1}{e}$, the minimum
\begin{align*}
g\left(\left(\frac{1}{p}\right)^{1/p}\right)=\frac{1}{p}\left({1}-\ln\frac{1}{p}\right)\geq0.
\end{align*}
Hence $g(x)$ is nonnegative on $(0,+\infty)$.
\end{proof}

\begin{lemma}\label{power and exponential}
For $q>{e}$,
\begin{align*}
y^q\leq e^y,\quad y\in\left[2q\ln q+q,+\infty\right).
\end{align*}
For $0<q\leq{e}$,
\begin{align*}
y^q\leq e^y,\quad y\in\left(0,+\infty\right).
\end{align*}
\end{lemma}
\begin{proof}
We can obtain the result directly by letting $x=y^{q}$ and $p=\frac{1}{q}$ in Lemma \ref{power and logarithm}.
\end{proof}

\begin{lemma}\label{xp convex}
For any $a,b\geq0,p\geq1$,
\begin{align*}
(a+b)^p\leq 2^{p-1}(a^p+b^p).
\end{align*}
\end{lemma}
\begin{proof}
The function $f(x)=x^p$ is convex on $(0,+\infty)$ since $f''(x)=p(p-1)x^{p-2}\geq0$. Then 
\begin{align*}
\left(\frac{a+b}{2}\right)^p\leq\frac{1}{2}a^p+\frac{1}{2}b^p.
\end{align*}
\end{proof}

\section{Error Decomposition}
In this section, we investigate some properties and relations of the continuous loss and the discrete loss. Building upon these results, we derive an error decomposition, which serves as the starting point for our subsequent work.

The continuous loss $L_R$ and $L_N$ have the following properties.
\begin{lemma}\label{loss difference and variable difference}
For any $u\in H^1(\Omega)$,
\begin{align*}
C(coe)\|u-u_R\|_{H^1(\Omega)}^2&\leq{L}_{R}\left(u\right)-{L}_{R}\left(u_R\right)
\leq\max\{1,1/\beta\}C(\Omega,coe)\|u-u_R\|_{H^1(\Omega)}^2,\\
C(coe)\|u-u_N\|_{H^1(\Omega)}^2&\leq{L}_{N}\left(u\right)-{L}_{N}\left(u_N\right)
\leq C(\Omega,coe)\|u-u_R\|_{H^1(\Omega)}^2.
\end{align*}
\end{lemma}
\begin{proof}
We only present a proof for the Robin problem. The proof is adapted to the Neumann problem after minor modifications. For any $u\in\mathcal{F}_{NN}$, set $v=u-u_R$, then
\begin{align*}
	&{L}_{R}\left(u\right)={L}_{R}\left(u_R+v\right)\\
	&= \frac{1}{2}(\nabla (u_R+v),\nabla (u_R+v))_{L^{2}(\Omega)}+\frac{1}{2}(u_R+v,u_R+v)_{L^{2}(\Omega;w)}-\langle u_R+v, f\rangle_{L^2({\Omega})}\\
	&\quad\ +\frac{1}{2\beta}(T_0u_R+T_0v,T_0u_R+T_0v)_{L^2({\partial \Omega})}- \frac{1}{\beta}\langle {T_0u_R+T_0v}, g\rangle_{L^2({\partial \Omega})}\\
	&=\frac{1}{2}(\nabla u_R,\nabla u_R)_{L^{2}(\Omega)}+\frac{1}{2}(u_R,u_R)_{L^{2}(\Omega;w)}-\langle u_R, f\rangle_{L^2({\Omega})} +\frac{1}{2\beta}(T_0u_R,T_0u_R)_{L^2({\partial \Omega})}\\
	&\quad - \frac{1}{\beta}\langle {T_0u_R}, g\rangle_{L^2({\partial \Omega})} +\frac{1}{2}(\nabla v,\nabla v)_{L^{2}(\Omega)}+\frac{1}{2}(v,v)_{L^{2}(\Omega;w)}+\frac{1}{2\beta}(T_0v,T_0v)_{L^2({\partial \Omega})}\\
	&\quad\ +\left[(\nabla u_R,\nabla v)_{L^{2}(\Omega)}+(u^*,v)_{L^{2}(\Omega;w)}-\langle v, f\rangle_{L^2({\Omega})}\right.\\
	&\quad\ +\left.\frac{1}{\beta}(T_0u_R,T_0v)_{L^2({\partial \Omega})}- \frac{1}{\beta}\langle {T_0v}, g\rangle_{L^2({\partial \Omega})}\right]\\
	&={L}_{R}\left(u_R\right)+\frac{1}{2}(\nabla v,\nabla v)_{L^{2}(\Omega)}+\frac{1}{2}(v,v)_{L^{2}(\Omega;w)}+\frac{1}{2\beta}(T_0v,T_0v)_{L^2({\partial \Omega})},
\end{align*}
where the last equality is due to the fact that $u_R$ is the solution of equation $(\ref{variational robin})$. Hence
\begin{align*}	C(coe)\|v\|_{H^1(\Omega)}^2&\leq{L}_{R}\left(u\right)-{L}_{R}\left(u_R\right)\\
	&=\frac{1}{2}(\nabla v,\nabla v)_{L^{2}(\Omega)}+\frac{1}{2}(v,v)_{L^{2}(\Omega;w)}+\frac{1}{2\beta}(T_0v,T_0v)_{L^2({\partial \Omega})}\\
	&\leq\max\{1,1/\beta\}C(\Omega,coe)\|v\|_{H^1(\Omega)}^2,
\end{align*}
where in the third step we make use of Lemma \ref{trace theorem}.

\end{proof}


The following result provides an estimation of the difference between the minimum value of the discrete loss and the minimum value of the continuous loss.
\begin{lemma}\label{continuous minimum and empirical minimum}
\begin{align*}
&\left|\inf_{f_W\in\mathcal{F}_{NN}}\widehat{L}_{R}(f_W)-L_R(u_R)\right|\\
&\leq\inf_{f_W\in\mathcal{F}_{NN}}[{L}_{R}({f_W})-L_R(u_R)]\\
&\quad+\max\left\{\sup_{f_W\in\mathcal{F}_{NN}}[\widehat{L}_{R}({f_W})-{L}_{R}({f_W})],\sup_{f_W\in\mathcal{F}_{NN}}[{L}_{R}({f_W})-\widehat{L}_{R}({f_W})]\right\},\\
&\left|\inf_{f_W\in\mathcal{F}_{NN}}\widehat{L}_{N}(f_W)-L_R(u_N)\right|\\
&\leq\inf_{f_W\in\mathcal{F}_{NN}}[{L}_{N}({f_W})-L_N(u_N)]\\
&\quad+\max\left\{\sup_{f_W\in\mathcal{F}_{NN}}[\widehat{L}_{N}({f_W})-{L}_{N}({u})],\sup_{f_W\in\mathcal{F}_{NN}}[{L}_{N}({f_W})-\widehat{L}_{N}({f_W})]\right\}.
\end{align*}
\end{lemma}
\begin{proof}
We only give a proof for the Robin problem. Let $f_{\bar{W}}$ be any function in $\mathcal{F}_{NN}$. We have
\begin{align*}
&\inf_{f_W\in\mathcal{F}_{NN}}\widehat{L}_{R}(f_W)-L_R(u_R)\leq\widehat{L}_{R}(f_{\bar{W}})-L_R(u_R)\\
&=[\widehat{L}_{R}(f_{\bar{W}})-{L}_{R}(f_{\bar{W}})]+[{L}_{R}(f_{\bar{W}})-L_R(u_R)]\\
&\leq\sup_{f_W\in\mathcal{F}_{NN}}[\widehat{L}_{R}({f_W})-{L}_{R}({f_W})]+[{L}_{R}(f_{\bar{W}})-L_R(u_R)],
\end{align*}
which implies
\begin{align}\label{minimizer1}
&\inf_{f_W\in\mathcal{F}_{NN}}\widehat{L}_{R}(f_W)-L_R(u_R)\nonumber\\
&\leq\sup_{f_W\in\mathcal{F}_{NN}}[\widehat{L}_{R}({f_W})-{L}_{R}({f_W})]+\inf_{f_W\in\mathcal{F}_{NN}}[{L}_{R}({f_W})-L_R(u_R)].
\end{align}
On the other hand, let $\{u_k\}_{k=1}^{\infty}\subseteq\mathcal{F}_{NN}$ be a sequence such that $\lim_{k\to\infty}\widehat{L}_{R}(u_k)=\inf_{f_W\in\mathcal{F}_{NN}}\widehat{L}_{R}(f_W)$. For any $k\in\mathbb{N}_{\geq1}$,
\begin{align*}
L_R(u_R)-\widehat{L}_{R}(u_k)&=[{L}_{R}(u_R)-{L}_{R}({u_k})]+[{L}_{R}({u_k})-\widehat{L}_{R}(u_k)]\\
&\leq{L}_{R}({u_k})-\widehat{L}_{R}(u_k)
\leq\sup_{f_W\in\mathcal{F}_{NN}}[{L}_{R}({f_W})-\widehat{L}_{R}({f_W})].
\end{align*}
Hence
\begin{align}\label{minimizer2}
L_R(u_R)-\inf_{f_W\in\mathcal{F}_{NN}}\widehat{L}_{R}(f_W)
\leq\sup_{f_W\in\mathcal{F}_{NN}}[{L}_{R}({f_W})-\widehat{L}_{R}({f_W})].
\end{align}
Combining (\ref{minimizer1}) and (\ref{minimizer2}) yields the conclusion.

\end{proof}

The following lemma decomposes the total error into three different types of errors, which we will handle using different tools.

\begin{lemma}\label{error decompostion}
(1) Consider Robin problem (\ref{variational robin}). Suppose $f_{W_T}\in\mathcal{F}_{NN}$. There holds
\begin{align*}
&{L}_{R}(f_{W_T})-L_R(u_R)\\
&\leq\left[\widehat{L}_{R}(f_{W_T})-\inf_{f_W\in\mathcal{F}_{NN}}\widehat{L}_{R}(f_W)\right]+\inf_{f_W\in\mathcal{F}_{NN}}[{L}_{R}(f_{W})-{L}_{R}(u_R)]
\\
&\quad\ +2\max\left\{\sup_{f_W\in\mathcal{F}_{NN}}[{L}_{R}(f_{W})-\widehat{L}_{R}(f_{W})],\sup_{f_W\in\mathcal{F}_{NN}}[\widehat{L}_{R}(f_{W})-{L}_{R}(f_{W})]\right\}.
\end{align*}
(2) Consider Neumann problem (\ref{variational neumann}). Suppose $f_{W_T}\in\mathcal{F}_{NN}$. There holds
\begin{align*}
&{L}_{N}(f_{W_T})-L_N(u_N)\\
&\leq\left[\widehat{L}_{N}(f_{W_T})-\inf_{f_W\in\mathcal{F}_{NN}}\widehat{L}_{N}(f_W)\right]
+\inf_{f_W\in\mathcal{F}_{NN}}[{L}_{N}(f_{W})-{L}_{N}(u_N)]\\
&\quad\ +2\max\left\{\sup_{f_W\in\mathcal{F}_{NN}}[{L}_{N}(f_{W})-\widehat{L}_{N}(f_{W})],\sup_{f_W\in\mathcal{F}_{NN}}[\widehat{L}_{N}(f_{W})-{L}_{N}(f_{W})]\right\}.
\end{align*}
\end{lemma}
\begin{remark}
The first term reflects the distance between the discrete loss value at the $T$-th iteration and the minimum value of the discrete loss, and is therefore referred to as the optimization error. According to Lemma \ref{loss difference and variable difference}, the second term shows the distance between the set $\mathcal{F}_{NN}$ and the target function $u_R$($u_N$), which is exactly the classical definition of approximation error. The third term is the generalization error, which measures the uniform difference between the continuous loss and the discrete loss over $\mathcal{F}_{NN}$. 
\end{remark}
\begin{proof}
We only give a proof for the Robin problem. It follows from Lemma \ref{continuous minimum and empirical minimum} that
\begin{align*}
&{L}_{R}(f_{W_T})-L_R(u_R)\\
&=[{L}_{R}(f_{W_T})-\widehat{L}_{R}(f_{W_T})]+\left[\widehat{L}_{R}(f_{W_T})-\inf_{f_W\in\mathcal{F}_{NN}}\widehat{L}_{R}(f_W)\right]\\
&\quad+\left[\inf_{f_W\in\mathcal{F}_{NN}}\widehat{L}_{R}(f_W)-{L}_{R}(u_R)\right]\\
&\leq\left[\widehat{L}_{R}(f_{W_T})-\inf_{f_W\in\mathcal{F}_{NN}}\widehat{L}_{R}(f_W)\right]
+\inf_{f_W\in\mathcal{F}_{NN}}[{L}_{R}({f_W})-L_R(u_R)]\\
&\quad\ +2\max\left\{\sup_{f_W\in\mathcal{F}_{NN}}[\widehat{L}_{R}({f_W})-{L}_{R}({f_W})],\sup_{f_W\in\mathcal{F}_{NN}}[{L}_{R}({f_W})-\widehat{L}_{R}({f_W})]\right\}.
\end{align*}
where in the second step we use the fact that $f_{W_T}\in\mathcal{F}_{NN}$.

\end{proof}

\section{Approximation Error}

In this section, we study the error of neural network approximation for functions in Sobolev spaces. We follow the proof strategy presented in \cite{de2021approximation}. To be specific, the main process can be divided into two steps. According to the Bramble-Hilbert lemma (Lemma \ref{Bramble-Hilbert lemma}), there exists polynomials that can locally approximate functions in Sobolev spaces. The first step in \cite{yarotsky2017error} is to approximate those polynomials by neural network functions. Secondly, following the ideas presented in \cite{yarotsky2017error}, \cite{de2021approximation} constructs an approximate partition of unity $\{\Phi_{j}^N\}_{j\in\{1,2,\cdots,N\}^d}$ by neural network functions to achieve localization. However, the study in \cite{de2021approximation} did not cover all integer order Sobolev spaces, as they only studied the approximation of neural network functions to $W^{s,\infty}$ functions. We generalize their result to the approximation of $W^{s,p}$ functions with $1\leq p\leq\infty$.

\begin{remark}
Here, we consider only the tanh neural network, i.e., $\sigma = \tanh$. However, as mentioned in \cite{de2021approximation}, this construction method is applicable to other smooth activation functions.
\end{remark}

We first introduce some notations we need in this section. Let $N\in\mathbb{N}_{\geq1}$. For $j\in\{1,\cdots,N\}^d$, define
\begin{align*}
I_j^N:=\prod_{i=1}^d\left(\frac{j_i-1}{N}, \frac{j_i}{N}\right),\quad J_j^N:=\prod_{i=1}^d\left(\frac{j_i-2}{N}, \frac{j_i+1}{N}\right).
\end{align*}
Define
\begin{align*}
\mathcal{V}:=\left\{v \in \mathbb{Z}^d: \max _{1 \leq i \leq d}\left|v_i\right| \leq 1\right\}.
\end{align*}
 For some $k\in\mathbb{N}_{\geq1}$ and some accuracy parameter $\epsilon>0$, let
\begin{align}\label{alpha}
\alpha=N\ln \left(\frac{(2 k)^{k+1}(N k)^k}{e^k \epsilon}\right).
\end{align}
For $y\in\mathbb{R}$, define
\begin{align*}
& \phi_1^N(y):=\frac{1}{2}-\frac{1}{2} \sigma\left(\alpha\left(y-\frac{1}{N}\right)\right), \\
& \phi_j^N(y):=\frac{1}{2} \sigma\left(\alpha\left(y-\frac{j-1}{N}\right)\right)-\frac{1}{2} \sigma\left(\alpha\left(y-\frac{j}{N}\right)\right) \quad \text { for } 2 \leq j \leq N-1,
\\
&\phi_N^N(y):=\frac{1}{2} \sigma\left(\alpha\left(y-\frac{N-1}{N}\right)\right)+\frac{1}{2} .
\end{align*}
and
\begin{align*}
\Phi_j^{N}(x):=\prod_{i=1}^d \phi_{j_i}^{N}\left(x_i\right),\quad x\in\mathbb{R}^d.
\end{align*}
\begin{remark}
Compared to \cite{de2021approximation}, we have removed $R$ in the definition of $\alpha$, which does not affect the subsequent conclusions.
\end{remark}
The following lemma shows that $\{\Phi_j\}_{j\in\{1,2,\cdots,N\}^d}$ is an approximate partition of unity.
\begin{lemma}[\cite{de2021approximation}, Lemma 4.1 and Lemma 4.2]\label{POU}
Let $k \in \mathbb{N}_{\geq0}$. If $0<\epsilon<1 / 4$ in (\ref{alpha}), then
\begin{align*}
\left\|\sum_{v \in \mathcal{V}} \Phi_{j+v}^{N}-1\right\|_{W^{k, \infty}\left(I_j^N\right)} &\leq 2^{d k} d \epsilon,\\
\left\|\Phi_{j+v}^{N}\right\|_{W^{k, \infty}\left(I_j^N\right)} &\leq \max \left\{1,(2 k)^{2 k} \alpha^k\right\} \epsilon,\quad\forall v \in \mathbb{Z}^d\text{ with }\|v\|_{\infty} \geq 2.
\end{align*}
\end{lemma}

The next lemma constructs neural networks $\{f_{NN,i}^{(poly)}\}_{i\in\{1,2\cdots,N\}^d}$ that approximate the target function in each local region $I_i^N$ by utilizing the results of neural network approximation of polynomials (\cite[Lemma 3.5]{de2021approximation}) and the Bramble-Hilbert lemma. 
\begin{lemma}\label{local approximation}
Let $N\in\mathbb{N}_{\geq1}$. Let $i\in\{1,\cdots,N\}^d$. Let $s,d \in \mathbb{N}, 1 \leq p \leq \infty$. Let $\epsilon>0$. Then for any $f \in W^{s, p}((0,1)^d))$, there exists a two-layer tanh neural network $f_{NN,i}^{(poly)}$ with width no more than $3(s-1)\left\lceil\frac{{s}}{2}\right\rceil\binom{s+d-2}{s-1}$ such that for $k=0,1,\cdots,s-1$,
\begin{align*}
\left\|f-f_{NN,i}^{(poly)}\right\|_{W^{k,p}(J_{i}^N)}
&\leq C(s,d,p)\|f\|_{W^{s,p}(J_i^N)}\left(\left(\frac{1}{N}\right)^{s-k}+\epsilon\right),\\
\left\|f-f_{NN,i+v}^{(poly)}\right\|_{W^{k,p}(I_i^N)}&\leq \|f\|_{W^{k,p}(I_i^N)}+C(s,d,p)\|f\|_{W^{s-1,p}(J_{i+v}^N)},\quad \forall\|v\|_{\infty} \geq 2.
\end{align*}
Furthermore, the weights of $f_{NN,i}^{(poly)}$ are upper bounded by 
\begin{align*}
C(s,d,p)\|f\|_{W^{s-1,p}(J_i^N)}{N}^{d/p}\epsilon^{-(s-1)/2}. 
\end{align*}
\end{lemma}
\begin{remark}
This lemma is crucial for extending the approximation of \cite{de2021approximation} from $W^{s,\infty}$ functions to $W^{s,p}$ functions. While \cite{de2021approximation} only utilizes the Bramble-Hilbert lemma for $W^{s,\infty}$ functions, we employ the Bramble-Hilbert lemma for arbitrary $W^{s,p}$ functions.
\end{remark}
\begin{proof}
Let $\Omega=J_{i}^N$ in Lemma \ref{Bramble-Hilbert lemma}, then there exists 
\begin{align*}
f_{i}^{(poly)}(x)=\sum_{\widetilde{s}=0}^{{s}-1}\sum_{\beta\in P_{\widetilde{s},d}}c_{\beta,i}x^{\beta}.
\end{align*}
with $|c_{\beta,i}|\leq C(s,d)(\frac{3}{2N})^{-d/p}\|f\|_{W^{s-1,p}(J_i^N)}$ such that for $k=0,1,\cdots,s-1$,
\begin{align}\label{local approximation1}
\left\|f-f_{i}^{(poly)}\right\|_{W^{k,p}(J_{i}^N)} \leq C(s, d) \left(\frac{3\sqrt{2}}{N}\right)^{s-k}|f|_{W^{s,p}(J_i^N)}.
\end{align}
By Lemma 3.5 in \cite{de2021approximation}, there exists a two-layer tanh neural network $f_{NN,\widetilde{s}}^{(mono)}:(0, 1)^d \rightarrow \mathbb{R}^{\binom{\widetilde{s}+d-1}{\widetilde{s}}}$ of width $3\left\lceil\frac{\widetilde{s}+1}{2}\right\rceil\binom{\widetilde{s}+d-1}{\widetilde{s}}$ such that
\begin{align*}
\max _{\beta \in P_{\widetilde{s}, d}}\left\|x^\beta-\left(f_{NN,\widetilde{s}}^{(mono)}(x)\right)_{\iota(\beta)}\right\|_{W^{k, \infty}((0, 1)^d)} \leq \epsilon.
\end{align*}
where $\iota: P_{\widetilde{s}, d} \rightarrow\left\{1, \ldots\left|P_{\widetilde{s}, d}\right|\right\}$ is a bijection. Furthermore, the weights of the network scale as $O\left(\epsilon^{-\widetilde{s} / 2}(\widetilde{s}(\widetilde{s}+2))^{3(\widetilde{s}+2)^2}\right)$ for small $\epsilon$ and large $\widetilde{s}$. Define
\begin{align*}
f_{NN,i}^{(poly)}(x)=c_{0,i}+\sum_{\widetilde{s}=1}^{{s}-1}\sum_{\beta\in P_{\widetilde{s},d}}c_{\beta,i}\left(f_{NN,\widetilde{s}}^{(mono)}(x)\right)_{\iota(\beta)}.
\end{align*}
There holds
\begin{align}\label{local approximation2}
&\left\|f_{i}^{(poly)}-f_{NN,i}^{(poly)}\right\|_{W^{k,p}(J_{i}^N)}
\leq\sum_{\widetilde{s}=1}^{{s}-1}\sum_{\beta\in P_{\widetilde{s},d}}|c_{\beta,i}|\left\|x^{\beta}-\left(f_{NN,\widetilde{s}}^{(mono)}(x)\right)_{\iota(\beta)}\right\|_{W^{k,p}(J_{i}^N)}\nonumber\\
&\leq\epsilon\left(\frac{3}{N}\right)^{d/p}\sum_{\widetilde{s}=1}^{{s}-1}\sum_{\beta\in P_{\widetilde{s},d}}|c_{\beta,i}|
\leq C(s,d,p)\|f\|_{W^{s-1,p}(J_i^N)}\epsilon.
\end{align}
Combining (\ref{local approximation1}) and (\ref{local approximation2}) yields the first result.

For $\beta\in P_{\widetilde{s},d}$, $\|x^{\beta}\|_{W^{k,p}(I_j^N)}\leq k^{1/p}\widetilde{s}^kN^{-d/p}$. Thus we have for $\|v\|_{\infty} \geq 2$,
\begin{align*}
\|f_{i+v}^{(poly)}\|_{W^{k,p}(I_i^N)}\leq\sum_{\widetilde{s}=0}^{{s}-1}\sum_{\beta\in P_{\widetilde{s},d}}|c_{\beta,i+v}|\|x^{\beta}\|_{W^{k,p}(I_i^N)}
\leq C(s,d,p)\|f\|_{W^{s-1,p}(J_{i+v}^N)}.
\end{align*}
By a similar derivation as obtaining (\ref{local approximation2}), we have
\begin{align*}
&\left\|f_{i+v}^{(poly)}-f_{NN,i+v}^{(poly)}\right\|_{W^{k,p}(I_i^N)}
\leq C(s,d,p)\|f\|_{W^{s-1,p}(J_{i+v}^N)}\epsilon.
\end{align*}
The second result then follows from the triangle inequality.
\end{proof}
The following lemma states that multiplication operations can be approximated by neural networks.
\begin{lemma}[\cite{de2021approximation}, Corollary 3.7]\label{multiplication approximation} 
Let  $d \in \mathbb{N}, k \in \mathbb{N}_{\geq0}$  and  $M\in\mathbb{R}_{>0}$. Then for every  $\epsilon>0$, there exists a two-layer tanh neural network  $f_{NN}^{(mul)}:[-M, M]^{d} \rightarrow \mathbb{R}$  of width $3\left\lceil\frac{d+1}{2}\right\rceil\binom{2d-1}{d}$  such that
\begin{align*}
\left\|f_{NN}^{(mul)}(x)-\prod_{i=1}^{d} x_{i}\right\|_{W^{k, \infty}([-M,M]^d)} \leq \epsilon.
\end{align*}
Furthermore, the weights of the network are upper bounded by $C(d,k)\epsilon^{-d / 2}{M}^{3d^2/4+2d}$.
\end{lemma}
\begin{remark}
The upper bound for the weights in \cite{de2021approximation} does not explicitly depend on the range of the input variable, denoted as M. However, in the subsequent estimation of the convergence rate, we cannot simply treat M as a constant. Therefore, here we explicitly calculate the relationship between M and the impact on the weights based on the proof of Lemma 3.5 in \cite{de2021approximation}.
\end{remark}

The following properties of tanh function is also needed.
\begin{lemma}[\cite{de2021approximation}, Lemma A.4]\label{estimate for higher derivative of activation}
Let $k \in \mathbb{N}$. Then for all $x \in \mathbb{R}$,
\begin{align*}
\left|\sigma^{(k)}(x)\right| \leq(2 k)^{k+1} \min \{\exp (-2 x), \exp (2 x)\}.
\end{align*}
\end{lemma}
Now we are able to present a fundamental result on neural network function approximation to functions in $W^{s,p}$.
\begin{proposition}\label{app err pre}
Let $s,d \in \mathbb{N}_{\geq1}, 1 \leq p \leq \infty$. Let $N\in\mathbb{N}_{\geq1}$. Then for any $f \in W^{s, p}((0,1)^d)$, there exists a three-layer neural network $f_{NN}$ with the first hidden 
 layer width of $N^d\left[3(s-1)\left\lceil\frac{{s}}{2}\right\rceil\binom{{s}+d-2}{{s}-1}+2d\right]$ and the second hidden layer width of $3N^d\left\lceil\frac{d+2}{2}\right\rceil\binom{2d+1}{d+1}
$, such that for $k=1,\cdots,s-1$,
\begin{align*}
&\|f_{NN}-f\|_{W^{k,p}((0,1)^d)}\leq C(s,d,p,f)\ln^k \left(\frac{(2 k)^{k+1}(N k)^kN^{s+2d}}{e^k}\right)\left(
\frac{1}{N}\right)^{s-k}.
\end{align*}
Furthermore, the weights of $f_{NN}$ are bounded by 
\begin{align*}
C(s,d,p,f)N^{\frac{3}{4p}d^3+\frac{k+p+7}{2p}d^2+\frac{2sp+2k+2p+15}{4p}d+\frac{1}{2}(s^2+3)}
\end{align*}
provided $N\geq 2k$. 
\end{proposition}


\begin{proof}
Let $\epsilon$ be $\epsilon_1$ in (\ref{alpha}). Let $f_{NN,j}^{(poly)}$ be the neural network in Lemma \ref{local approximation} with $\epsilon=\epsilon_2$ and $f_{NN}^{(mul)}$ be the neural network in Lemma \ref{multiplication approximation} with $\epsilon=\epsilon_3$. Let
\begin{align*}
f_{NN}=\sum_{j\in\{1,\cdots,N\}^d}f_{NN}^{(mul)}(f_{NN,j}^{(poly)}(x),\phi_{j_1}^N(x_1),\cdots,\phi_{j_d}^N(x_d)).
\end{align*}
We decompose the total error into three terms and then bound each term in succession:
\begin{align}\label{app err1}
&\|f-f_{NN}\|_{W^{k,p}((0,1)^d)}\nonumber\\
&\leq\left\|f-\sum_{j\in\{1,\cdots,N\}^d}f\Phi_j^{N}\right\|_{W^{k,p}((0,1)^d)}+\sum_{j\in\{1,\cdots,N\}^d}\left\|(f-f_{NN,j}^{(poly)})\Phi_j^{N}\right\|_{W^{k,p}((0,1)^d)}\nonumber\\
&
\quad+\sum_{j\in\{1,\cdots,N\}^d}\|f_{NN,j}^{(poly)}(x)\Phi_j^{N}(x)-f_{NN}^{(mul)}(f_{NN,j}^{(poly)}(x),\phi_{j_1}^N(x_1),\cdots,\phi_{j_d}^N(x_d))\|_{W^{k,p}((0,1)^d)}.
\end{align}
For the first term, we firstly study the integral over $I_{i}^N$ with $i\in\{0,\cdots,N\}^d$. Applying Lemma \ref{estimate for sobolev norm of product} and Lemma \ref{POU}, we have
\begin{align*}
&\left\|f-\sum_{j \in\{1, \ldots, N\}^d} f \Phi_j^{N}\right\|_{W^{k, p}\left(I_i^N\right)}\\
&\leq C(k,d,p)\|f\|_{W^{k, p}\left(I_i^N\right)}\left\|1-\sum_{v \in \mathcal{V}} \Phi_{i+v}^{N}\right\|_{W^{k, \infty}\left(I_i^N\right)}
\\
&\quad+C(k,d,p)\|f\|_{W^{k, p}\left(I_i^N\right)}\left\|\sum_{\substack{j \in\{1, \ldots, N\}^d\\
j-i \notin \mathcal{V}}} \Phi_j^{N}\right\|_{W^{k, \infty\left(I_i^N\right)}}\\
&\leq  C(k,d,p)\|f\|_{W^{k, p}\left(I_i^N\right)}\left(2^{k d} d \epsilon_1+N^d(2 k)^{2 k} \alpha^k \epsilon_1\right)\\
&\leq C(s,d,p)\|f\|_{W^{s-1, p}(I_i^N)}\epsilon_1 N^{k+d}\ln^k \left(\frac{(2 k)^{k+1}(N k)^k}{e^k \epsilon_1}\right).
\end{align*}
Then
\begin{align*}
&\left\|f-\sum_{j \in\{1, \ldots, N\}^d} f\Phi_j^{N}\right\|_{W^{k, p}((0,1)^d)}\\
&\leq C(s,d,p)\|f\|_{W^{s-1, p}((0,1)^d)}\epsilon_1 N^{k+2d}\ln^k \left(\frac{(2 k)^{k+1}(N k)^k}{e^k \epsilon_1}\right).
\end{align*}
Similarly, for the second term of (\ref{app err1}), we firstly handle the integration over $I_i^N$. Employing Lemma \ref{estimate for sobolev norm of product}, we have
\begin{align*}
&\sum_{j \in\{1, \ldots, N\}^d}\left\|\left(f-f_{NN,j}^{(poly)}\right) \Phi_j^{N} \right\|_{W^{k, p}\left(I_i^N\right)} \\
&\leq C(k,d,p)\sum_{j \in\{1, \ldots, N\}^d}\sum_{\widetilde{k}=0}^{k}\left\|f-f_{NN,j}^{(poly)}\right\|_{W^{k-\widetilde{k}, p}\left(I_i^N\right)}\left\|\Phi_j^{N}\right\|_{W^{\widetilde{k}, \infty}\left(I_i^N\right)}
\\
&\leq C(k,d,p)\sum_{v \in \mathcal{V}}\sum_{\widetilde{k}=0}^{k}\left\|f-f_{NN,i+v}^{(poly)}\right\|_{W^{k-\widetilde{k}, p}\left(I_i^N\right)}\left\|\Phi_{i+v}^{N}\right\|_{W^{\widetilde{k}, \infty}\left(I_i^N\right)}
\\
&\quad+C(k,d,p)\sum_{\substack{j \in\{1, \ldots, N\}^d \\
j-i \notin \mathcal{V}}}\left\|f-f_{NN,j}^{(poly)}\right\|_{W^{k, p}\left(I_i^N\right)}\left\|\Phi_j^{N}\right\|_{W^{k, \infty}\left(I_i^N\right)}.
\end{align*}
From (94) in \cite{de2021approximation} we have $\left\|\Phi_{i+v}^{N}\right\|_{W^{\widetilde{k}, \infty}\left(I_i^N\right)}\leq N^{\widetilde{k}}(2\widetilde{k})^{2\widetilde{k}}\ln^{\widetilde{k}} \left(\frac{(2 \widetilde{k})^{\widetilde{k}+1}(N \widetilde{k})^{\widetilde{k}}}{e^{\widetilde{k}} \epsilon_1}\right)$. By utilizing Lemma \ref{local approximation}, we obtain an upper bound of 
\begin{align*}
C(s,d,p)\|f\|_{W^{s,p}(J_i^N)}\ln^k \left(\frac{(2 k)^{k+1}(N k)^k}{e^k \epsilon_1}\right)\left(\left(\frac{1}{N}\right)^{s-k}+\epsilon_2N^{k}\right)
\end{align*}
for the first term. By employing Lemma \ref{POU} and Lemma \ref{local approximation}, we obtain an upper bound for the second term as follows:
\begin{align*}
C(s,d,p)(N^d\|f\|_{W^{k,p}(I_i^N)}+N^{d(1-1/p)}\|f\|_{W^{s-1,p}((0,1)^d))})N^{k} \ln^k \left(\frac{(2 k)^{k+1}(N k)^k}{e^k \epsilon_1}\right)\epsilon_1.
\end{align*}
Combining these two bounds we get
\begin{align*}
&\sum_{j \in\{1, \ldots, N\}^d}\left\|\left(f-f_{NN,j}^{(poly)}\right)\Phi_j^{N} \right\|_{W^{k, p}\left(I_i^N\right)}\\
&\leq C(s,d,p)\ln^k \left(\frac{(2 k)^{k+1}(N k)^k}{e^k \epsilon_1}\right)\\
&\quad\left(\|f\|_{W^{s,p}(J_i^N)}\left(N^{k+d}\epsilon_1+\left(\frac{1}{N}\right)^{s-k}+\epsilon_2N^{k}\right)+\|f\|_{W^{s-1,p}((0,1)^d))}N^{k+d(1-1/p)} \epsilon\right).
\end{align*}
By utilizing Lemma \ref{xp convex}, we obtain the global estimation
\begin{align*}
&\sum_{j \in\{1, \ldots, N\}^d}\left\|\left(f-f_{NN,j}^{(poly)}\right)\Phi_j^{N} \right\|_{W^{k, p}((0,1)^d)}\\
&\leq C(s,d,p)\|f\|_{W^{s,p}((0,1)^d)}\ln^k \left(\frac{(2 k)^{k+1}(N k)^k}{e^k \epsilon_1}\right)\left(N^{k+d}\epsilon_1+\left(\frac{1}{N}\right)^{s-k}+N^{k}\epsilon_2\right).
\end{align*}
Next, we bound the third term of (\ref{app err1}). Based on a similar approach as obtaining (\ref{local approximation2}), we have
$\left\|f_{i}^{(poly)}-f_{NN,i}^{(poly)}\right\|_{W^{k,\infty}((0,1)^d)}\leq C(s,d,p)\|f\|_{W^{s-1,p}(J_i^N)}N^{d/p}\epsilon_2$. Furthermore, due to $\|x^\beta\|_{W^{k,\infty}((0,1)^d)}\leq\widetilde{s}^k$ for $\beta\in P_{\widetilde{s},d}:=\left\{\alpha \in \mathbb{N}_{\geq0}^d:|\alpha|=\widetilde{s}\right\}$, 
\begin{align*}
\left\|f_{i}^{(poly)}\right\|_{W^{k,\infty}((0,1)^d)}
&\leq\sum_{\widetilde{s}=1}^{{s}-1}\sum_{\beta\in P_{\widetilde{s},d}}|c_{\beta,i}|\left\|x^{\beta}\right\|_{W^{k,\infty}((0,1)^d)}\\
&\leq C(s,d,p)\|f\|_{W^{s-1,p}(J_i^N)}N^{d/p}.
\end{align*}
Hence by triangle inequality we have 
\begin{align*}
\left\|f_{NN,i}^{(poly)}\right\|_{W^{k,\infty}((0,1)^d)}\leq C(s,d,p)\|f\|_{W^{s-1,p}(J_i^N)}N^{d/p}. 
\end{align*}
For $i\in[d],\|\phi_{j_i}^N\|_{W^{k,\infty}((0,1)^d)}\leq (2k)^{k+1}\alpha^k$ due to Lemma \ref{estimate for higher derivative of activation}. Employing Lemma \ref{multiplication approximation} with $M=C(s,d,p)\|f\|_{W^{s-1,p}((0,1)^d))}N^{d/p}$ and Lemma \ref{estimate for sobolev norm of composition}, we have 
\begin{align*}
&\|f_{NN}^{(mul)}(f_{NN,j}^{(poly)}(x),\phi_{j_1}^N(x_1),\cdots,\phi_{j_d}^N(x_d))-f_{NN,j}^{(poly)}(x)\Phi_j^{N}(x)\|_{W^{k,\infty}((0,1)^d)}\\
&\leq16(e^2k^4d^2(d+1))^k\left\|f_{NN}^{(mul)}(y)-\prod_{i=1}^{d+1} y_{i}\right\|_{W^{k, \infty}([-M,M]^{d+1})}\\
&\quad\ \max\{\|f_{NN,j}^{(poly)}\|_{W^{k,\infty}((0,1)^d)}^k,\|\phi_{j_1}^N\|_{W^{k,\infty}((0,1)^d)}^k,\cdots,\|\phi_{j_d}^N\|_{W^{k,\infty}((0,1)^d)}^k\}\\
&\leq C(s,d,p)\epsilon_3\left(\|f\|_{W^{s-1,p}(J_j^N)}^kN^{kd/p}+N^k\ln^k \left(\frac{(2 k)^{k+1}(N k)^k}{e^k \epsilon_1}\right)\right).
\end{align*}
Then
\begin{align*}
&\sum_{j\in\{1,\cdots,N\}^d}\|f_{NN}^{(mul)}(f_{NN,j}^{(poly)}(x),\phi_{j_1}^N(x_1),\cdots,\phi_{j_d}^N(x_d))-f_{NN,j}^{(poly)}(x)\Phi_j^{N}(x)\|_{W^{k,p}((0,1)^d)}\\
&\leq C(s,d,p)\epsilon_3\left(\|f\|_{W^{s-1,p}((0,1)^d)}^kN^{d+kd/p}+N^{k+d}\ln^k \left(\frac{(2 k)^{k+1}(N k)^k}{e^k \epsilon_1}\right)\right).
\end{align*}
Combining the estimates for the three terms of (\ref{app err1}), we conclude that
\begin{align*}
&\|f_{NN}-f\|_{W^{k,p}((0,1)^d)}\\
&\leq
C(s,d,p)\|f\|_{W^{s,p}((0,1)^d)}\ln^k \left(\frac{(2 k)^{k+1}(N k)^k}{e^k \epsilon_1}\right)\left(
N^{k+2d}\epsilon_1+\left(\frac{1}{N}\right)^{s-k}+N^{k}\epsilon_2\right)\\
&\quad+C(s,d,p)\left(\|f\|_{W^{s-1,p}((0,1)^d)}^kN^{d+kd/p}+N^{k+d}\ln^k \left(\frac{(2 k)^{k+1}(N k)^k}{e^k \epsilon_1}\right)\right)\epsilon_3.
\end{align*}
We finish the proof by further letting
\begin{align*}
\epsilon_1=\left(\frac{1}{N}\right)^{s+2d},\epsilon_2=\left(\frac{1}{N}\right)^{s},\epsilon_3=\left(\frac{1}{N}\right)^{s+d+k\max\{d/p-1,0\}}.
\end{align*}
\end{proof}
Below, we present the approximation error estimation that we need in this paper. This type of result not only demonstrates the existence of optimal approximation elements that allow neural network functions to sufficiently approximate the target function but also shows that neural network functions in the vicinity of the optimal approximation elements can sufficiently approximate the target function, too. We also need the result that the 2-norm of the outer parameters of the best approximation elements is sufficiently small. Following the ideas proposed in \cite{kohler2022analysis, drews2023analysis}, we achieve this by increasing the network width, which leads to a proportional decrease in the 2-norm of the outer parameters. We will see that such results contribute to the analysis of global convergence in the subsequent discussion of projected gradient descent.
\begin{theorem}\label{app err}
Let $d,A'\in \mathbb{N}_{\geq1}$. Let $N\in\mathbb{N}_{\geq1}$ and $N\geq C(d)$. Suppose that $A,A'$ can be divided by $N^d$. For any $f \in H^{2}((0,1)^d)$ and any injection $\varphi:\{1,2,\cdots,A'\}\to\{1,2,\cdots,A\}$ such that for $k\in[A']$,
\begin{align*}
\varphi(k)\in\{(k-1)A/A'+i':i'\in[A/A']\},
\end{align*}
there exists $W^*=\{((W^*)_s^l,(b^*)_s^l):l=1,2,3;s\in[A]\}$ such that for $W=\{(W_s^l,b_s^l):l=1,2,3;s\in[A]\}$ satisfying for $s\in[A]$, $W_s^3=(W^*)_s^3,b_s^3=(b^*)_s^3$; for $s\in\{\varphi(1),\varphi(2),\cdots,\varphi(A')\},l=1,2$,
\begin{align}\label{inner layer weights}
\|W_{s}^{l}-(W^*)_{s}^{l}\|_F,\|b_{s}^{l}-(b^*)_{s}^{l}\|_2\leq \left(\frac{1}{N}\right)^{\frac{3}{2}d^3+10d^2+\frac{31}{2}d+\frac{29}{2}},
\end{align}
there holds
\begin{align*}
\|f_W-f\|_{H^{1}(\Omega)}\leq C(d,f,\sigma,\Omega)\frac{1}{N^{1/2}}.
\end{align*}
Moreover, 
\begin{align*}
\|(W^*)^3\|_1&\leq C(d,f)N^{\frac{3}{8}d^3+\frac{5}{2}d^2+\frac{37}{8}d+\frac{7}{2}},\\
\|(W^*)^3\|&\leq C(d,f)\frac{1}{A'}N^{\frac{3}{8}d^3+\frac{5}{2}d^2+\frac{41}{8}d+\frac{7}{2}}.
\end{align*}
\end{theorem}
\begin{proof}
Let $s=2,k=1,p=2$ in Proposition \ref{app err pre} and let $p=\frac{1}{2(2d+3)}$ in Lemma \ref{power and logarithm}, we know that provided $N\geq C(d)$, there exists $W^{**}=\{((W^{**})_s^l,(b^{**})_s^l):l=1,2,3;s\in[N^d]\}$
with $(W^{**})_s^1\in\mathbb{R}^{m_1\times d},(W^{**})_s^2\in\mathbb{R}^{m_2\times m_1},(W^{**})_s^3\in\mathbb{R}^{1\times m_2},
(b^{**})_s^1\in\mathbb{R}^{m_1},(b^{**})_s^2\in\mathbb{R}^{m_2},(b^{**})_s^3\in\mathbb{R}$ for $s\in[N^d]$, such that
\begin{align*}
&\|f_{W^{**}}-f\|_{H^1((0,1)^d)}\leq C(d,f)
\frac{1}{N^{1/2}}.
\end{align*}
Furthermore, each component in $(W^{**})_s^l$ and $(b^{**})_s^l$
is bounded by $C(d,f)N^{\frac{3}{8}d^3+\frac{5}{2}d^2+\frac{29}{8}d+\frac{7}{2}}$. 
Define $W^*=\{((W^*)_s^l,(b^*)_s^l):l=1,2,3;s\in[A]\}$ in the following way: for $s\in[N^d]$,
\begin{align*}
&(W^*)_{\varphi((s-1)A'/N^d+i)}^3=\frac{N^d}{A'}(W^{**})_s^3, \quad(b^*)_{\varphi((s-1)A'/N^d+i)}^3=\frac{N^d}{A'}(b^{**})_s^3,\quad i\in[{A'}/{N^d}];\\
&(W^*)_j^3=(b^*)_j^3=0,\quad j\in \{1,2,\cdots,A\}\setminus\{\varphi(1),\varphi(2),\cdots,\varphi(A')\};\\
&(W^*)_{(s-1)A/N^d+i}^l=(W^{**})_s^l,\quad(b^*)_{(s-1)A/N^d+i}^l=(b^{**})_s^l,\quad l=1,2;i\in[{A}/{N^d}].
\end{align*}
It is easy to see $f_{W^*}\equiv f_{W^{**}}$ and
\begin{align*}
\|(W^*)^3\|_1&=\|(W^{**})^3\|_1\leq C(d,f)N^{\frac{3}{8}d^3+\frac{5}{2}d^2+\frac{37}{8}d+\frac{7}{2}},\\
\|(W^*)^3\|&=\frac{N^d}{A'}\|(W^{**})^3\|\leq C(d,f)\frac{1}{A'}N^{\frac{3}{8}d^3+\frac{5}{2}d^2+\frac{41}{8}d+\frac{7}{2}}.
\end{align*}
Next we estimate $\|f_W-f_{W^*}\|_{H^{1}(\Omega)}$ for $W$ satisfying condition (\ref{inner layer weights}). We begin with the following two inequalities:
\begin{align*}
&\|f_{s,1}^{org}-(f^*)_{s,1}^{org}\|_{\infty}\leq\|W_s^1-(W^*)_s^1\|_{\infty}+\|b_s^1-(b^*)_s^1\|_{\infty},\\
&\|f_{s,2}^{org}-(f^*)_{s,2}^{org}\|_{\infty}\\
&\leq\|W_s^2-(W^*)_s^2\|_{\infty}\|f_s^1\|_{\infty}+\|(W^*)_s^2\|_{\infty}\|(f^*)_s^1-(f^*)_s^1\|_{\infty}+\|b_s^2-(b^*)_s^2\|_{\infty}\\
&\leq B_{\sigma}\|W_s^2-(W^*)_s^2\|_{\infty}+\|b_s^2-(b^*)_s^2\|_{\infty}+L_{\sigma}\|(W^*)_s^2\|_{\infty}\|W_s^1-(W^*)_s^1\|_{\infty}\\
&\quad+L_{\sigma}\|(W^*)_s^2\|_{\infty}\|b_s^1-(b^*)_s^1\|_{\infty}.
\end{align*}
With these two inequalities, we have
\begin{align*}
&|f_{W}-f_{W^*}|\leq\sum_{s=1}^A|f_s^3-(f^*)_s^3|
\leq\sum_{s=1}^A\|(W^*)_s^3\|_{\infty}\|f_s^2-(f^*)_s^2\|_{\infty}\\
&\leq \sum_{s=1}^A\|(W^*)_s^3\|_{\infty}[
L_{\sigma}B_{\sigma}\|W_s^2-(W^*)_s^2\|_{\infty}+L_{\sigma}\|b_s^2-(b^*)_s^2\|_{\infty}\\
&\quad\ +L_{\sigma}^2\|(W^*)_s^2\|_{\infty}\|W_s^1-(W^*)_s^1\|_{\infty}+L_{\sigma}^2\|(W^*)_s^2\|_{\infty}\|b_s^1-(b^*)_s^1\|_{\infty}].
\end{align*}
Using the formula for calculating the gradient of $f_W$:
\begin{align*}
\nabla_xf_W=\sum_{s=1}^A(W_s^1)^T\mathrm{diag}[\sigma'(f_{s,1}^{org})](W_s^2)^T\mathrm{diag}[\sigma'(f_{s,2}^{org})](W_s^3)^T.
\end{align*}
and the above two inequalities, we derive that
\begin{align*}
&\|\nabla_xf_{W}-\nabla_xf_{W^*}\|_{\infty}\\
&\leq\sum_{s=1}^A[\|(W_s^1)^T-((W^*)_s^1)^T\|_{\infty}\|\mathrm{diag}[\sigma'(f_{s,1}^{org})](W_s^2)^T\mathrm{diag}[\sigma'(f_{s,2}^{org})]((W^*)_s^3)^T\|_{\infty}\\
&\quad\ +\|((W^*)_s^1)^T\|_{\infty}\|\mathrm{diag}[\sigma'(f_{s,1}^{org})]-\mathrm{diag}[\sigma'((f^*)_{s,1}^{org})]\|_{\infty}\\
&\qquad\ \ \|(W_s^2)^T\mathrm{diag}[\sigma'(f_{s,2}^{org})]((W^*)_s^3)^T\|_{\infty}\\
&\quad\ +\|((W^*)_s^1)^T\mathrm{diag}[\sigma'((f^*)_{s,1}^{org})]\|_{\infty}\|(W_s^2)^T-((W^*)_s^2)^T\|_{\infty}\\
&\qquad\ \ \|\mathrm{diag}[\sigma'(f_{s,2}^{org})]((W^*)_s^3)^T\|_{\infty}\\
&\quad\ +\|((W^*)_s^1)^T\mathrm{diag}[\sigma'((f^*)_{s,1}^{org})]((W^*)_s^2)^T\|_{\infty}\\
&\qquad\ \ \|\mathrm{diag}[\sigma'(f_{s,2}^{org})]-\mathrm{diag}[\sigma'((f^*)_{s,2}^{org})]\|_{\infty}\|((W^*)_s^3)^T\|_{\infty}\\
&\leq \sum_{s=1}^A\|((W^*)_s^3)^T\|_{\infty}[B_{\sigma'}^2\|(W_s^2)^T\|_{\infty}\|(W_s^1)^T-((W^*)_s^1)^T\|_{\infty}\\
&\quad\ +L_{\sigma'}B_{\sigma'}\|(W_s^2)^T\|_{\infty}\|((W^*)_s^1)^T\|_{\infty}\|f_{s,1}^{org}-(f^*)_{s,1}^{org}\|_{\infty}\\
&\quad\ +B_{\sigma'}^2\|((W^*)_s^1)^T\|_{\infty}\|(W_s^2)^T-((W^*)_s^2)^T\|_{\infty}\\
&\quad\ +L_{\sigma'}B_{\sigma'}\|((W^*)_s^2)^T\|_{\infty}\|((W^*)_s^1)^T\|_{\infty}\|f_{s,2}^{org}-(f^*)_{s,2}^{org}\|_{\infty}]\\
&\leq\sum_{s=1}^A\|((W^*)_s^3)^T\|_{\infty}[B_{\sigma'}^2\|((W^*)_s^1)^T\|_{\infty}\|(W_s^2)^T-((W^*)_s^2)^T\|_{\infty}\\
&\quad\ +L_{\sigma'}B_{\sigma}B_{\sigma'}\|((W^*)_s^2)^T\|_{\infty}\|((W^*)_s^1)^T\|_{\infty}\|W_s^2-(W^*)_s^2\|_{\infty}\\
&\quad\ +L_{\sigma'}B_{\sigma'}\|((W^*)_s^2)^T\|_{\infty}\|((W^*)_s^1)^T\|_{\infty}\|b_s^2-(b^*)_s^2\|_{\infty}\\
&\quad\ +B_{\sigma'}^2\|(W_s^2)^T\|_{\infty}\|(W_s^1)^T-((W^*)_s^1)^T\|_{\infty}\\
&\quad\ +L_{\sigma'}B_{\sigma'}\|((W^*)_s^1)^T\|_{\infty}\|x\|_{\infty}\\
&\qquad\ (L_{\sigma}\|((W^*)_s^2)^T\|_{\infty}\|(W^*)_s^2\|_{\infty}+\|(W_s^2)^T\|_{\infty})\|W_s^1-(W^*)_s^1\|_{\infty}\\
&\quad\ +L_{\sigma'}B_{\sigma'}\|((W^*)_s^1)^T\|_{\infty}\\
&\qquad\ (L_{\sigma}\|((W^*)_s^2)^T\|_{\infty}\|(W^*)_s^2\|_{\infty}+\|(W_s^2)^T\|_{\infty})\|b_s^1-(b^*)_s^1\|_{\infty}].
\end{align*}
Plugging the upper bound for the norm of weights and biases:
\begin{align*}
&\|((W^*)_s^2)^T\|_{\infty}\leq m_2B_2,\|(W^*)_s^2\|_{\infty}\leq m_1B_2,\\
&\|((W^*)_s^1)^T\|_{\infty}\leq m_1B_1,\|(W^*)_s^1\|_{\infty}\leq dB_1,\\
&\|(W^*)^3\|_1\leq C(d,f)N^{\frac{3}{8}d^3+\frac{5}{2}d^2+\frac{37}{8}d+\frac{7}{2}},\\
&\|(W_s^2)^T\|_{\infty}\leq\|((W^*)_s^2)^T\|_{\infty}+\|(W_s^2)^T-((W^*)_s^2)^T\|_{\infty}\leq 2m_2B_2.
\end{align*}
with
\begin{align*}
m_1&=3\left|P_{{1}, d}\right|+2d,m_2=3\left\lceil\frac{d+2}{2}\right\rceil\left|P_{d+1, d+1}\right|,\\
B_{1,2}&=C(d,f)N^{\frac{3}{8}d^3+\frac{5}{2}d^2+\frac{29}{8}d+\frac{7}{2}}
\end{align*}
and simplifing the expressions, we finally obtain
\begin{align*}
&\|f_W-f_{W^*}\|_{W^{1,\infty}(\Omega)}=\max\left\{\max_{x\in\Omega}|f_W-f_{W^*}|,\max_{x\in\Omega}\|\nabla_xf_W-\nabla_xf_{W^*}\|_{\infty}\right\}\\
&\leq C(d,f,\sigma)N^{\frac{3}{8}d^3+\frac{5}{2}d^2+\frac{37}{8}d+\frac{7}{2}}\max_{s=1,2,\cdots,A}\\
&\quad\ [N^{\frac{3}{8}d^3+\frac{5}{2}d^2+\frac{29}{8}d+\frac{7}{2}}\|(W_s^2)^T-((W^*)_s^2)^T\|_{\infty}
+N^{\frac{3}{4}d^3+5d^2+\frac{29}{4}d+7}\|W_s^2-(W^*)_s^2\|_{\infty}\\
&\quad\ +N^{\frac{3}{4}d^3+5d^2+\frac{29}{4}d+7}\|b_s^2-(b^*)_s^2\|_{\infty}
+N^{\frac{3}{8}d^3+\frac{5}{2}d^2+\frac{29}{8}d+\frac{7}{2}}\|(W_s^1)^T-((W^*)_s^1)^T\|_{\infty}\\
&\quad\ +N^{\frac{9}{8}d^3+\frac{15}{2}d^2+\frac{87}{8}d+\frac{21}{2}}\|W_s^1-(W^*)_s^1\|_{\infty}
+N^{\frac{9}{8}d^3+\frac{15}{2}d^2+\frac{87}{8}d+\frac{21}{2}}\|b_s^1-(b^*)_s^1\|_{\infty}]\\
&\leq C(d,f,\sigma)\frac{1}{N^{1/2}},
\end{align*}
which implies the result due to triangle inequality and the fact that $H^1$ norm can be controlled by $W^{1,\infty}$ norm.
\end{proof}

\section{Generalization Error}

In this section, we utilize tools such as Rademacher complexity and covering numbers to handle generalization error
\begin{align*}
\max\left\{\sup_{f_W\in\mathcal{F}_{NN}}[{L}_{R}(f_{W})-\widehat{L}_{R}(f_{W})],\sup_{f_W\in\mathcal{F}_{NN}}[\widehat{L}_{R}(f_{W})-{L}_{R}(f_{W})]\right\}
\end{align*}
and
\begin{align*}
\max\left\{\sup_{f_W\in\mathcal{F}_{NN}}[{L}_{N}(f_{W})-\widehat{L}_{N}(f_{W})],\sup_{f_W\in\mathcal{F}_{NN}}[\widehat{L}_{N}(f_{W})-{L}_{N}(f_{W})]\right\},
\end{align*}
where $\mathcal{F}_{NN}$ is an abbreviation for
\begin{align*}
&\mathcal{F}_{NN}(\{m_1,m_2,A\},B_{inn},B_{out})=\\
&\{f_W:\|W_s^l\|_F,\|b_s^l\|_2\leq B_{inn}(l=1,2;s\in[A]),\|W^3\|_1\leq B_{out}\}.
\end{align*}
In this section we assume that the activation $\sigma\in C^2(\Omega)$.

\begin{lemma} \label{triangle inequality}
\begin{align*}
&\mathbb{E}_{\{{X_i}\}_{i=1}^{n},\{{Y_j}\}_{j=1}^{m}}\sup_{f_W\in\mathcal{F}_{NN}}\pm\left[{L}_R(f_W)-\widehat{L}_R(f_W)\right]\\
&\leq\sum_{k=1}^{5}\mathbb{E}_{\{{X_i}\}_{i=1}^{n},\{{Y_j}\}_{j=1}^{m}}\sup_{f_W\in\mathcal{F}_{NN}}\pm\left[{L}_k^{(R)}(f_W)-\widehat{{L}}_k^{(R)}(f_W)\right],\\
&\mathbb{E}_{\{{X_i}\}_{i=1}^{n},\{{Y_j}\}_{j=1}^{m}}\sup_{f_W\in\mathcal{F}_{NN}}\pm\left[{L}_N(f_W)-\widehat{L}_N(f_W)\right]\\
&\leq\sum_{k=1}^{5}\mathbb{E}_{\{{X_i}\}_{i=1}^{n},\{{Y_j}\}_{j=1}^{m}}\sup_{f_W\in\mathcal{F}_{NN}}\pm\left[{L}_k^{(N)}(f_W)-\widehat{{L}}_k^{(N)}(f_W)\right],
\end{align*}
where
\begin{align*}
&{L}_1^{(R)}(f_W)={L}_1^{(N)}(f_W)=\frac{|\Omega|}{2}\mathbb{E}_{X\sim U(\Omega)}\|\nabla_x f_W(X)\|_2^2,\\
&{L}_2^{(R)}(f_W)={L}_2^{(N)}(f_W)=\frac{|\Omega|}{2}\mathbb{E}_{X\sim U(\Omega)}w(X)f_W^2(X),\\
&{L}_3^{(R)}(f_W)={L}_3^{(N)}(f_W)=-|\Omega|\mathbb{E}_{X\sim U(\Omega)}f(X)f_W(X),\\
&{L}_4^{(R)}(f_W)=\frac{|\partial\Omega|}{2\beta}\mathbb{E}_{Y\sim U(\partial\Omega)}(T_0f_W)^2(Y),\quad{L}_4^{(N)}(f_W)=0,\\
&{L}_5^{(R)}(f_W)=-\frac{|\partial\Omega|}{\beta}\mathbb{E}_{Y\sim U(\partial\Omega)}g(Y)T_0f_W(Y),\\
&{L}_5^{(N)}(f_W)=-|\partial\Omega|\mathbb{E}_{Y\sim U(\partial\Omega)}g(Y)T_0f_W(Y),
\end{align*}
	and $\widehat{{L}}_k^{(R)}(f_W),\widehat{{L}}_k^{(N)}(f_W)$ is the discrete version of ${{L}}_k^{(R)}(f_W),{{L}}_k^{(N)}(f_W)$, for example,
	\begin{equation*}
\widehat{{L}}_1^{(R)}(f_W)=\frac{|\Omega|}{2n}\sum_{i=1}^{n}\|\nabla f_W(X_i)\|_2^2.
	\end{equation*}
\end{lemma}
\begin{proof}
{The result can be obtained easily from the property of supremum.}
\end{proof}

By the symmetrization argument, we can convert the estimation of generalization error into an estimation of the Rademacher complexity of the relevant function classes.
\begin{lemma}[\cite{jiao2024error}, Lemma 5.3] \label{symmetrization}
Define
\begin{align*}
	\mathcal{F}_1&:=\{\|\nabla _xf_W\|_2^2:f_W\in\mathcal{F}_{NN}\},
	&\mathcal{F}_2&:=\{f_W^2:f_W\in\mathcal{F}_{NN}\},\\
	\mathcal{F}_3&:=\{f_W:f_W\in\mathcal{F}_{NN}\},
	&\mathcal{F}_4&:=\{f_W^2|_{\partial\Omega}:f_W\in\mathcal{F}_{NN}\},\\
	\mathcal{F}_5&:=\{f_W|_{\partial\Omega}:f_W\in\mathcal{F}_{NN}\}.
\end{align*}
There holds
\begin{align*}
&\mathbb{E}_{\{{X_i}\}_{i=1}^{n}}\sup_{f_W\in\mathcal{F}_{NN}}\pm \left[{L}_1^{(R)}(f_W)-\widehat{{L}}_1^{(R)}(f_W)\right]\leq C(\Omega,coe)\mathfrak{R}_n(\mathcal{F}_1),\\
&\mathbb{E}_{\{{X_i}\}_{i=1}^{n}}\sup_{f_W\in\mathcal{F}_{NN}}\pm \left[{L}_2^{(R)}(f_W)-\widehat{{L}}_2^{(R)}(f_W)\right]\leq C(\Omega,coe)\mathfrak{R}_n(\mathcal{F}_2),\\
&\mathbb{E}_{\{{X_i}\}_{i=1}^{n}}\sup_{f_W\in\mathcal{F}_{NN}}\pm \left[{L}_3^{(R)}(f_W)-\widehat{{L}}_3^{(R)}(f_W)\right]\leq C(\Omega,coe)\mathfrak{R}_n(\mathcal{F}_3),\\
&\mathbb{E}_{\{{X_i}\}_{i=1}^{n}}\sup_{f_W\in\mathcal{F}_{NN}}\pm \left[{L}_4^{(R)}(f_W)-\widehat{{L}}_4^{(R)}(f_W)\right]
\leq
C(\Omega,coe)\max\{1,1/\beta\}\mathfrak{R}_m(\mathcal{F}_4),\\
&\mathbb{E}_{\{{X_i}\}_{i=1}^{n}}\sup_{f_W\in\mathcal{F}_{NN}}\pm \left[{L}_5^{(R)}(f_W)-\widehat{{L}}_5^{(R)}(f_W)\right]
\leq C(\Omega,coe)\max\{1,1/\beta\}\mathfrak{R}_m(\mathcal{F}_5)
\end{align*}
and
\begin{align*}
&\mathbb{E}_{\{{X_i}\}_{i=1}^{n}}\sup_{f_W\in\mathcal{F}_{NN}}\pm \left[{L}_1^{(N)}(f_W)-\widehat{{L}}_1^{(N)}(f_W)\right]
\leq C(\Omega,coe)\mathfrak{R}_n(\mathcal{F}_1),\\
&\mathbb{E}_{\{{X_i}\}_{i=1}^{n}}\sup_{f_W\in\mathcal{F}_{NN}}\pm \left[{L}_2^{(N)}(f_W)-\widehat{{L}}_2^{(N)}(f_W)\right]
\leq C(\Omega,coe)\mathfrak{R}_n(\mathcal{F}_2),\\
&\mathbb{E}_{\{{X_i}\}_{i=1}^{n}}\sup_{f_W\in\mathcal{F}_{NN}}\pm \left[{L}_3^{(N)}(f_W)-\widehat{{L}}_3^{(N)}(f_W)\right]
\leq C(\Omega,coe)\mathfrak{R}_n(\mathcal{F}_3),\\
&\mathbb{E}_{\{{X_i}\}_{i=1}^{n}}\sup_{f_W\in\mathcal{F}_{NN}}\pm\left[{L}_4^{(N)}(f_W)-\widehat{{L}}_4^{(N)}(f_W)\right]
=0,\\
&\mathbb{E}_{\{{X_i}\}_{i=1}^{n}}\sup_{f_W\in\mathcal{F}_{NN}}\pm \left[{L}_5^{(N)}(f_W)-\widehat{{L}}_5^{(N)}(f_W)\right]
\leq C(\Omega,coe)\mathfrak{R}_m(\mathcal{F}_5).
\end{align*}
\end{lemma}

According to Lemma \ref{to covering}, we can estimate the Rademacher complexity of a given function class by derive an upper bound of the covering number of the function class. The following lemma provide upper bounds for the covering number of bounded smooth function classes.
\begin{lemma}[\cite{shiryayev2010selected}, Chapter 7, Theorem XIV]\label{covering number bound}
Let $s\in\mathbb{N}_{\geq1},B\in\mathbb{R}_{>0}$. Let $C_{\mathcal{B}(0,B)}^s(\Omega):=\{f\in C^s(\Omega):\|f\|_{C^s(\Omega)}\leq B\}$. We have for $\epsilon>0$,
\begin{align*}
\ln\mathcal{N}(\epsilon,C_{\mathcal{B}(0,B)}^s(\Omega),\|\cdot\|_{\infty})\leq C(s,d)\left(\frac{B}{\epsilon}\right)^{d/s}.
\end{align*}
\end{lemma}
\begin{remark}
The functions in set $C_{\mathcal{B}(0,B)}^s(\Omega)$ have uniformly bounded derivatives, thus satisfying the conditions (64) and (67) in \cite{shiryayev2010selected}.
\end{remark}
\begin{remark}
Theorem XIV in \cite{shiryayev2010selected} does not explicitly state the dependence of the covering number on the parameter B, the upper bound for derivatives. However, this relationship can be immediately obtained using the technique of scaling.
\end{remark}
\begin{remark}
In fact, Theorem XIV in \cite{shiryayev2010selected} only provides upper bounds for the so-called minimal $\epsilon$-entropy and $\epsilon$-capacity, while what we need is an upper bound for the so-called $\epsilon$-entropy there. However, Theorem IV in \cite{shiryayev2010selected} implies that for a totally bounded set, the $\epsilon$-entropy is always less than or equal to the $\epsilon$-capacity. On the other hand, the Arzela-Ascoli theorem guarantees that $C_{\mathcal{B}(0,B)}^s(\Omega)$ is a totally bounded set.
\end{remark}

The following result shows that $\{\mathcal{F}_i\}_{i=1}^5$, the function classes we concern, are all subsets of bounded smooth function classes. 
\begin{lemma}\label{function classes C1 norm}
For function classes $\{\mathcal{F}_i\}_{i=1}^5$ defined in Lemma \ref{symmetrization}, there holds
\begin{align*}
\mathcal{F}_i\subset C_{\mathcal{B}(0,B_{\mathcal{F}_i})}^1(\Omega),\quad i=1,2,3,4,5
\end{align*}
with
\begin{align*}
B_{\mathcal{F}_1}&=2dm_2^2m_1^{3/2}B_{\sigma'}^4B_{\sigma''}B_{inn}^5B_{out}^2,\\
B_{\mathcal{F}_2},B_{\mathcal{F}_4}&=2m_2^2\sqrt{m_1}B_{\sigma}^2B_{\sigma'}^2B_{inn}B_{out}^2,\\
B_{\mathcal{F}_3},B_{\mathcal{F}_5}&=m_2\sqrt{m_1}B_{\sigma}B_{\sigma'}^2B_{inn}B_{out}.
\end{align*}
\end{lemma}
\begin{proof}
For any $f_W\in\mathcal{F}_{NN}(\{m_1,m_2,A\},B_{inn},B_{out})$,
\begin{align*}
|f_W|\leq\sum_{s=1}^A\|W_s^3\|_F\|f_s^2\|_2\leq{m_2}B_{\sigma}\|W^3\|_1\leq m_2B_{\sigma}B_{out}.
\end{align*}
where we employ the relation
\begin{align*}
&\sum_{s=1}^A\|W_s^3\|_F\leq\sqrt{m_2}\sum_{s=1}^A\max_{k_2\in[m_2]}|w_{s,k_2}^3|\leq\sqrt{m_2}\|W^3\|_1.
\end{align*}
The first order derivative of $f_W$ with respect to the spatial variable $x$ can be calculated by
\begin{align*}
\frac{\partial f_W}{\partial x_j}=\sum_{s=1}^A(W_s^{1,j})^T\mathrm{diag}[\sigma'(f_{s,1}^{org})](W_s^2)^T\mathrm{diag}[\sigma'(f_{s,2}^{org})](W_s^3)^T,\quad j\in[d].
\end{align*}
Hence
\begin{align*}
\left|\frac{\partial f_W}{\partial x_j}\right|&\leq\sum_{s=1}^A\|W_s^{1,j}\|_F\|\mathrm{diag}[\sigma'(f_{s,1}^{org})]\|_F\|W_s^2\|_F\|\mathrm{diag}[\sigma'(f_{s,2}^{org})]\|_F\|W_s^3\|_F\\
&\leq m_2\sqrt{m_1}B_{\sigma'}^2B_{inn}B_{out}.
\end{align*}
By now, we have obtained upper bound estimates for the $C^1$ norms of $\mathcal{F}_3$ and $\mathcal{F}_5$. For any $h\in\mathcal{F}_2$, there exists $f_W\in\mathcal{F}_{NN}$ such that $h=f_W^2$. Then $\frac{\partial h}{\partial x_{j}}=2f_W\frac{\partial f_W}{\partial x_j}$ and
\begin{align*}
|h|&=|f_W|^2\leq m_2B_{\sigma}B_{out},\\
\left|\frac{\partial h}{\partial x_{j}}\right|&=2|f_W|\left|\frac{\partial f_W}{\partial x_j}\right|
\leq2m_2^2\sqrt{m_1}B_{\sigma}B_{\sigma'}^2B_{inn}B_{out}^2,
\end{align*}
which implies $B_{\mathcal{F}_2}$. We can derive $B_{\mathcal{F}_4}$ in the same way. To estimate $C^1$ norm of $\mathcal{F}_1$, we need to calculate the second order derivative of $f_W$ with respect to the spatial variable $x$. For $j,j'\in[d]$,
\begin{align*}
&\frac{\partial^2 f_W}{\partial x_j\partial x_{j'}}=\\
&\sum_{s=1}^A(W_s^{1,j'})^T [(\mathrm{diag}[\sigma''(f_{s,1}^{org})](W_s^2)^T\mathrm{diag}[\sigma'(f_{s,2}^{org})](W_s^3)^T)\odot W_s^{1,j}\\
&+\mathrm{diag}[\sigma'(f_{s,1}^{org})](W_s^2)^T ((\mathrm{diag}[\sigma''(f_{s,2}^{org})]W_s^2\mathrm{diag}[\sigma'(f_{s,1}^{org})]W_{s}^{1,j})\odot (W_s^3)^T)].
\end{align*}
Then
\begin{align*}
&\left|\frac{\partial^2 f_W}{\partial x_j\partial x_{j'}}\right|\\
&\leq\sum_{s=1}^A\|W_s^{1,j'}\|_F [\|\mathrm{diag}[\sigma''(f_{s,1}^{org})]\|_F\|W_s^2\|_F\|\mathrm{diag}[\sigma'(f_{s,2}^{org})]\|_F\|W_s^3\|_F\|W_s^{1,j}\|_F\\
&\qquad\qquad\qquad\qquad+\|\mathrm{diag}[\sigma'(f_{s,1}^{org})]\|_F\|W_s^2\|_F\|\mathrm{diag}[\sigma''(f_{s,2}^{org})]\|_F\\
&\|W_s^2\|_F\|\mathrm{diag}[\sigma'(f_{s,1}^{org})]\|_F\|W_{s}^{1,j}\|_F\|W_s^3\|_F]\\
&\leq m_2m_1B_{\sigma'}^2B_{\sigma''}B_{inn}^4B_{out}.
\end{align*}
For any $h\in\mathcal{F}_1$, there exists $f_W\in\mathcal{F}_{NN}$ such that $h=\|\nabla_xf_W\|_2^2$. It follows that
\begin{align*}
|h|\leq\sum_{j=1}^d\left|\frac{\partial f_W}{\partial x_j}\right|^2\leq dm_2^2{m_1}B_{\sigma'}^4B_{inn}^2B_{out}^2.
\end{align*}
Since $\frac{\partial h}{\partial x_{j}}=\frac{\partial }{\partial x_{j}}\sum_{j'=1}^d\left(\frac{\partial f_W}{\partial x_{j'}}\right)^2
=2\sum_{j'=1}^d\frac{\partial f_W}{\partial x_{j'}}\frac{\partial^2 f_W}{\partial x_j\partial x_{j'}}$, we have
\begin{align*}
\left|\frac{\partial h}{\partial x_{j}}\right|\leq2\sum_{j'=1}^d\left|\frac{\partial f_W}{\partial x_{j'}}\right|\left|\frac{\partial^2 f_W}{\partial x_j\partial x_{j'}}\right|
\leq2dm_2^2m_1^{3/2}B_{\sigma'}^4B_{\sigma''}B_{inn}^5B_{out}^2.
\end{align*}
Combining the above two estimates we obtain $B_{\mathcal{F}_1}$. 
\end{proof}

\begin{proposition}\label{sta error exp}
Let $m=n$. There holds
\begin{align*}
\mathbb{E}_{\{{X_i}\}_{i=1}^{n},\{{Y_j}\}_{j=1}^{m}}\sup_{f_W\in\mathcal{F}_{NN}}\pm\left[{L}_R(f_W)-\widehat{{L}}_R(f_W)\right]&\leq C(d,\sigma,\Omega,coe)\max\left\{1,\frac{1}{\beta}\right\}\frac{B_{inn}^5B_{out}^2}{\sqrt{n}},\\
\mathbb{E}_{\{{X_i}\}_{i=1}^{n},\{{Y_j}\}_{j=1}^{m}}\sup_{f_W\in\mathcal{F}_{NN}}\pm\left[{L}_N(f_W)-\widehat{{L}}_N(f_W)\right]&\leq C(d,\sigma,\Omega,coe)\frac{B_{inn}^5B_{out}^2}{\sqrt{n}}.
\end{align*}
\end{proposition}
\begin{proof}
We apply Lemma \ref{to covering} to calculate an upper bound of $\mathfrak{R}_n(\mathcal{F}_i)$ for $i=1,2,3,4,5$:
\begin{align*}
\mathfrak{R}_n(\mathcal{F}_i)&\leq\inf_{0<\delta<B_{\mathcal{F}_i}/2}\left(4\delta+\frac{12}{\sqrt{n}}\int_{\delta}^{B_{\mathcal{F}_i}/2}\sqrt{\ln\mathcal{N}(\epsilon,\mathcal{F}_i,\|\cdot\|_{\infty})}d\epsilon\right)\\
&\leq\inf_{0<\delta<B_{\mathcal{F}_i}/2}\left(4\delta+C(d)\frac{12}{\sqrt{n}}\int_{\delta}^{B_{\mathcal{F}_i}/2}\left(\frac{B_{\mathcal{F}_i}}{\epsilon}\right)^{d/2}d\epsilon\right)\\
&\leq\inf_{0<\delta<B_{\mathcal{F}_i}/2}\left(4\delta+C(d)\frac{B_{\mathcal{F}_i}}{\sqrt{n}}\right)=C(d)\frac{B_{\mathcal{F}_i}}{\sqrt{n}},
\end{align*}
where in the second step we employ Lemma \ref{covering number bound} with $s=1$. The result is then implied by Lemma \ref{triangle inequality}, \ref{symmetrization} and \ref{function classes C1 norm}.
\end{proof}

\begin{theorem}\label{sta err}
Let $m=n$. Let $\tau\in\mathbb{R}_{>0}$. With probability at least $1-2e^{\frac{-\min\{1,\beta^2\}n\tau^2}{C(d,\sigma,coe,\Omega)B_{inn}^4B_{out}^4}}$,
\begin{small}
\begin{align*}
&\max\left\{\sup_{f_W\in\mathcal{F}_{NN}}[\widehat{L}_{R}({f_W})-{L}_{R}({f_W})],\sup_{f_W\in\mathcal{F}_{NN}}[{L}_{R}({u})-\widehat{L}_{R}({f_W})]\right\}\\
&\leq C(d,\sigma,\Omega,coe)\max\{1,1/\beta\}\frac{B_{inn}^5B_{out}^2}{\sqrt{n}}+\tau.
\end{align*}
\end{small}
With probability at least $1-2e^{\frac{-n\tau^2}{C(d,\sigma,coe,\Omega)B_{inn}^4B_{out}^4}}$,
\begin{small}
\begin{align*}
&\max\left\{\sup_{f_W\in\mathcal{F}_{NN}}[\widehat{L}_{N}({f_W})-{L}_{N}({f_W})],\sup_{f_W\in\mathcal{F}_{NN}}[{L}_{N}({f_W})-\widehat{L}_{N}({f_W})]\right\}\\
&\leq C(d,\sigma,\Omega,coe)\frac{B_{inn}^5B_{out}^2}{\sqrt{n}}+\tau.
\end{align*}
\end{small}
\end{theorem}
\begin{proof}
We only prove the inequality of the Robin problem since the inequality of the Neumann problem can be proved similarly. Define
\begin{align*}
h(X_1,\cdots,X_n,Y_1,\cdots,Y_m):=\sup_{f_W\in\mathcal{F}_{NN}}\left[\widehat{L}_R(f_W)-L_R(f_W)\right].
\end{align*}
We examine the difference of $h$:
\begin{align*}
&h(X_1,\cdots,\widetilde{X}_i,\cdots,X_n,Y_1,\cdots,Y_m)-h(X_1,\cdots,{X}_i,\cdots,X_n,Y_1,\cdots,Y_m)\\
&\leq\frac{|\Omega|}{n}\sup_{f_W\in\mathcal{F}_{NN}}\left[\|\nabla_x f_W(\widetilde{X}_i)\|_2^2
+\frac{1}{2}w(\widetilde{X}_i)f_W^2(\widetilde{X}_i)-f(\widetilde{X}_i)f_W(\widetilde{X}_i)-\|\nabla_x f_W({X}_i)\|_2^2
\right.\\
&\left.\qquad\qquad\qquad\quad-\frac{1}{2}w(X_i)f_W^2(X_i)+f(X_i)f_W(X_i)\right]\\
&\leq C(d,\sigma,coe,\Omega)\frac{B_{inn}^2B_{out}^2}{n}.
\end{align*}
where in the second step we apply Lemma \ref{function classes C1 norm}. We can bound 
\begin{align*}
h(X_1,\cdots,{X}_i,\cdots,X_n,Y_1,\cdots,Y_m)-h(X_1,\cdots,\widetilde{X}_i,\cdots,X_n,Y_1,\cdots,Y_m) 
\end{align*}
similarly and hence obtain
\begin{align*}
&\left|h(X_1,\cdots,\widetilde{X}_i,\cdots,X_n,Y_1,\cdots,Y_m)-h(X_1,\cdots,{X}_i,\cdots,X_n,Y_1,\cdots,Y_m)\right|\\
&\leq C(d,\sigma,coe,\Omega)\frac{B_{inn}^2B_{out}^2}{n}.
\end{align*}
In the same way we can show
\begin{align*}
&\left|h(X_1,\cdots,X_n,Y_1,\cdots,\widetilde{Y}_j\cdots,Y_m)-h(X_1,\cdots,X_n,Y_1,\cdots,{Y}_j\cdots,Y_m)\right|\\
&\leq C(d,\sigma,coe,\Omega)\max\{1,1/\beta\}\frac{B_{inn}^2B_{out}^2}{n}.
\end{align*}
Therefore according to Lemma \ref{McDiarmid’s inequality}, we have $h-\mathbb{E}h\leq\tau$ with probability at least $1-e^{\frac{-\min\{1,\beta^2\}n\tau^2}{C(d,\sigma,coe,\Omega)B_{inn}^4B_{out}^4}}$. We can also show that $\sup_{f_W\in\mathcal{F}_{NN}}\left[{L}_R(f_W)-\widehat{L}_R(f_W)\right]-\mathbb{E}\sup_{f_W\in\mathcal{F}_{NN}}\left[{L}_R(f_W)-\widehat{L}_R(f_W)\right]\leq\tau$ with probability at least $1-e^{\frac{-\min\{1,\beta^2\}n\tau^2}{C(d,\sigma,coe,\Omega)B_{inn}^4B_{out}^4}}$ in a similar manner. Then the estimate of 
\begin{align*}
\max\left\{\sup_{f_W\in\mathcal{F}_{NN}}\left[\widehat{L}_R(f_W)-L_R(f_W)\right],\sup_{f_W\in\mathcal{F}_{NN}}\left[{L}_R(f_W)-\widehat{L}_R(f_W)\right]\right\}
\end{align*} 
is implied by the union bound.
\end{proof}

\section{Optimization Error}

In this section we study the convergence of PGD:
\begin{align*}
\text{Robin problem: }W_{t+1}&=\mathrm{proj}_{\mathcal{C}}(W_t-\eta\nabla_{W}\widehat{L}_R(W_t)),\quad t=0,1,\cdots,T-1;\\
\text{Neumann problem: }W_{t+1}&=\mathrm{proj}_{\mathcal{C}}(W_t-\eta\nabla_{W}\widehat{L}_N(W_t)),\quad t=0,1,\cdots,T-1.
\end{align*}
where $\mathcal{C}$ is some projected set, which will be determined later. Our goal is to derive an upper bound of $\widehat{L}_R(f_{W_T})-\inf_{f_W\in\mathcal{F}_{NN}}\widehat{L}_R(f_W)$ and $\widehat{L}_N(f_{W_T})-\inf_{f_W\in\mathcal{F}_{NN}}\widehat{L}_N(f_W)$, where $\mathcal{F}_{NN}$ is an abbreviation for
\begin{align*}
&\mathcal{F}_{NN}(\{m_1,m_2,A\},B_{inn},B_{out})=\\
&\{f_W:\|W_s^l\|_F,\|b_s^l\|_2\leq B_{inn}(l=1,2;s\in[A]),\|W^3\|_1\leq B_{out}\}.
\end{align*}

In this section we assume that the activation $\sigma\in C^2(\Omega)$.

\begin{lemma}\label{estimate of gradient}
Define
\begin{align*}
\|\nabla_{W^3} \widehat{L}_R(f_W)\|:=\left[\sum_{s=1}^A(\|\nabla_{W_s^3}\widehat{L}_R(f_W)\|_F^2+\|\nabla_{b_s^3}\widehat{L}_R(f_W)\|_2^2)\right]^{1/2}.
\end{align*}
For $f_W\in\mathcal{F}_{NN}(\{m_1,m_2,A\},B_{inn},B_{out})$,
\begin{align*}
&\|\nabla_{W^3}\widehat{L}_R(f_{W})\|
\leq C(d,\sigma,coe,\Omega)\max\{1,1/\beta\}\sqrt{A}B_{inn}^4B_{out},\\
&\|\nabla_{W^3}\widehat{L}_N(f_{W})\|
\leq C(d,\sigma,coe,\Omega)\sqrt{A}B_{inn}^4B_{out}.
\end{align*}
\end{lemma}
\begin{proof}
We only present a proof for $\widehat{L}_R(f_W)$. The inequality of $\widehat{L}_N(f_W)$ can be proved in almost the same way. 

It suffice to derive upper bounds of $\|\nabla_{W_s^3}\widehat{L}_R(f_{W})\|_F$ and $\left|\frac{\partial }{\partial b_s^3}\widehat{L}_R(f_{W})\right|$. By definition,
\begin{align*}
&\nabla_{W_s^3}\widehat{L}_R(f_W)\\
&=\frac{|\Omega|}{n}\sum_{i=1}^{n}\left[\sum_{j=1}^{d}\frac{\partial f_W(X_i)}{\partial x_j}\nabla_{W_s^3}\left(\frac{\partial f_W(X_i)}{\partial x_j}\right)
+w(X_i)f_W(X_i)\nabla_{W_s^3}f_W(X_i)\right.\\
&\qquad\qquad\quad\left.-f(X_i)\nabla_{W_s^3}f_W(X_i)\right]\\
&\quad\ \ +\frac{|\partial\Omega|}{\beta m}\sum_{i=1}^{m}[f_W(Y_i)\nabla_{W_s^3}f_W(Y_i)-g(Y_j)\nabla_{W_s^3}f_W(Y_j)],
\end{align*}
from which we can see that in order to obtain upper bound of $\|\nabla_{W_s^3}\widehat{L}_R(f_W)\|_F$, it suffices to study $|f_W|,\left|\frac{\partial f_W}{\partial x_j}\right|,\|\nabla_{W_s^3}f_W\|_F,\left\|\nabla_{W_s^3}\left(\frac{\partial f_W}{\partial x_j}\right)\right\|_F$. But in the proof of Lemma \ref{function classes C1 norm} we already have $|f_W|\leq m_2B_{\sigma}B_{out}$ and $\left|\frac{\partial f_W}{\partial x_j}\right|\leq m_2\sqrt{m_1}B_{\sigma'}^2B_{inn}B_{out}$. Since $\nabla_{W_s^3}f_W=(f_s^2)^T$, we have
\begin{align*}
\|\nabla_{W_s^3}f_W\|_F&=\|f_s^2\|_F\leq\sqrt{m_2}B_{\sigma}.
\end{align*}
The gradient of $\frac{\partial f_W}{\partial x_j}$ with respect to $W_s^3$ can be calculated by
\begin{align*}
\nabla_{W_s^3}\left(\frac{\partial f_W}{\partial x_j}\right)&=(W_s^{1,j})^T\mathrm{diag}[\sigma'(f_{s,1}^{org})](W_s^2)^T\mathrm{diag}[\sigma'(f_{s,2}^{org})].
\end{align*}
where $W_s^{1,j}$ is the $j$th column of $W_s^1$. Thus
\begin{align*}
\left\|\nabla_{W_s^3}\left(\frac{\partial f_W}{\partial x_j}\right)\right\|_F&\leq\|W_s^{1,j}\|_F\|\mathrm{diag}[\sigma'(f_{s,1}^{org})]\|_F\|W_s^2\|_F\|\mathrm{diag}[\sigma'(f_{s,2}^{org})]\|_F\\
&\leq \sqrt{m_2m_1}B_{\sigma'}^2B_{inn}^2.
\end{align*}
Combining the above estimates, we derive that
\begin{align*}
&\|\nabla_{W_s^3}\widehat{L}_R(f_{W})\|_F
\leq C(d,\sigma,coe,\Omega)\max\{1,1/\beta\}B_{inn}^4B_{out}.
\end{align*}
We can derive in the same way that
\begin{align*}
\left|\frac{\partial }{\partial b_s^3}\widehat{L}_R(f_{W})\right|
\leq C(d,\sigma,coe,\Omega)\max\{1,1/\beta\}B_{out}. 
\end{align*}
\end{proof}

The following lemma present esimations for the strongly smooth parameter of $\widehat{L}_R$ and $\widehat{L}_N$.
\begin{lemma}\label{gradient Lip}
Define
\begin{align*}
&\|\nabla_W \widehat{L}_R(f_W)-\nabla_W \widehat{L}_R(f_{\widetilde{W}})\|:=\\
&\left[\sum_{s=1}^A\sum_{l=1}^3(\|\nabla_{W_s^l} \widehat{L}_R(f_W)-\nabla_{W_s^l} \widehat{L}_R(f_{\widetilde{W}})\|_F^2+\|\nabla_{b_s^l} \widehat{L}_R(f_W)-\nabla_{b_s^l} \widehat{L}_R(f_{\widetilde{W}})\|_2^2\right]^{1/2}
\end{align*}
and
\begin{align*}
\|{W}-\widetilde{W}\|:=\left[\sum_{s=1}^A\sum_{l=1}^3\|W_{{s}}^l-(\widetilde{W})_{s}^l\|_F^2+\|b_{{s}}^l-(\widetilde{b})_{s}^l\|_2^2\right]^{1/2}.
\end{align*}
For $f_W\in\mathcal{F}_{NN}(\{m_1,m_2,A\},B_{inn},B_{out})$, there holds
\begin{align*}
\|\nabla_{W}\widehat{L}_R(f_W)-\nabla_{W}\widehat{L}_R(f_{\widetilde{W}})\|
&\leq C(d,\sigma,coe,\Omega)\max\{1/\beta,1\}AB_{inn}^7B_{out}^3
\|W-\widetilde{W}\|,\\
\|\nabla_{W}\widehat{L}_N(f_W)-\nabla_{W}\widehat{L}_N(f_{\widetilde{W}})\|
&\leq C(d,\sigma,coe,\Omega)AB_{inn}^7B_{out}^3
\|W-\widetilde{W}\|.
\end{align*}
\end{lemma}
\begin{proof}
Since the proof is too lengthy, we put it in the appendix.
\end{proof}

The following result shows that the discrete losses are convex with respect to the neural network parameters in the outer layer. This property helps to ensure the global convergence of optimization algorithms.
\begin{lemma}\label{convex with respec to outer layer}
Given $W^1$ and $W^2$, $\widehat{L}_R(f_W)$ and $\widehat{L}_N(f_W)$ are convex with respect to $W^3$.
\end{lemma}
\begin{proof}
We only present a proof for $\widehat{L}_R(f_W)$. The convexity of $\widehat{L}_N(f_W)$ can be proved in almost the same way. Write $W^3$ as $W^3=(W_1^3,b_1^3,W_2^3,b_2^3,\cdots,W_A^3,b_A^3)^T\in\mathbb{R}^{A(m_2+1)}$. Let
$D_s=(W_s^1)^T\mathrm{diag}[\sigma'(f_{s,1}^{org})](W_s^2)^T\mathrm{diag}[\sigma'(f_{s,2}^{org})]\in\mathbb{R}^{d\times m_2}$ for $s\in[A]$ and
$D=(D_1,0,D_2,0,\cdots,D_A,0)\in\mathbb{R}^{d\times A(m_2+1)}$, then $\nabla_xf_W=\sum_{s=1}^AD_s(W_s^3)^T=DW^3$. Let $
\overline{f^2}=((f_1^2)^T,1,(f_2^2)^T,1,\cdots,(f_A^2)^T,1)\in\mathbb{R}^{1\times A(m_2+1)}$, then $f_W =\sum_{s=1}^{A}(W_s^3f_s^2+b_s^3)=\overline{f^2}W^3$ and $f_W^2=(W^3)^T(\overline{f^2})^T\overline{f^2}W^3$. Therefore we can express $\widehat{L}_{R}(f_W)$ as 
\begin{align*}
&\widehat{L}_{R}(f_W)=
(W^3)^THW^3-\left[\frac{|\Omega|}{n}\sum_{i=1}^{n}f^{(source)}(X_i)\overline{f^2}(X_i)+\frac{|\partial\Omega|}{\beta m}\sum_{i=1}^{m}g(Y_i)\overline{f^2}(Y_i)\right]W^3,
\end{align*}
where the Hessian matrix 
\begin{align*}
H=&\frac{|\Omega|}{n}\sum_{i=1}^{n}D(X_i)^TD(X_i)
+\frac{|\Omega|}{2n}\sum_{i=1}^{n}w(X_i)(\overline{f^2}(X_i))^T\overline{f^2}(X_i)\\
&+\frac{|\partial\Omega|}{2\beta m}\sum_{i=1}^{m}(\overline{f^2}(Y_i))^T\overline{f^2}(Y_i).
\end{align*}
Since $\|w\|_{L^{\infty(\Omega)}}\geq c_w$, $H$ is positive semi-definite. According to Lemma \ref{convex hessian condition}, we conclude that $\widehat{L}_{R}(f_W)$ is convex.
\end{proof}

According to \cite{kohler2022analysis}, as long as the number of subnetworks, $A$, is sufficiently large, with high probability close to 1, there exist $A'$ subnetworks whose parameters initialized randomly are sufficiently close to the parameters of the subnetwork corresponding to a given neural network function in $\mathcal{F}_{NN}$.
\begin{lemma}\label{initialization probability}
Let $N,A'\in\mathbb{N}_{>0}$. Let $f^*$ be the target function to be estimated. Let $f_{\bar{W}}$ be a given function in $\mathcal{F}_{NN}$. Suppose that $A$ can be divided by $A'$. Let $(W_0)_{s,i,j}^l,(b_0)_{s,i}^l(s\in[A];l=1,2;i\in[m_l];j\in [m_{l-1}])$ be uniformly distributed on 
\begin{align*}
\left[-C(d,f^*)N^{10d^3},C(d,f^*)N^{10d^3}\right].
\end{align*}
Let $E$ be the event that there exists an injection $\varphi:\{1,2,\cdots,A'\}\to\{1,2,\cdots,A\}$ such that for $k\in[A']$,
\begin{align*}
\varphi(k)\in\{(k-1)A/A'+i':i'\in[A/A']\}
\end{align*}
and for $l=1,2;i\in[m_l];j\in[m_{l-1}];s\in\{\varphi(1),\cdots,\varphi(A')\}$,
\begin{align*}
|(W_0)_{s,i,j}^l-(\bar{W})_{s,i,j}^l|&\leq\frac{1}{2\sqrt{m_lm_{l-1}}}\left(\frac{1}{N}\right)^{83d^3/2},\\
|(b_0)_{s,i}^l-(\bar{b})_{s,i}^l|&\leq\frac{1}{2\sqrt{m_l}}\left(\frac{1}{N}\right)^{83d^3/2}.
\end{align*}
We claim that $P(E)\geq1-A'e^{-C(d,f^*)\frac{A}{A'}\left(\frac{1}{N}\right)^{206d^4(d+3)5^{d+2}}}$.
\end{lemma}
\begin{proof}
We follow the idea in \cite{kohler2022analysis}. Given $k\in[A'],s\in\{(k-1)A/A'+i':i'\in[A/A']\}$, let $E_{k,s}$ be the event that for $l=1,2;i\in[m_l];j\in[m_{l-1}]$, 
\begin{align*}
|(W_0)_{s,i,j}^l-(\bar{W})_{s,i,j}^l|&\leq\frac{1}{2\sqrt{m_lm_{l-1}}}\left(\frac{1}{N}\right)^{\frac{3}{2}d^3+10d^2+\frac{31}{2}d+\frac{29}{2}},\\
|(b_0)_{s,i}^l-(\bar{b})_{s,i}^l|&\leq\frac{1}{2\sqrt{m_l}}\left(\frac{1}{N}\right)^{\frac{3}{2}d^3+10d^2+\frac{31}{2}d+\frac{29}{2}}.
\end{align*}
Then 
\begin{align*}
E=\bigcap_{k\in[A']}\bigcup_{s\in\{(k-1)A/A'+i':i'\in[A/A']\}} E_{k,s}.
\end{align*}
The probability of $E_{k,s}$ being true can be calculated directly.
\begin{align*}
P(E_{k,s})&=\left(\frac{1}{\sqrt{m_1d}}\right)^{m_1d}\left(\frac{1}{\sqrt{m_1}}\right)^{m_1}\left(\frac{1}{\sqrt{m_2m_1}}\right)^{m_2m_1}\left(\frac{1}{\sqrt{m_2}}\right)^{m_2}\\
&\quad\ \left(\frac{1}{N}\right)^{83d^3(m_1d+m_1+m_2m_1+m_2)/2}\left(\frac{1}{C(d,f^*)N^{10d^3}}\right)^{(m_1d+m_1+m_2m_1+m_2)}\\
&\geq C(d,f^*)\left(\frac{1}{N}\right)^{206d^4(d+3)5^{d+2}},
\end{align*}
where we use that $m_1=5d,m_2=3\left\lceil\frac{d+2}{2}\right\rceil\binom{2d+1}{d+1}
\leq\frac{3}{2}(d+3)5^{d+1}$. Denote 
\begin{align*}
E_{k}=\bigcup_{s\in\{(k-1)A/A'+i':i'\in[A/A']\}} E_{k,s}.
\end{align*}
Then 
\begin{align*}
&P(E_{k}^c)=P\left(\bigcap_{s\in\{(k-1)A/A'+i':i'\in[A/A']\}}E_{k,s}^c\right)\\
&=\prod_{s\in\{(k-1)A/A'+i':i'\in[A/A']\}}P(E_{k,s}^c)\leq\left[1-C(d,f^*)\left(\frac{1}{N}\right)^{206d^4(d+3)5^{d+2}}\right]^{A/A'}.
\end{align*}
By union bound we have
\begin{align*}
&P(E^c)=P\left(\bigcup_{k\in[A']}E_{k}^c\right)\leq A'\left[1-C(d,f^*)\left(\frac{1}{N}\right)^{206d^4(d+3)5^{d+2}}\right]^{A/A'}.
\end{align*}
Applying the inequality $e^x\geq1+x$ for any $x\in\mathbb{R}$, we finish the proof.
\end{proof}

The following result, with the initial idea from \cite{drews2023analysis}, demonstrates that the optimization error can be controlled by the sum of four terms. We can see that generalization error and approximation error occur again here. The difference from traditional approximation error lies in the fact that the approximation error here considers not only the approximation properties of a single neural network function but also the approximation properties of all neural network functions in the vicinity of a given neural network function.
\begin{theorem}\label{opt err}
Let $N,A'\in\mathbb{N}_{>0}$. Suppose that $A$ can be divided by $A'$. Let $(W_0)_{s,i,j}^l,(b_0)_{s,i}^l(s\in[A];l=1,2;i\in[m_l];j\in [m_{l-1}])$ be uniformly distributed on 
\begin{align*}
\left[-C(d,\sigma,coe,\Omega)N^{10d^3},C(d,\sigma,coe,\Omega)N^{10d^3}\right].
\end{align*}
Let $B_{inn}=C(d,coe,\Omega)N^{10d^3},B_{out}=C(d,coe,\Omega)N^{11d^3}$. Let
\begin{align*}
\mathcal{C}=\{W: &\|W^3\|_1\leq C(d,coe,\Omega)N^{11d^3};\\
&\left.\|W_s^l-(W_0)_s^l\|_F,\|b_s^l-(b_0)_s^l\|_2\leq\frac{1}{2}\left(\frac{1}{N}\right)^{83d^3/2},l=1,2,s\in[A]\right\}.
\end{align*}
Let $f_{\bar{W}}$ be a given function in $\mathcal{F}_{NN}$. Let
\begin{align*}
&\mathcal{B}_{\bar{W}}=\\
&\{W:W^3=\bar{W}^3;\exists\varphi\ s.t.\ \varphi(k)\in\{(k-1)A/A'+i':i'\in[A/A']\}\text{ for }k\in[A']\text{ and}\\
&\left.\|W_s^l-(\bar{W})_s^l\|_F,\|b_s^l-(\bar{b})_s^l\|_2\leq\left(\frac{1}{N}\right)^{83d^3/2},s\in\{\varphi(1),\cdots,\varphi(A')\},l=1,2\right\}.
\end{align*}
Let
\begin{align*}
\eta\leq C(d,\sigma,coe,\Omega)\left(\frac{1}{N}\right)^{103d^3}\frac{1}{A}.
\end{align*}
With probability at least $1-A'e^{-C(d,f^*)\frac{A}{A'}\left(\frac{1}{N}\right)^{206d^4(d+3)5^{d+2}}}$,
\begin{align*}
&\widehat{L}_R(f_{W_T})-\inf_{f_W\in\mathcal{F}_{NN}}\widehat{L}_R(f_W)\\
&\leq2\max\left\{\sup_{f_W\in\mathcal{F}_{NN}}[\widehat{L}_R(f_W)-{L}_R(f_W)],\sup_{f_W\in\mathcal{F}_{NN}}[{L}_R(f_W)-\widehat{L}_R(f_W)]\right\}\\
&\quad+2\sup_{W\in\mathcal{B}_{\bar{W}}}[L_R(f_W)-L_R(u_R)]+\frac{1}{2\eta T}\|(W_0)^3-(\bar{W})^3\|^2+\frac{\eta}{2T}\sum_{t=0}^{T-1}\|\nabla_{W^3}\widehat{L}_R\left(f_{W_t}\right)\|^2,\\
&\widehat{L}_N(f_{W_T})-\inf_{f_W\in\mathcal{F}_{NN}}\widehat{L}_N(f_W)\\
&\leq2\max\left\{\sup_{f_W\in\mathcal{F}_{NN}}[\widehat{L}_N(f_W)-{L}_N(f_W)],\sup_{f_W\in\mathcal{F}_{NN}}[{L}_N(f_W)-\widehat{L}_N(f_W)]\right\}\\
&\quad+2\sup_{W\in\mathcal{B}_{\bar{W}}}[L_N(f_W)-L_N(u_N)]+\frac{1}{2\eta T}\|(W_0)^3-(\bar{W})^3\|^2+\frac{\eta}{2T}\sum_{t=0}^{T-1}\|\nabla_{W^3}\widehat{L}_N\left(f_{W_t}\right)\|^2.
\end{align*}
\end{theorem}
\begin{proof}
We only present a proof for the Robin problem. The Neumann problem can be proved in the same way.

Let ${\bar{W}_t}=\{((W_t)_s^1,(b_t)_s^1),((W_t)_s^2,(b_t)_s^2),((\bar{W})_s^3,(\bar{b})_s^3),s\in[A]\}$.
By Lemma \ref{convex iff condition} and Lemma \ref{convex with respec to outer layer}, we have
\begin{align*}
\widehat{L}_R\left(f_{W_t}\right)-\widehat{L}_R\left(f_{\bar{W}_t}\right) 
&\leq[\nabla_{W^3} \widehat{L}_R\left(f_{W_t}\right)]^T((W_t)^3-(\bar{W})^3)\\ 
&=\frac{1}{\eta} [\eta\nabla_{W^3} \widehat{L}_R\left(f_{W_t}\right)]^T((W_t)^3-(\bar{W})^3) \\
&=\frac{1}{2\eta}\left(-\|(W_t)^3-(\bar{W})^3-\eta\nabla_{W^3} L_R\left(f_{W_t}\right)\|^2\right.\\
&\quad\left.+\|(W_t)^3-(\bar{W})^3\|^2+\|\eta\nabla_{W^3} \widehat{L}_R\left(f_{W_t}\right)\|^2\right).
\end{align*}
By Lemma \ref{projection operator property2},
\begin{align*}
&\|(W_t)^3-(\bar{W})^3-\eta\nabla_{W^3} \widehat{L}_R\left(f_{W_t}\right)\|^2\\
&\geq\|\mathrm{proj}_{\mathcal{C}}((W_t)^3-\eta\nabla_{W^3} \widehat{L}_R\left(f_{W_t}\right))-\mathrm{proj}_{\mathcal{C}}((\bar{W})^3)\|^2
=\|(W_{t+1})^3-(\bar{W})^3\|^2.
\end{align*}
Thus
\begin{align*}
& \widehat{L}_R\left(f_{W_t}\right)-\widehat{L}_R\left(f_{\bar{W}_t}\right)\\ 
&\leq\frac{1}{2\eta}\left(-\|(W_{t+1})^3-(\bar{W})^3\|^2+\|(W_t)^3-(\bar{W})^3\|^2+\|\eta\nabla_{W^3} \widehat{L}_R\left(f_{W_t}\right)\|^2\right),
\end{align*}
which implies
\begin{align}\label{opt err1}
&\frac{1}{T}\sum_{t=0}^{T-1} \widehat{L}_R\left(f_{W_t}\right)\nonumber\\
&\leq\frac{1}{T}\sum_{t=0}^{T-1} \widehat{L}_R\left(f_{\bar{W}_t}\right)+\frac{1}{2\eta T}\|(W_0)^3-(\bar{W})^3\|^2+\frac{\eta}{2T}\sum_{t=0}^{T-1}\|\nabla_{W^3} \widehat{L}_R\left(f_{W_t}\right)\|^2.
\end{align}
Next we prove that $\hat{L}_R$ decreases during iteration. Let $a\in[0,1]$. For $l=1,2;s\in[A]$, since $W_{t+1},W_t\in\mathcal{F}_{NN}(\{m_1,m_2,A\},B_{inn},B_{out})$, there holds
\begin{align*}
\|a(W_{t+1})_s^l+(1-a)(W_t)_s^l\|_F&\leq a\|(W_{t+1})_s^l\|_F+(1-a)\|(W_t)_s^l\|_F\leq B_{inn},\\
\|a(b_{t+1})_s^l+(1-a)(b_t)_s^l\|_2&\leq a\|(b_{t+1})_s^l\|_F+(1-a)\|(b_t)_s^l\|_F\leq B_{inn},\\
\|a(W_{t+1})^3+(1-a)(W_{t})^3\|_1&\leq a\|(W_{t+1})^3\|_1+(1-a)\|(W_t)^3\|_1\leq B_{out},
\end{align*}
which implies $f_{aW_{t+1}+(1-a)W_t}\in\mathcal{F}_{NN}(\{m_1,m_2,A\},B_{inn},B_{out})$. By Lemma \ref{gradient Lip},
\begin{align*}
\|\nabla_{W} \widehat{L}_R(f_{aW_{t+1}+(1-a)W_t})-\nabla_{W} \widehat{L}_R(f_{W_t})\|\leq L\|W_{t+1}-W_t\|.
\end{align*}
where $L=C(d,\sigma,coe,\Omega)\max\{1/\beta,1\}AB_{inn}^7B_{out}^3$.
Now we can apply Lemma \ref{smoothness inequality} and get
\begin{align*}
\widehat{L}_R(f_{W_{t+1}})&\leq \widehat{L}_R(f_{W_t})+[\nabla_W \widehat{L}_R(f_{W_t})]^T(W_{t+1}-W_s)+\frac{L}{2}\|W_{t+1}-W_t\|^2\\
&\leq\widehat{L}_R(f_{W_t})-\left(\frac{1}{\eta}-\frac{L}{2}\right)\|W_{t+1}-W_t\|^2,
\end{align*}
where we also make use of the relation
\begin{align*}
&[\nabla_W \widehat{L}_R(f_{W_t})]^T(W_{t+1}-W_s)\\
&=-\frac{1}{\eta}[W_{t}-\eta\nabla_W \widehat{L}_R(f_{W_t})-W_t]^T[\mathrm{proj}_{\mathcal{C}}(W_{t}-\eta\nabla_W \widehat{L}_R(f_{W_t}))-\mathrm{proj}_{\mathcal{C}}(W_t)] \\
&\leq-\frac{1}{\eta}\|\mathrm{proj}_{\mathcal{C}}(W_{t}-\eta\nabla_W \widehat{L}_R(f_{W_t}))-\mathrm{proj}_{\mathcal{C}}(W_t)\|^2=-\frac{1}{\eta}\|W_{t+1}-W_t\|^2.
\end{align*}
Here the second step is due to Lemma \ref{projection operator property2}. Now we know that $\hat{L}_R$ decreases during iteration provided  $\eta\leq\frac{2}{L}$. So based on (\ref{opt err1}), we obtain
\begin{align*}
&\widehat{L}_R(f_{W_T})-\inf_{f_W\in\mathcal{F}_{NN}}\widehat{L}_R(f_W)\leq\frac{1}{T}\sum_{t=0}^{T-1} \widehat{L}_R\left(f_{W_t}\right)-\inf_{f_W\in\mathcal{F}_{NN}}\widehat{L}_R(f_W)\\
&\leq\frac{1}{T}\sum_{t=0}^{T-1}\left[\widehat{L}_R\left(f_{\bar{W}_t}\right)-\inf_{f_W\in\mathcal{F}_{NN}}\widehat{L}_R(f_W)\right]+\frac{1}{2\eta T}\|(W_0)^3-(\bar{W})^3\|^2\\
&\quad+\frac{\eta}{2T}\sum_{t=0}^{T-1}\|\nabla_{W^3}\widehat{L}_R\left(f_{W_t}\right)\|^2.
\end{align*}
We further divide the first term on the right-hand side into three terms:
\begin{align*}
&\widehat{L}_R(f_{\bar{W}_t})-\inf_{f_W\in\mathcal{F}_{N N}}\widehat{L}_R(f_W)\\
&=[\widehat{L}_R(f_{\bar{W}_t})-{L}_R(f_{\bar{W}_t})]+[{L}_R(f_{\bar{W}_t})-{L}_R(u_R)]+\left[{L}_R(u_R)-\inf_{f_W\in\mathcal{F}_{N N}}\widehat{L}_R(f_W)\right].
\end{align*}
Since ${W_t}\in\mathcal{C}$ and $f_{\bar{W}}\in\mathcal{F}_{NN}$, there holds $f_{\bar{W}_t}\in\mathcal{F}_{NN}$ by our choice of $B_{inn}$ and $B_{out}$. Thus first term can be controlled by $\sup_{f_W\in\mathcal{F}_{NN}}[\widehat{L}_R(f_W)-{L}_R(f_W)]$. An upper bound of the third term is provided by Lemma \ref{continuous minimum and empirical minimum}. So it suffice to show that with probability at least $1-A'e^{-C(d,coe,\Omega)\frac{A}{A'}\left(\frac{1}{N}\right)^{206d^4(d+3)5^{d+2}}}$,
$\bar{W}_t\in\mathcal{B}_{\bar{W}}$ and hence 
\begin{align*}
{L}_R(f_{\bar{W}_t})-{L}_R(u_R)\leq\sup_{W\in\mathcal{B}_{\bar{W}}}[L_R(f_W)-L_R(u_R)].
\end{align*}
By Lemma \ref{initialization probability}, with that probability, there exists an injection $\varphi:\{1,2,\cdots,A'\}\to\{1,2,\cdots,A\}$ such that for $k\in[A']$,
\begin{align*}
\varphi(k)\in\{(k-1)A/A'+i':i'\in[A/A']\}
\end{align*}
and for $l=1,2;s\in\{\varphi(1),\cdots,\varphi(A')\}$,
\begin{align*}
\|(W_0)_{s}^l-(\bar{W})_{s}^l\|_F,\|(b_0)_{s}^l-(\bar{b})_{s}^l\|_2\leq\frac{1}{2}\left(\frac{1}{N}\right)^{83d^3/2}.
\end{align*}
Since $W_t\in\mathcal{C}$,
\begin{align*}
\|(W_t)_s^l-(W_0)_s^l\|_F,\|(b_t)_s^l-(b_0)_s^l\|_2\leq\frac{1}{2}\left(\frac{1}{N}\right)^{83d^3/2},s\in[A],l=1,2.
\end{align*}
From triangle inequality, we conclude that $\bar{W}_t\in\mathcal{B}_{\bar{W}}$ with that probability.
\end{proof}

\section{Convergence Rate of DRM}

\begin{theorem}\label{convergence rate}
Let $m=n$. Let $(W_0)_{s,i,j}^l,(b_0)_{s,i}^l(s\in[A];l=1,2;i\in[m_l];j\in [m_{l-1}])$ be uniformly distributed on 
\begin{align*}
\left[-C(d,coe,\Omega)n^{\frac{10d^3}{144d^3+2}},C(d,coe,\Omega)n^{\frac{10d^3}{144d^3+2}}\right].
\end{align*}
Let $B_{inn}=C(d,coe,\Omega)n^{\frac{10d^3}{144d^3+2}}
,B_{out}=C(d,coe,\Omega)n^{\frac{11d^3}{144d^3+2}}
$. Let 
\begin{align}\label{overparameterization condition}
A=n^{\frac{415d^4(d+3)5^{d+2}}{288d^3+4}}.
\end{align}
Let
\begin{align*}
\mathcal{C}=\{W: &\|W^3\|_1\leq C(d,coe,\Omega)n^{\frac{11d^3}{144d^3+2}};\\
&\left.\|W_s^l-(W_0)_s^l\|_F,\|b_s^l-(b_0)_s^l\|_2\leq\frac{1}{2}n^{-\frac{83d^3}{288d^3+4}},l=1,2,s\in[A]\right\}.
\end{align*}
Let
\begin{align}\label{step size}
\eta\leq C(d,coe,\Omega)n^{-\frac{103d^3}{144d^3+2}}\frac{1}{A}
\end{align}
and $T=\frac{1}{\eta}$. 
\begin{enumerate}[label=(\arabic*)]
\item
Consider Robin problem \eqref{second order elliptic equation}\eqref{robin}. With probability at least $1-\frac{C(d,coe,\Omega)}{n}$,
\begin{align*}
&\|f_{W_T}-u_R\|_{H^1(\Omega)}\leq C(d,coe,\Omega)\max\{1,1/\beta\}n^{-\frac{1}{288d^3+4}}.
\end{align*}
\item
Consider Neumann problem \eqref{second order elliptic equation}\eqref{neumann}. With probability at least $1-\frac{C(d,coe,\Omega)}{n}$,
\begin{align*}
&\|f_{W_T}-u_N\|_{H^1(\Omega)}\leq C(d,coe,\Omega)n^{-\frac{1}{288d^3+4}}.
\end{align*}
\item
Consider Dirichlet problem \eqref{second order elliptic equation}\eqref{dirichlet}. With probability at least $1-\frac{C(d,coe,\Omega)}{n}$,
\begin{align*}
&\|f_{W_T}-u_D\|_{H^1(\Omega)}\leq C(d,coe,\Omega)n^{-\frac{1}{576d^3+8}}.
\end{align*}
\end{enumerate}
\end{theorem}
\begin{remark}
\cite{lu2022machine} provides a lower bound of $\mathcal{O}(n^{-1/d})$ (when the solution of the equation is in $H^2(\Omega)$) for DRM, and they found a neural network estimator with an upper error bound of $\mathcal{O}(n^{-1/(d+2)}\log^{1/2} n)$. The upper bound we obtained here is far from this lower bound, mainly due to the large parameter values used in our three-layer network estimator (Theorem \ref{app err}, Proposition \ref{app err pre}). However, from the perspective of constructing a three-layer network to approximate functions in Sobolev space, this is inevitable. One solution is to decrease the parameter values by increasing the depth of the neural network (refer to \cite{guhring2021approximation}). However, increasing the network depth also means increasing the number of parameters, leading to an increase in computational complexity during network training. Therefore, finding a neural network estimator that improves the error bound of this work while maintaining manageable computational complexity is one of the future research focuses.
\end{remark}
\begin{proof}
We only provide a proof for the Robin problem. The result of the Neumann problem can be derived in the same way. The result of the Dirichlet problem can be obtained by combining the convergence rate of the Robin problem and Lemma \ref{penalty convergence}.

Since $W_t\in\mathcal{C}$ and 
$\|(W_0)_s^l\|_F,\|(b_0)_s^l\|_2\leq C(d,coe,\Omega)n^{\frac{10d^3}{144d^3+2}}$ for $l=1,2;s\in[A]$, we have $f_{W_t}\in\mathcal{F}_{NN}$ by our choice of $B_{inn}$ and $B_{out}$. Choosing $\bar{W}$ in Theorem \ref{opt err} to be $W^*$, the approximator in Theorem \ref{app err} when approximating $u_R$, and plugging it into
Lemma \ref{error decompostion}, we have with probability at least 
$1-A'e^{-C(d,coe,\Omega)\frac{A}{A'}n^{-\frac{206d^4(d+3)5^{d+2}}{144d^3+2}}}$,
\begin{align*}
&{L}_{R}(f_{W_T})-L_R(u_R)\\
&\leq4\max\left\{\sup_{f_W\in\mathcal{F}_{NN}}[\widehat{L}_R(f_W)-{L}_R(f_W)],\sup_{f_W\in\mathcal{F}_{NN}}[{L}_R(f_W)-\widehat{L}_R(f_W)]\right\}\\
&\quad+3\sup_{W\in\mathcal{B}_{W^*}}[L_R(f_W)-L_R(u_R)]+\frac{1}{2\eta T}\|(W^*)^3\|^2+\frac{\eta}{2T}\sum_{t=0}^{T-1}\|\nabla_{W^3}\widehat{L}_R\left(f_{W_t}\right)\|^2.
\end{align*}
Here we set
\begin{align*}
N=C(d,coe,\Omega)n^{\frac{1}{144d^3+2}}
\end{align*}
in Theorem \ref{opt err}. Our previous results provide estimates for each term on the right-hand side of the inequality. Choosing $\tau=C(d,\Omega,coe)\max\{1,1/\beta\}n^{-\frac{1}{144d^3+2}}$ in Theorem \ref{sta err},  we have with probability at least $1-2e^{-C(d,coe,\Omega)n^{\frac{60d^3}{144d^3+2}}}$,
\begin{align*}
&\max\left\{\sup_{f_W\in\mathcal{F}_{NN}}[\widehat{L}_{R}({f_W})-{L}_{R}({f_W})],\sup_{f_W\in\mathcal{F}_{NN}}[{L}_{R}({u})-\widehat{L}_{R}({f_W})]\right\}\\
&\leq C(d,\Omega,coe)\max\{1,1/\beta\}n^{-\frac{1}{144d^3+2}}.
\end{align*}
Choosing 
\begin{align*}
N=C(d,coe,\Omega)n^{\frac{1}{144d^3+2}}
\end{align*}
in Theorem \ref{app err}, there holds
\begin{align*}
\|(W^*)^3\|\leq C(d,coe,\Omega)\frac{1}{A'}N^{23d^3/2}=C(d,coe,\Omega)\frac{1}{A'}n^{\frac{23d^3}{288d^3+4}}.
\end{align*}
By Lemma \ref{estimate of gradient},
\begin{align*}
&\|\nabla_{W^3}\widehat{L}_R(f_{W})\|
\leq C(d,\sigma,\Omega)\max\{1,1/\beta\}\sqrt{A}n^{\frac{51d^3}{144d^3+2}}.
\end{align*}
According to Lemma \ref{uR regularity}, Lemma \ref{loss difference and variable difference} and Theorem \ref{app err}, 
\begin{align*}
\sup_{W\in\mathcal{B}_{W^*}}[L_R(f_W)-L_R(u_R)]&\leq C(\Omega,coe)\max\{1,1/\beta\}\sup_{W\in\mathcal{B}_{W^*}}\|f_W-u_R\|_{H^1(\Omega)}^2\\
&\leq
C(d,coe,\Omega)\max\{1,1/\beta\}n^{-\frac{1}{144d^3+2}}.
\end{align*}
Combining all the above upper bounds, we derive that with probability at least $1-A'e^{-C(d,coe,\Omega)\frac{A}{A'}n^{-\frac{206d^4(d+3)5^{d+2}}{144d^3+2}}}-2e^{-C(d,coe,\Omega)n^{\frac{60d^3}{144d^3+2}}}$,
\begin{align*}
&{L}_{R}(f_{W_T})-L_R(u_R)\\
&\leq C(d,coe,\Omega)\max\{1,1/\beta^2\}\left(n^{-\frac{1}{144d^3+2}}+\frac{1}{A'^2}n^{\frac{23d^3}{144d^3+2}}+\eta An^{\frac{51d^3}{72d^3+1}}\right).
\end{align*}
Let $A$ satisfy (\ref{overparameterization condition}) and $\eta$ satisfy (\ref{step size}), and further let
\begin{align*}
A'=C(d,coe,\Omega)n^{\frac{23d^3+1}{288d^3+4}}.
\end{align*}
After organizing and simplifying, we obtain that with probability at least $1-\frac{C(d,coe,\Omega)}{n}$,
\begin{align*}
{L}_{R}(f_{W_T})-L_R(u_R)\leq C(d,coe,\Omega)\max\{1,1/\beta^2\}n^{-\frac{1}{144d^3+2}}.
\end{align*}
The convergence rate of $\|f_W-u_R\|_{H^1(\Omega)}$ is then implied by Lemma \ref{loss difference and variable difference}.
\end{proof}

\section{Conclusion}

In this work, we focus on employing a three-layer tanh neural network within the framework of the deep Ritz method(DRM) to solve second-order elliptic equations with three different types of boundary conditions. To the best of our knowledge, we are the first to provide a comprehensive error analysis of using overparameterized networks to solve PDE problems. We present error bound in terms of the sample size $n$ and our work provides guidance on how to set the network depth, width, step size, and number of iterations for the projected gradient descent algorithm. Importantly, our assumptions in this work are classical and we do not require any additional assumptions on the solution of the equation. This ensures the broad applicability and generality of our results.

As an initial study, this work focuses on a three-layer neural network. However, we believe that the framework presented in this work can also be extended to analyze deep neural networks. We speculate that deep neural network estimators may have better upper bounds, and this will be one of the topics for our future research. Additionally, investigating other solution formats for PDEs, such as PINNs and WAN, as well as considering other types of PDEs, are also promising directions for future research. On the training front, this work investigates the projected gradient descent algorithm. Analyzing the convergence properties of gradient descent and stochastic gradient descent algorithms when applied to PDE problems presents technical challenges. In the future, we aim to address these technical difficulties and explore the convergence analysis of these optimization algorithms in the context of PDE problem-solving.

\appendix

\section{Proof of Lemma \ref{uR regularity}}
We first give an auxiliary lemma.
\begin{lemma}\label{lemma:regularity:robin}
Let assumption $(\ref{Assumption})$ holds. Let $u$ be the weak solution of the following Robin problem
	\begin{equation}\label{eq:regularity:robin} 
		\left\{\begin{aligned}
			-\Delta u + wu                                    & =0 &  & \text { in } \Omega          \\
u +	\beta \frac{\partial u}{\partial n}& =g &  & \text { on } \partial\Omega.
		\end{aligned}\right.
	\end{equation}
	Then $u_R\in H^2(\Omega)$ and
	\begin{equation*}
		\left\|u\right\|_{H^{2}(\Omega)} \leq C(\Omega,w)\frac{1}{\beta} \|g\|_{H^{1/2}(\partial\Omega)}.
	\end{equation*}
\end{lemma}

\begin{proof}
We follow the idea proposed in \cite{martin1996singularly} in a slightly different context. We first estimate the trace $T_0u=\left.u\right|_{\partial\Omega}$. We define the Dirichlet-to-Neumann map
	\begin{equation*}
		\widetilde{T}:\left.u\right|_{\partial\Omega} \mapsto \left.\frac{\partial u}{\partial n}\right|_{\partial\Omega},
	\end{equation*}
	where $u$ satisfies $-\Delta u+wu=0$ in $\Omega$, then
	\begin{equation*}
		T_0u=\left(\beta\widetilde{T}+I \right)^{-1}g.
	\end{equation*}
	Now we are going to show that $\beta\widetilde{T}+I$ is a positive definite operator in $L^{2}(\partial\Omega)$. We notice that the variational formulation of \eqref{eq:regularity:robin} can be read as follow:
	\begin{equation*}
		\int_{\Omega}\nabla u\cdot\nabla vdx+\int_{\Omega}wuvdx+\frac{1}{\beta}\int_{\partial\Omega}uvds=\frac{1}{\beta}\int_{\partial\Omega}gvds, \ \forall v\in H^{1}(\Omega).
	\end{equation*}
	Taking $v=u$, then we have
	\begin{equation*}
		\| T_0u \|_{L^{2}(\partial\Omega)}^{2} \leq \left\langle \left(\beta\widetilde{T}+I \right)T_0u,T_0u \right\rangle.
	\end{equation*}
	This means that $\beta\widetilde{T}+I$ is a positive definite operator in $L^{2}(\partial\Omega)$, and further, $(\beta\widetilde{T}+I)^{-1}$ is bounded. We have the estimate
	\begin{equation}\label{eq:trace:estimate}
		\| T_0u \|_{H^{1/2}(\partial\Omega)} \leq C(\Omega)\| g \|_{H^{1/2}(\partial\Omega)}.
	\end{equation}
	We rewrite the Robin problem \eqref{eq:regularity:robin} as follows
	\begin{equation*}
		\left\{\begin{aligned}
			-\Delta u + wu                  & =0                                                    &  & \text { in } \Omega          \\
			u+\frac{\partial u}{\partial n} & = \frac{1}{\beta} \left(g-\left(1-\beta\right)u\right) &  & \text { on } \partial\Omega.
		\end{aligned}\right.
	\end{equation*}
By Lemma \ref{uN regularity} we have
\begin{align}\label{eq:neumann:estimate}
\| u \|_{H^{2}(\Omega)} &\leq C(\Omega,w)\frac{1}{\beta}\left\| g-\left(1-\beta\right)T_0u \right\|_{H^{1/2}(\partial\Omega)}\nonumber\\
&\leq C(\Omega,w)\frac{1}{\beta} \left( \left\| g \right\|_{H^{1/2}(\partial\Omega)} + \left\| T_0u \right\|_{H^{1/2}(\partial\Omega)} \right).
	\end{align}
	Combining \eqref{eq:trace:estimate} and \eqref{eq:neumann:estimate}, we obtain the desired estimation.
\end{proof}
With the help of the above lemma, we now turn to prove Lemma \ref{uR regularity}.
\begin{proof}[Proof of Lemma \ref{uR regularity}]
We decompose the Robin problem (\ref{second order elliptic equation})(\ref{robin}) into two equations	\begin{equation*}
		\left\{ \begin{aligned}
			-\Delta u_{0}+wu_{0} & =f &  & \text{in}\ \Omega          \\
			u_{0}                & =g &  & \text{on}\ \partial\Omega, \\
		\end{aligned}\right.
	\end{equation*}
	\begin{equation*}
		\left\{ \begin{aligned}
			-\Delta u_{1}+wu_{1}                                     & =0                                  &  & \text{in}\ \Omega          \\
u_1+\beta\frac{\partial u_{1}}{\partial n} & =-\frac{\partial u_{0}}{\partial n} &  & \text{on}\ \partial\Omega. \\
		\end{aligned}\right.
	\end{equation*}
	and obtain the solution of (\ref{second order elliptic equation})(\ref{robin}):
	\begin{equation*}
		u=u_{0}+\beta u_{1}.
	\end{equation*}
	According to Lemma \ref{uD regularity}, we have	\begin{equation}\label{eq:reg:u0}
		\|u_{0}\|_{H^{2}(\Omega)} \leq C(\Omega,w)(\|f\|_{L^{2}(\Omega)}+\|g\|_{H^{3/2}(\partial\Omega)}).
	\end{equation}
Using Lemma \ref{lemma:regularity:robin}, it is easy to obtain
	\begin{equation}\label{eq:reg:u1}		\left\|u_{1}\right\|_{H^{2}(\Omega)} \leq C(\Omega,w)\frac{1}{\beta}\left\|\frac{\partial u_{0}}{\partial n}\right\|_{H^{1/2}(\partial\Omega)}
		\leq  C(\Omega,w)\frac{1}{\beta}\|u_{0}\|_{H^{2}(\Omega)},
	\end{equation}
	where the last inequality follows from Lemma \ref{trace theorem}. Combining \eqref{eq:reg:u0} and \eqref{eq:reg:u1}, the desired estimation can be derived by triangle inequality.
\end{proof}

\section{Proof of Lemma \ref{gradient Lip}}
\begin{proof}
We only present a proof for $\widehat{L}_R(f_W)$. The inequality of $\widehat{L}_N(f_W)$ can be proved in almost the same way. The whole proof is divided into five steps. In the first step, we provide upper bounds for $|f_W|$ and $|f_W-f_{\widetilde{W}}|$. The former can be derived easily:
\begin{align*}
|f_W|\leq\sum_{s=1}^A[\|W_s^3\|_F\|f_s^2\|_2+|b_s^3|]\leq \sum_{s=1}^A(\sqrt{m_2}B_{\sigma}\|W_s^3\|_F+|b_s^3|).
\end{align*}
In order to estimate $|f_{{W}}-f_{\widetilde{W}}|$, we need the following two inequalities:
\begin{align*}
&\|{f}_{s,1}^{org}(W)-f_{s,1}^{org}(\widetilde{W})\|_{2}\leq\|{W}_s^1-(\widetilde{W})_s^1\|_{2}\|x\|_{2}+\|{b}_s^1-(\widetilde{b})_s^1\|_{2},\\
&\|{f}_{s,2}^{org}(W)-f_{s,2}^{org}(\widetilde{W})\|_{2}\\
&\leq\|{W}_s^2-(\widetilde{W})_s^2\|_{2}\|{f}_s^1(W)\|_{2}+\|(\widetilde{W})_s^2\|_{2}\|{f}_s^1(W)-{f}_s^1(\widetilde{W})\|_{2}+\|{b}_s^2-(\widetilde{b})_s^2\|_{2}\\
&\leq {\sqrt{m_1}}B_{\sigma}\|{W}_s^2-(\widetilde{W})_s^2\|_{2}+\|{b}_s^2-(\widetilde{b})_s^2\|_{2}+L_{\sigma}\|(\widetilde{W})_s^2\|_{2}\|x\|_{2}\|{W}_s^1-(\widetilde{W})_s^1\|_{2}\\
&\quad+L_{\sigma}\|(\widetilde{W})_s^2\|_{2}\|{b}_s^1-(\widetilde{b})_s^1\|_{2}.
\end{align*}
Thus
\begin{align*}
&|f_{{W}}-f_{\widetilde{W}}|\leq\sum_{s=1}^A|{f}_s^3(W)-f_s^3(\widetilde{W})|\\
&\leq\sum_{s=1}^A(\|{W}_s^3-(\widetilde{W})_s^3\|_{2}\|{f}^2(W)\|_{2}+\|(\widetilde{W})_s^3\|_{2}\|{f}_s^2(W)-{f}_s^2(\widetilde{W})\|_{2}+|{b}_s^3-(\widetilde{b})_s^3|)\\
&\leq \sum_{s=1}^A({\sqrt{m_2}}B_{\sigma}\|{W}_s^3-(\widetilde{W})_s^3\|_{2}+|{b}_s^3-(\widetilde{b})_s^3|+{\sqrt{m_1}}L_{\sigma}B_{\sigma}\|(\widetilde{W})_s^3\|_{2}\|{W}_s^2-(\widetilde{W})_s^2\|_{2}\\
&\qquad\quad +L_{\sigma}\|(\widetilde{W})_s^3\|_{2}\|{b}_s^2-(\widetilde{b})_s^2\|_{2}+L_{\sigma}^2B_x\|(\widetilde{W})_s^3\|_{2}\|(\widetilde{W})_s^2\|_{2}\|{W}_s^1-(\widetilde{W})_s^1\|_{2}\\
&\qquad\quad +L_{\sigma}^2\|(\widetilde{W})_s^3\|_{2}\|(\widetilde{W})_s^2\|_{2}\|{b}_s^1-(\widetilde{b})_s^1\|_{2}).
\end{align*}

The second step is to give upper bounds of $\|\nabla_xf_W\|_2$ and $\|\nabla_xf_W-\nabla_xf_{\widetilde{W}}\|_2$. The gradient of $f_W$ with respect to the spatial variable $x$ can be calculated by
\begin{align*}
\nabla_xf_W=\sum_{s=1}^A(W_s^1)^T\mathrm{diag}[\sigma'(f_{s,1}^{org})](W_s^2)^T\mathrm{diag}[\sigma'(f_{s,2}^{org})](W_s^3)^T,
\end{align*}
therefore
\begin{align*}
\|\nabla_xf_W\|_2&\leq\sum_{s=1}^A\|W_s^1\|_2\|\mathrm{diag}[\sigma'(f_{s,1}^{org})]\|_2\|W_s^2\|_2\|\mathrm{diag}[\sigma'(f_{s,2}^{org})]\|_2\|W_s^3\|_2\\
&\leq \sum_{s=1}^A\sqrt{m_2m_1}B_{\sigma'}^2\|W_s^1\|_2\|W_s^2\|_2\|W_s^3\|_2.
\end{align*}
and
\begin{align*}
&\|\nabla_xf_W-\nabla_xf_{\widetilde{W}}\|_2\\
&\leq\sum_{s=1}^A[\|W_s^1-(\widetilde{W})_s^1\|_F\|\mathrm{diag}[\sigma'(f_{s,1}^{org}(W))]\|_F\|W_s^2\|_F\|\mathrm{diag}[\sigma'(f_{s,2}^{org}(W))]\|_F\|W_s^3\|_F
\\
&\quad\qquad+\|(\widetilde{W})_s^1\|_F\|\mathrm{diag}[\sigma'(f_{s,1}^{org}(W))]-\mathrm{diag}[\sigma'(f_{s,1}^{org}(\widetilde{W}))]\|_F\\
&\quad\qquad\quad\ \|W_s^2\|_F\|\mathrm{diag}[\sigma'(f_{s,2}^{org}(W))]\|_F\|W_s^3\|_F\\
&\quad\qquad+\|(\widetilde{W})_s^1\|_F\|\mathrm{diag}[\sigma'(f_{s,1}^{org}(\widetilde{W}))]\|_F\|W_s^2-(\widetilde{W})_s^2\|_F\\
&\quad\qquad\quad\ \|\mathrm{diag}[\sigma'(f_{s,2}^{org}(W))]\|_F\|W_s^3\|_F\\
&\quad\qquad+\|(\widetilde{W})_s^1\|_F\mathrm{diag}[\sigma'(f_{s,1}^{org}(\widetilde{W}))]\|_F\|(\widetilde{W})_s^2\|_F\\
&\quad\qquad\quad\ \|\mathrm{diag}[\sigma'(f_{s,2}^{org}(W))]-\mathrm{diag}[\sigma'(f_{s,2}^{org}(\widetilde{W}))]\|_F\|(W_s^3)^T\|_2\\
&\quad\qquad+\|(\widetilde{W})_s^1\|_F\|\mathrm{diag}[\sigma'(f_{s,1}^{org}(\widetilde{W}))]\|_F\|(\widetilde{W})_s^2\|_F\\
&\quad\qquad\quad\ \|\mathrm{diag}[\sigma'(f_{s,2}^{org}(\widetilde{W}))]\|_F\|W_s^3-(\widetilde{W})_s^3\|_F]\\
&\leq C(\sigma)\sum_{s=1}^A[\sqrt{m_2m_1}\|(\widetilde{W})_s^1\|_F\|(\widetilde{W})_s^2\|_F\|W_s^3-(\widetilde{W})_s^3\|_F\\
&\quad\qquad+\max\{\sqrt{m_2},\sqrt{m_1}\}\sqrt{m_1}\|(\widetilde{W})_s^1\|_F\|(\widetilde{W})_s^2\|_F\|W_s^3\|_F\|{W}_s^2-(\widetilde{W})_s^2\|_{F}\\
&\quad\qquad+\sqrt{m_1}\|(\widetilde{W})_s^1\|_F\|(\widetilde{W})_s^2\|_F\|W_s^3\|_F\|{b}_s^2-(\widetilde{b})_s^2\|_{2}\\
&\quad\qquad+\sqrt{m_2m_1}\|W_s^3\|_F\|(\widetilde{W})_s^1\|_F\max\{\|W_s^2\|_F^2,\|(\widetilde{W})_s^2\|_F^2\}\|{W}_s^1-(\widetilde{W})_s^1\|_{F}\\
&\quad\qquad+\max\{\sqrt{m_2},\sqrt{m_1}\}\|(\widetilde{W})_s^1\|_F\|W_s^3\|_F\max\{\|W_s^2\|_F^2,\|(\widetilde{W})_s^2\|_F^2\}\|{b}_s^1-(\widetilde{b})_s^1\|_{2}].
\end{align*}

The third step is to derive the estimates of $\|\nabla_{W_s^l}f_W\|_F,\|\nabla_{b_s^l}f_W\|_F$ and $\|\nabla_{W_s^l}f_W-\nabla_{W_s^l}f_{\widetilde{W}}\|_F,\|\nabla_{b_s^l}f_W-\nabla_{b_s^l}f_{\widetilde{W}}\|_F(l=1,2,3)$. The gradient of $f_W$ with respect to weights and biases in each layer can be calculated by
\begin{align*}
\nabla_{W_s^3}f_W&=(f_s^2)^T,\\
\nabla_{b_s^3}f_W&=1,\\
\nabla_{W_s^2}f_W&=\mathrm{diag}[\sigma'(f_{s,2}^{org})](W_s^3)^T(f_s^1)^T,\\
\nabla_{b_s^2}f_W&=\mathrm{diag}[\sigma'(f_{s,2}^{org})](W_s^3)^T,\\
\nabla_{W_s^1}f_W&=\mathrm{diag}[\sigma'(f_{s,1}^{org})](W_s^2)^T\mathrm{diag}[\sigma'(f_{s,2}^{org})](W_s^3)^Tx^T,\\
\nabla_{b_s^1}f_W&=\mathrm{diag}[\sigma'(f_{s,1}^{org})](W_s^2)^T\mathrm{diag}[\sigma'(f_{s,2}^{org})](W_s^3)^T,
\end{align*}
hence we have
\begin{align*}
\|\nabla_{W_s^3}f_W\|_F&=\|f_s^2\|_F\leq\sqrt{m_2}B_{\sigma},\quad |\nabla_{b_s^3}f_W|=1,\\
\|\nabla_{W_s^2}f_W\|_F&\leq\|\mathrm{diag}[\sigma'(f_{s,2}^{org})]\|_F\|W_s^3\|_F\|f_s^1\|_F
\leq \sqrt{m_2m_1}B_{\sigma}B_{\sigma'}\|W_s^3\|_F,\\
\|\nabla_{b_s^2}f_W\|_F&\leq\|\mathrm{diag}[\sigma'(f_{s,2}^{org})]\|_F\|W_s^3\|_F
\leq \sqrt{m_2}B_{\sigma'}\|W_s^3\|_F,\\
\|\nabla_{W_s^1}f_W\|_F&=\|\mathrm{diag}[\sigma'(f_{s,1}^{org})]\|_F\|W_s^2\|_F\|\mathrm{diag}[\sigma'(f_{s,2}^{org})]\|_F\|W_s^3\|_F\|x\|_F\\
&\leq \sqrt{m_1m_2}B_{\sigma'}^2B_x\|W_s^2\|_F\|W_s^3\|_F,\\
\|\nabla_{b_s^1}f_W\|_F&=\|\mathrm{diag}[\sigma'(f_{s,1}^{org})]\|_F\|W_s^2\|_F\|\mathrm{diag}[\sigma'(f_{s,2}^{org})]\|_F\|W_s^3\|_F\\
&\leq \sqrt{m_1m_2}B_{\sigma'}^2\|W_s^2\|_F\|W_s^3\|_F.
\end{align*}
and
\begin{align*}
&\|\nabla_{W_s^3}f_W-\nabla_{W_s^3}f_{\widetilde{W}}\|_F=\|f_s^2(W)-f_s^2(\widetilde{W})\|_2\\
&\leq {\sqrt{m_1}}L_{\sigma}B_{\sigma}\|{W}_s^2-(\widetilde{W})_s^2\|_{2}+L_{\sigma}\|{b}_s^2-(\widetilde{b})_s^2\|_{2}+L_{\sigma}^2B_x\|(\widetilde{W})_s^2\|_{2}\|{W}_s^1-(\widetilde{W})_s^1\|_{2}\\
&\quad+L_{\sigma}^2\|(\widetilde{W})_s^2\|_{2}\|{b}_s^1-(\widetilde{b})_s^1\|_{2},\\
&\|\nabla_{b_s^3}f_W-\nabla_{b_s^3}f_{\widetilde{W}}\|_F=0,\\
&\|\nabla_{W_s^2}f_W-\nabla_{W_s^2}f_{\widetilde{W}}\|_F\\
&\leq\|\mathrm{diag}[\sigma'(f_{s,2}^{org}(W))]-\mathrm{diag}[\sigma'(f_{s,2}^{org}(\widetilde{W}))]\|_F\|W_s^3\|_F\|f_s^1(W)\|_2\\
&\quad+\|\mathrm{diag}[\sigma'(f_{s,2}^{org}(\widetilde{W}))]\|_F\|W_s^3-(\widetilde{W})_s^3\|_F\|f_s^1(W)\|_2\\
&\quad+\|\mathrm{diag}[\sigma'(f_{s,2}^{org}(\widetilde{W}))]\|_F\|(\widetilde{W})_s^3\|_F\|f_s^1(W)-f_s^1(\widetilde{W})\|_2\\
&\leq \sqrt{m_2m_1}B_{\sigma}B_{\sigma'}\|W_s^3-(\widetilde{W})_s^3\|_F+{\sqrt{m_1}}\sqrt{m_1}L_{\sigma'}B_{\sigma}^2\|W^3\|_F\|{W}_s^2-(\widetilde{W})_s^2\|_{2}\\
&+\sqrt{Am_1}L_{\sigma'}B_{\sigma}\|W_s^3\|_F\|{b}_s^2-(\widetilde{b})_s^2\|_{2}\\
&\quad+L_{\sigma}B_x(\sqrt{m_1}L_{\sigma'}B_{\sigma}\|W_s^3\|_F\|(\widetilde{W})_s^2\|_{2}+\sqrt{m_2}B_{\sigma'}\|(\widetilde{W})_s^3\|_F)\|{W}_s^1-(\widetilde{W})_s^1\|_{2}\\
&\quad+L_{\sigma}(\sqrt{m_1}L_{\sigma'}B_{\sigma}\|W_s^3\|_F\|(\widetilde{W})_s^2\|_{2}+\sqrt{m_2}B_{\sigma'}\|(\widetilde{W})_s^3\|_F)\|{b}_s^1-(\widetilde{b})_s^1\|_{2},\\
&\|\nabla_{b_s^2}f_W-\nabla_{b_s^2}f_{\widetilde{W}}\|_F\\
&\leq\|\mathrm{diag}[\sigma'(f_{s,2}^{org}(W))]-\mathrm{diag}[\sigma'(f_{s,2}^{org}(\widetilde{W}))]\|_F\|W_s^3\|_F\\
&\quad+\|\mathrm{diag}[\sigma'(f_{s,2}^{org}(\widetilde{W}))]\|_F\|W_s^3-(\widetilde{W})_s^3\|_F\\
&\leq\sqrt{m_2}B_{\sigma'}\|W_s^3-(\widetilde{W})_s^3\|_F+{\sqrt{m_1}}L_{\sigma'}B_{\sigma}\|W_s^3\|_F\|{W}_s^2-(\widetilde{W})_s^2\|_{2}\\
&\quad+L_{\sigma'}\|W_s^3\|_F\|{b}_s^2-(\widetilde{b})_s^2\|_{2}+L_{\sigma}L_{\sigma'}B_x\|W_s^3\|_F\|(\widetilde{W})_s^2\|_{2}\|{W}_s^1-(\widetilde{W})_s^1\|_{2}\\
&\quad+L_{\sigma}L_{\sigma'}\|W_s^3\|_F\|(\widetilde{W})_s^2\|_{2}\|{b}_s^1-(\widetilde{b})_s^1\|_{2},\\
&\|\nabla_{W_s^1}f_W-\nabla_{W_s^1}f_{\widetilde{W}}\|_F\\
&\leq\|\mathrm{diag}[\sigma'(f_{s,1}^{org}(W))]-\mathrm{diag}[\sigma'(f_{s,1}^{org}(\widetilde{W}))]\|_F\|W_s^2\|_F\\
&\quad\ \|\mathrm{diag}[\sigma'(f_{s,2}^{org}(W))]\|_F\|W_s^3\|_F\|x\|_2\\
&\quad+\|\mathrm{diag}[\sigma'(f_{s,1}^{org}(\widetilde{W}))]\|_F\|W_s^2-(\widetilde{W})_s^2\|_F\|\mathrm{diag}[\sigma'(f_{s,2}^{org}(W))]\|_F\|W_s^3\|_F\|x\|_2\\
&\quad+\|\mathrm{diag}[\sigma'(f_{s,1}^{org}(\widetilde{W}))]\|_F\|(\widetilde{W})_s^2\|_F\\
&\qquad\ \|\mathrm{diag}[\sigma'(f_{s,2}^{org}(W))]-\mathrm{diag}[\sigma'(f_{s,2}^{org}(\widetilde{W}))]\|_F\|W_s^3\|_F\|x\|_2\\
&\quad+\|\mathrm{diag}[\sigma'(f_{s,1}^{org}(\widetilde{W}))]\|_F\|(\widetilde{W})_s^2\|_F\|\mathrm{diag}[\sigma'(f_{s,2}^{org}(\widetilde{W}))]\|_F\|W_s^3-(\widetilde{W})_s^3\|_F\|x\|_2\\
&\leq\sqrt{m_2m_1}B_{\sigma'}^2\|(\widetilde{W})_s^2\|_F\|W_s^3-(\widetilde{W})_s^3\|_F\\
&\quad+\sqrt{m_1}B_{\sigma'}\|W_s^3\|_F(\sqrt{m_2}B_{\sigma'}+{\sqrt{m_1}}L_{\sigma'}B_{\sigma}\|(\widetilde{W})_s^2\|_F)\|{W}_s^2-(\widetilde{W})_s^2\|_{F}\\
&\quad+\sqrt{m_1}L_{\sigma'}B_{\sigma'}\|(\widetilde{W})_s^2\|_F\|W_s^3\|_F\|{b}_s^2-(\widetilde{b})_s^2\|_{2}\\
&\quad+L_{\sigma'}B_{\sigma'}\|W_s^3\|_F(\sqrt{m_2}\|W_s^2\|_F+\sqrt{m_1}L_{\sigma}\|(\widetilde{W})_s^2\|_F^2)\|{W}_s^1-(\widetilde{W})_s^1\|_{2}\\
&\quad+L_{\sigma'}B_{\sigma'}\|W_s^3\|_F(\sqrt{m_2}\|W_s^2\|_F+\sqrt{m_1}L_{\sigma}\|(\widetilde{W})_s^2\|_F^2)\|{b}_s^1-(\widetilde{b})_s^1\|_{2},\\
&\|\nabla_{b_s^1}f_W-\nabla_{b_s^1}f_{\widetilde{W}}\|_F\\
&\leq\|\mathrm{diag}[\sigma'(f_{s,1}^{org}(W))]-\mathrm{diag}[\sigma'(f_{s,1}^{org}(\widetilde{W}))]\|_F\|W_s^2\|_F\\
&\quad\ \|\mathrm{diag}[\sigma'(f_{s,2}^{org}(W))]\|_F\|W_s^3\|_F\\
&\quad+\|\mathrm{diag}[\sigma'(f_{s,1}^{org}(\widetilde{W}))]\|_F\|W_s^2-(\widetilde{W})_s^2\|_F\|\mathrm{diag}[\sigma'(f_{s,2}^{org}(W))]\|_F\|W_s^3\|_F\\
&\quad+\|\mathrm{diag}[\sigma'(f_{s,1}^{org}(\widetilde{W}))]\|_F\|(\widetilde{W})_s^2\|_F\\
&\qquad\ \|\mathrm{diag}[\sigma'(f_{s,2}^{org}(W))]-\mathrm{diag}[\sigma'(f_{s,2}^{org}(\widetilde{W}))]\|_F\|W_s^3\|_F\\
&\quad+\|\mathrm{diag}[\sigma'(f_{s,1}^{org}(\widetilde{W}))]\|_F\|(\widetilde{W})_s^2\|_F\|\mathrm{diag}[\sigma'(f_{s,2}^{org}(\widetilde{W}))]\|_F\|W_s^3-(\widetilde{W})_s^3\|_F\\
&\leq\sqrt{m_2m_1}B_{\sigma'}^2\|(\widetilde{W})_s^2\|_F\|W_s^3-(\widetilde{W})_s^3\|_F\\
&\quad+\sqrt{m_1}B_{\sigma'}\|W_s^3\|_F(\sqrt{m_2}B_{\sigma'}+{\sqrt{m_1}}L_{\sigma'}B_{\sigma}\|(\widetilde{W})_s^2\|_F)\|{W}_s^2-(\widetilde{W})_s^2\|_{F}\\
&\quad+\sqrt{m_1}L_{\sigma'}B_{\sigma'}\|(\widetilde{W})_s^2\|_F\|W_s^3\|_F\|{b}_s^2-(\widetilde{b})_s^2\|_{2}\\
&\quad+L_{\sigma'}B_{\sigma'}\|W_s^3\|_F(\sqrt{m_2}\|W_s^2\|_F+\sqrt{m_1}L_{\sigma}\|(\widetilde{W})_s^2\|_F^2)\|{W}_s^1-(\widetilde{W})_s^1\|_{2}\\
&\quad+L_{\sigma'}B_{\sigma'}\|W_s^3\|_F(\sqrt{m_2}\|W_s^2\|_F+\sqrt{m_1}L_{\sigma}\|(\widetilde{W})_s^2\|_F^2)\|{b}_s^1-(\widetilde{b})_s^1\|_{2}.
\end{align*}

The fourth step is to present upper bounds of $\left\|\nabla_{W_s^l}\left(\frac{\partial f_W}{\partial x_j}\right)\right\|_F,\left\|\nabla_{b_s^l}\left(\frac{\partial f_W}{\partial x_j}\right)\right\|_F$ and $\left\|\nabla_{W_s^l}\left(\frac{\partial f_W}{\partial x_j}\right)-\nabla_{W_s^l}\left(\frac{\partial f_{\widetilde{W}}}{\partial x_j}\right)\right\|_F,\left\|\nabla_{b_s^l}\left(\frac{\partial f_W}{\partial x_j}\right)-\nabla_{b_s^l}\left(\frac{\partial f_{\widetilde{W}}}{\partial x_j}\right)\right\|_F(l=1,2,3)$. Let $W_s^1=(W_s^{1,1},\cdots,W_s^{1,d})$ with $W_{s}^{1,j}\in\mathbb{R}^{m_1},j=1,\cdots,d$. The partial derivative of $f_W$ with respect to the spatial variables $x_j(j=1,\cdots,d)$ is
\begin{align*}
\frac{\partial f_W}{\partial x_j}=\sum_{s=1}^A(W_s^{1,j})^T\mathrm{diag}[\sigma'(f_{s,1}^{org})](W_s^2)^T\mathrm{diag}[\sigma'(f_{s,2}^{org})](W_s^3)^T.
\end{align*}
and its gradient with respect to weights and biases in each layer can be calculated by
\begin{align*}
\nabla_{W_s^3}\left(\frac{\partial f_W}{\partial x_j}\right)&=(W_s^{1,j})^T\mathrm{diag}[\sigma'(f_{s,1}^{org})](W_s^2)^T\mathrm{diag}[\sigma'(f_{s,2}^{org})],\\
\nabla_{W_s^2}\left(\frac{\partial f_W}{\partial x_j}\right)
&=\mathrm{diag}[\sigma'(f_{s,2}^{org})](W_s^3)^T(W_s^{1,j})^T\mathrm{diag}[\sigma'(f_{s,1}^{org})]\\
&\quad+(W_s^2\mathrm{diag}[\sigma'(f_{s,1}^{org})]W_s^{1,j})\odot (\mathrm{diag}[\sigma''(f_{s,2}^{org})](W_s^3)^T)(f_s^1)^T,\\
\nabla_{b_s^2}\left(\frac{\partial f_W}{\partial x_j}\right)&=(W_s^2\mathrm{diag}[\sigma'(f_{s,1}^{org})]W_s^{1,j})\odot(\mathrm{diag}[\sigma''(f_{s,2}^{org})](W_s^3)^T),\\
\nabla_{W_s^1}\left(\frac{\partial f_W}{\partial x_j}\right)
&=(0,\cdots,0,\underbrace{\mathrm{diag}[\sigma'(f_{s,1}^{org})](W_s^2)^T\mathrm{diag}[\sigma'(f_{s,2}^{org})](W_s^3)^T}_{j\mathrm{th}},0,\cdots,0)\\
&\quad+(\mathrm{diag}[\sigma''(f_{s,1}^{org})]W_s^{1,j})\odot ((W_s^2)^T\mathrm{diag}[\sigma'(f_{s,2}^{org})](W_s^3)^T)x^T\\
&\quad+\mathrm{diag}[\sigma'(f_{s,1}^{org})](W_s^2)^T\\
&\qquad((W_s^2\mathrm{diag}[\sigma'(f_{s,1}^{org})]W_s^{1,j})\odot (\mathrm{diag}[\sigma''(f_2^{org})](W_s^3)^T))x^T,\\
\nabla_{b_s^1}\left(\frac{\partial f_W}{\partial x_j}\right)&=
(\mathrm{diag}[\sigma''(f_{s,1}^{org})]W_s^{1,j})\odot((W_s^2)^T\mathrm{diag}[\sigma'(f_{s,2}^{org})](W_s^3)^T)\\
&\quad+\mathrm{diag}[\sigma'(f_{s,1}^{org})](W_s^2)^T\\
&\qquad((W_s^2\mathrm{diag}[\sigma'(f_{s,1}^{org})]W_s^{1,j})\odot(\mathrm{diag}[\sigma''(f_{s,2}^{org})](W_s^3)^T)).
\end{align*} 
Employing the inequality $\|a\odot b\|_2\leq\|a\|_2\|b\|_2$ for any vectors $a$ and $b$, we have
\begin{align*}
\left\|\nabla_{W_s^3}\left(\frac{\partial f_W}{\partial x_j}\right)\right\|_F&\leq\|W_s^{1,j}\|_F\|\mathrm{diag}[\sigma'(f_{s,1}^{org})]\|_F\|W_s^2\|_F\|\mathrm{diag}[\sigma'(f_{s,2}^{org})]\|_F\\
&\leq \sqrt{m_2m_1}\|W_s^{1,j}\|_F\|W_s^{2}\|_F,\\
\left\|\nabla_{W_s^2}\left(\frac{\partial f_W}{\partial x_j}\right)\right\|_F
&\leq\|\mathrm{diag}[\sigma'(f_{s,2}^{org})]\|_F\|W_s^3\|_F\|W_s^{1,j}\|_F\|\mathrm{diag}[\sigma'(f_{s,1}^{org})]\|_F
\\
&\quad+\|W_s^2\|_F\|\mathrm{diag}[\sigma'(f_{s,1}^{org})]\|_F\|W_s^{1,j}\|_F\|\\
&\qquad\ \|\mathrm{diag}[\sigma''(f_{s,2}^{org})]\|_F\|W_s^3\|_F\|f_s^1\|_F\\
&\leq 2\sqrt{m_2}m_1B_{\sigma}B_{\sigma'}^2B_{\sigma''}\|W_s^{1,j}\|_F\|W_s^2\|_F\|W_s^3\|_F,\\
\left\|\nabla_{b_s^2}\left(\frac{\partial f_W}{\partial x_j}\right)\right\|_F
&\leq\|W_s^2\|_F\|\mathrm{diag}[\sigma'(f_{s,1}^{org})]\|_F\|W_s^{1,j}\|_F\|\|\mathrm{diag}[\sigma''(f_{s,2}^{org})]\|_F\|W_s^3\|_F\\
&\leq \sqrt{m_2m_1}B_{\sigma'}B_{\sigma''}\|W_s^{1,j}\|_F\|W_s^2\|_F\|W_s^3\|_F,\\
\left\|\nabla_{W_s^1}\left(\frac{\partial f_W}{\partial x_j}\right)\right\|_F&\leq
\|{\mathrm{diag}[\sigma'(f_{s,1}^{org})]\|_F\|W_s^2\|_F\|\mathrm{diag}[\sigma'(f_{s,2}^{org})]\|_F\|W_s^3}\|_F\\
&\quad+\|\mathrm{diag}[\sigma''(f_{s,1}^{org})]\|_F\|W_s^{1,j}\|_F\|W_s^2\|_F\\
&\qquad\ \|\mathrm{diag}[\sigma'(f_{s,2}^{org})]\|_F\|W_s^3\|_F\|x\|_2\\
&\quad+\|\mathrm{diag}[\sigma'(f_{s,1}^{org})]\|_F\|W_s^2\|_F\|W_s^2\|_F\\
&\qquad\ \|\mathrm{diag}[\sigma'(f_{s,1}^{org})]\|_F\|W_s^{1,j}\|_F\|\mathrm{diag}[\sigma''(f_{s,2}^{org})]\|_F\|W_s^3\|_F\|x\|_2\\
&\leq 3\sqrt{m_2}m_1B_{\sigma'}^2B_{\sigma''}B_x\|W_s^{1,j}\|_F\|W_s^2\|_F^2\|W_s^3\|_F,\\
\left\|\nabla_{b_s^1}\left(\frac{\partial f_W}{\partial x_j}\right)\right\|_F&\leq
\|\mathrm{diag}[\sigma''(f_{s,1}^{org})]\|_F\|W_s^{1,j}\|_F\|W_s^2\|_F\|\mathrm{diag}[\sigma'(f_{s,2}^{org})]\|_F\|W_s^3\|_F\\
&\quad+\|\mathrm{diag}[\sigma'(f_{s,1}^{org})]\|_F\|W_s^2\|_F\|W_s^2\|_F\\
&\qquad\ \|\mathrm{diag}[\sigma'(f_{s,1}^{org})]\|_F\|W_s^{1,j}\|_F\|\mathrm{diag}[\sigma''(f_{s,2}^{org})]\|_F\|W_s^3\|_F\\
&\leq 2\sqrt{m_2}m_1B_{\sigma'}^2B_{\sigma''}\|W_s^{1,j}\|_F\|W_s^2\|_F^2\|W_s^3\|_F.
\end{align*} 
and
\begin{align*}
&\left\|\nabla_{W_s^3}\left(\frac{\partial f_W}{\partial x_j}\right)-\nabla_{W_s^3}\left(\frac{\partial f_{\widetilde{W}}}{\partial x_j}\right)\right\|_F\\
&\leq\|W_s^{1,j}-(\widetilde{W})_s^{1,j}\|_F\|\mathrm{diag}[\sigma'(f_{s,1}^{org}(W))]\|_F\|W_s^2\|_F\|\mathrm{diag}[\sigma'(f_{s,2}^{org}(W))]\|_F\\
&\quad+\|(\widetilde{W})_s^{1,j}\|_F\|\mathrm{diag}[\sigma'(f_{s,1}^{org}(W))]-\mathrm{diag}[\sigma'(f_{s,1}^{org}(\widetilde{W}))]\|_F\\
&\qquad\ \|W_s^2\|_F\|\mathrm{diag}[\sigma'(f_{s,2}^{org}(W))]\|_F\\
&\quad+\|(\widetilde{W})_s^{1,j}\|_F\|\mathrm{diag}[\sigma'(f_{s,1}^{org}(\widetilde{W}))]\|_F\|W_s^2-(\widetilde{W})_s^2\|_F\|\mathrm{diag}[\sigma'(f_{s,2}^{org}(W))]\|_F\\
&\quad+\|(\widetilde{W})_s^{1,j}\|_F\|\mathrm{diag}[\sigma'(f_{s,1}^{org}(\widetilde{W}))]\|_F\\
&\qquad\ \|(\widetilde{W})_s^2\|_F\|\mathrm{diag}[\sigma'(f_{s,2}^{org}(W))]-\mathrm{diag}[\sigma'(f_{s,2}^{org}(\widetilde{W}))]\|_F\\
&\leq C(\sigma)(\max\{{m_2},m_1\}\|(\widetilde{W})_s^{1,j}\|_F\|(\widetilde{W})_s^2\|_F\|W_s^2-(\widetilde{W})_s^2\|_F\\
&\quad+\sqrt{m_1}\|(\widetilde{W})_s^{1,j}\|_F\|(\widetilde{W})_s^2\|_F\|b_s^2-(\widetilde{b})_s^2\|_2\\
&\quad+\max\{\sqrt{m_2},\sqrt{m_1}\}\|(\widetilde{W})_s^{1,j}\|_F\max\{\|(\widetilde{W})_s^2\|_F^2,\|W_s^2\|_F^2\}\|W_s^1-(\widetilde{W})_s^1\|_F\\
&\quad+\max\{\sqrt{m_2},\sqrt{m_1}\}\|(\widetilde{W})_s^{1,j}\|_F\max\{\|(\widetilde{W})_s^2\|_F^2,\|W_s^2\|_F^2\}\|b_s^1-(\widetilde{b})_s^1\|_2\\
&\quad+\sqrt{m_2m_1}\|W_s^2\|_F\|W_s^{1,j}-(\widetilde{W})_s^{1,j}\|_F),\\
&\left\|\nabla_{b_s^3}\left(\frac{\partial f_W}{\partial x_j}\right)-\nabla_{b_s^3}\left(\frac{\partial f_W}{\partial x_j}\right)\right\|_F=0,\\
&\left\|\nabla_{W_s^2}\left(\frac{\partial f_W}{\partial x_j}\right)-\nabla_{W_s^2}\left(\frac{\partial f_{\widetilde{W}}}{\partial x_j}\right)\right\|_F\\
&\leq\|\mathrm{diag}[\sigma'(f_{s,2}^{org}(W))]-\mathrm{diag}[\sigma'(f_{s,2}^{org}(\widetilde{W}))]\|_F\|W_s^3\|_F\\
&\quad\ \|W_s^{1,j}\|_F\|\mathrm{diag}[\sigma'(f_{s,1}^{org}(W)]\|_F\\
&\quad+\|\mathrm{diag}[\sigma'(f_{s,2}^{org}(\widetilde{W}))]\|_F\|W_s^3-(\widetilde{W})_s^3\|_F\|W_s^{1,j}\|_F\|\mathrm{diag}[\sigma'(f_{s,1}^{org}(W)]\|_F\\
&\quad+\|\mathrm{diag}[\sigma'(f_{s,2}^{org}(\widetilde{W}))]\|_F\|(\widetilde{W})_s^3\|_F\|W_s^{1,j}-(\widetilde{W})_s^{1,j}\|_F\|\mathrm{diag}[\sigma'(f_{s,1}^{org}(W)]\|_F\\
&\quad+\|\mathrm{diag}[\sigma'(f_{s,2}^{org}(\widetilde{W}))]\|_F\|(\widetilde{W})_s^3\|_F\\
&\qquad\ \|(\widetilde{W})_s^{1,j}\|_F\|\mathrm{diag}[\sigma'(f_{s,1}^{org}(W)]-\mathrm{diag}[\sigma'(f_{s,1}^{org}(\widetilde{W})]\|_F\\
&\quad+\|W_s^2-(\widetilde{W})_s^2\|_F\|\mathrm{diag}[\sigma'(f_{s,1}^{org}(W))]\|_F\|W_s^{1,j}\|_F\\
&\qquad\ \|\mathrm{diag}[\sigma''(f_{s,2}^{org}(W))]\|_F\|W_s^3\|_F\|f_s^1(W)\|_F\\
&\quad+\|(\widetilde{W})_s^2\|_F\|\mathrm{diag}[\sigma'(f_{s,1}^{org}(W))]-\mathrm{diag}[\sigma'(f_{s,1}^{org}(\widetilde{W}))]\|_F\|W_s^{1,j}\|_F\\
&\qquad\ \|\mathrm{diag}[\sigma''(f_{s,2}^{org}(W))]\|_F\|W_s^3\|_F\|f_s^1(W)\|_F\\
&\quad+\|(\widetilde{W})_s^2\|_F\|\mathrm{diag}[\sigma'(f_{s,1}^{org}(\widetilde{W}))]\|_F\|W_s^{1,j}-(\widetilde{W})_s^{1,j}\|_F\\
&\qquad\ \|\mathrm{diag}[\sigma''(f_{s,2}^{org}(W))]\|_F\|W_s^3\|_F\|f_s^1(W)\|_F\\
&\quad+\|(\widetilde{W})_s^2\|_F\|\mathrm{diag}[\sigma'(f_{s,1}^{org}(\widetilde{W}))]\|_F\|(\widetilde{W})_s^{1,j}\|_F\\
&\qquad\ \|\mathrm{diag}[\sigma''(f_{s,2}^{org}(W))]-\mathrm{diag}[\sigma''(f_{s,2}^{org}(\widetilde{W}))]\|_F\|W_s^3\|_F\|f_s^1(W)\|_F\\
&\quad+\|(\widetilde{W})_s^2\|_F\|\mathrm{diag}[\sigma'(f_{s,1}^{org}(\widetilde{W}))]\|_F\|(\widetilde{W})_s^{1,j}\|_F\\
&\qquad\ \|\mathrm{diag}[\sigma''(f_{s,2}^{org}(\widetilde{W}))]\|_F\|W_s^3-(\widetilde{W})_s^3\|_F\|f_s^1(W)\|_F\\
&\quad+\|(\widetilde{W})_s^2\|_F\|\mathrm{diag}[\sigma'(f_{s,1}^{org}(\widetilde{W}))]\|_F\|(\widetilde{W})_s^{1,j}\|_F\\
&\qquad\ \|\mathrm{diag}[\sigma''(f_{s,2}^{org}(\widetilde{W}))]\|_F\|(\widetilde{W})_s^3\|_F\|f_s^1(W)-f_s^1(\widetilde{W})\|_F\\
&\leq C(\sigma)(\sqrt{m_2}m_1\max\{\|W_s^{1,j}\|_F,\|(\widetilde{W})_s^{1,j}\|_F\}\|(\widetilde{W})_s^2\|_F\|W_s^3-(\widetilde{W})_s^3\|_F\\
&\quad+\max\{\sqrt{m_2},\sqrt{m_1}\}{m_1}\|W_s^3\|_F\\
&\qquad\max\{\|W_s^{1,j}\|_F,\|(\widetilde{W})_s^{1,j}\|_F\}\|(\widetilde{W})_s^2\|_F\|W_s^2-(\widetilde{W})_s^2\|_F\\
&\quad+{m_1}\|W_s^3\|_F\max\{\|W_s^{1,j}\|_F,\|(\widetilde{W})_s^{1,j}\|_F\}\|(\widetilde{W})_s^2\|_F\|b_s^2-(\widetilde{b})_s^2\|_2\\
&\quad+\sqrt{m_2}m_1\max\{\|W_s^{1,j}\|_F,\|(\widetilde{W})_s^{1,j}\|_F\}\|(\widetilde{W})_s^2\|_F^2\\
&\qquad\max\{\|W_s^{3}\|_F,\|(\widetilde{W})_s^{3}\|_F\}\|W_s^1-(\widetilde{W})_s^1\|_F\\
&\quad+\sqrt{m_2}m_1\|(\widetilde{W})_s^2\|_F\max\{\|W_s^{3}\|_F,\|(\widetilde{W})_s^{3}\|_F\}\|W_s^{1,j}-(\widetilde{W})_s^{1,j}\|_F\\
&\quad+\sqrt{m_2}m_1\max\{\|W_s^{1,j}\|_F,\|(\widetilde{W})_s^{1,j}\|_F\}\|(\widetilde{W})_s^2\|_F^2\\
&\qquad\max\{\|W_s^{3}\|_F,\|(\widetilde{W})_s^{3}\|_F\}\|b_s^1-(\widetilde{b})_s^1\|_2),\\
&\left\|\nabla_{b_s^2}\left(\frac{\partial f_W}{\partial x_j}\right)-\nabla_{b_s^2}\left(\frac{\partial f_{\widetilde{W}}}{\partial x_j}\right)\right\|_F\\
&\leq\|W_s^2-(\widetilde{W})_s^2\|_F\|\mathrm{diag}[\sigma'(f_{s,1}^{org}(W))]\|_F\|W_s^{1,j}\|_F\|\mathrm{diag}[\sigma''(f_{s,2}^{org}(W))]\|_F\|W_s^3\|_F\\
&\quad+\|(\widetilde{W})_s^2\|_F\|\mathrm{diag}[\sigma'(f_{s,1}^{org}(W))]-\mathrm{diag}[\sigma'(f_{s,1}^{org}(\widetilde{W}))]\|_F\\
&\qquad\ \|W_s^{1,j}\|_F\|\mathrm{diag}[\sigma''(f_{s,2}^{org}(W))]\|_F\|W_s^3\|_F\\
&\quad+\|(\widetilde{W})_s^2\|_F\|\mathrm{diag}[\sigma'(f_{s,1}^{org}(\widetilde{W}))]\|_F\|W_s^{1,j}-(\widetilde{W})_s^{1,j}\|_F\\
&\qquad\ \|\mathrm{diag}[\sigma''(f_{s,2}^{org}(W))]\|_F\|W_s^3\|_F\\
&\quad+\|(\widetilde{W})_s^2\|_F\|\mathrm{diag}[\sigma'(f_{s,1}^{org}(\widetilde{W}))]\|_F\|(\widetilde{W})_s^{1,j}\|_F\\
&\qquad\ \|\mathrm{diag}[\sigma''(f_{s,2}^{org}(W))]-\mathrm{diag}[\sigma''(f_{s,2}^{org}(\widetilde{W}))]\|_F\|W_s^3\|_F\\
&\quad+\|(\widetilde{W})_s^2\|_F\|\mathrm{diag}[\sigma'(f_{s,1}^{org}(\widetilde{W}))]\|_F\|(\widetilde{W})_s^{1,j}\|_F\\
&\qquad\ \|\mathrm{diag}[\sigma''(f_{s,2}^{org}(\widetilde{W}))]\|_F\|W_s^3-(\widetilde{W})_s^3\|_F\\
&\leq C(\sigma)(\sqrt{m_2}m_1\|(\widetilde{W})_s^2\|_F\|(\widetilde{W})_s^{1,j}\|_F\|W_s^3-(\widetilde{W})_s^3\|_F\\
&\quad+\max\{\sqrt{m_2},\sqrt{m_1}\}{m_1}\|W_s^3\|_F\|(\widetilde{W})_s^2\|_F\\
&\qquad\ \max\{\|W_s^{1,j}\|_F,\|(\widetilde{W})_s^{1,j}\|_F\}\|W_s^2-(\widetilde{W})_s^2\|_F\\
&\quad+{m_1}\|W_s^3\|_F\|(\widetilde{W})_s^2\|_F\|(\widetilde{W})_s^{1,j}\|_F\|b_s^2-(\widetilde{b})_s^2\|_2\\
&\quad+\sqrt{m_2}m_1\|(\widetilde{W})_s^2\|_F\|W_s^3\|_F\|W_s^{1,j}-(\widetilde{W})_s^{1,j}\|_F\\
&\quad+\max\{\sqrt{m_2},\sqrt{m_1}\}\|(\widetilde{W})_s^2\|_F\|W_s^3\|_F\|(\widetilde{W})_s^2\|_F\\
&\qquad\max\{\|W_s^{1,j}\|_F,\|(\widetilde{W})_s^{1,j}\|_F\}\|W_s^1-(\widetilde{W})_s^1\|_F\\
&\quad+\max\{\sqrt{m_2},\sqrt{m_1}\}\|(\widetilde{W})_s^2\|_F\|W_s^3\|_F\|(\widetilde{W})_s^2\|_F\\
&\qquad\max\{\|W_s^{1,j}\|_F,\|(\widetilde{W})_s^{1,j}\|_F\}\|b_s^1-(\widetilde{b})_s^1\|_2),\\
&\left\|\nabla_{W_s^1}\left(\frac{\partial f_W}{\partial x_j}\right)-\nabla_{W_s^1}\left(\frac{\partial f_{\widetilde{W}}}{\partial x_j}\right)\right\|_F\\
&\leq\|\mathrm{diag}[\sigma'(f_{s,1}^{org}(W))]-\mathrm{diag}[\sigma'(f_{s,1}^{org}(\widetilde{W}))]\|_F\|W_s^2\|_F\\
&\quad\ \|\mathrm{diag}[\sigma'(f_{s,2}^{org}(W))]\|_F\|W_s^3\|_F\\
&\quad+\|\mathrm{diag}[\sigma'(f_{s,1}^{org}(\widetilde{W}))]\|_F\|W_s^2-(\widetilde{W})_s^2\|_F\|\mathrm{diag}[\sigma'(f_{s,2}^{org}(W))]\|_F\|W_s^3\|_F\\
&\quad+\|\mathrm{diag}[\sigma'(f_{s,1}^{org}(\widetilde{W}))]\|_F\|(\widetilde{W})_s^2\|_F\\
&\qquad\ \|\mathrm{diag}[\sigma'(f_{s,2}^{org}(W))]-\mathrm{diag}[\sigma'(f_{s,2}^{org}(\widetilde{W}))]\|_F\|W_s^3\|_F\\
&\quad+\|\mathrm{diag}[\sigma'(f_{s,1}^{org}(\widetilde{W}))]\|_F\|(\widetilde{W})_s^2\|_F\|\mathrm{diag}[\sigma'(f_{s,2}^{org}(\widetilde{W}))]\|_F\|W_s^3-(\widetilde{W})_s^3\|_F\\
&\quad+\|\mathrm{diag}[\sigma''(f_{s,1}^{org}(W))]-\mathrm{diag}[\sigma''(f_{s,1}^{org}(\widetilde{W}))]\|_F\|W_s^{1,j}\|_F\|W_s^2\|_F\\
&\qquad\ \|\mathrm{diag}[\sigma'(f_{s,2}^{org}(W))]\|_F\|W_s^3\|_F\|x\|_2\\
&\quad+\|\mathrm{diag}[\sigma''(f_{s,1}^{org}(\widetilde{W}))]\|_F\|W_s^{1,j}-(\widetilde{W})_s^{1,j}\|_F\|W_s^2\|_F\\
&\qquad\ \|\mathrm{diag}[\sigma'(f_{s,2}^{org}(W))]\|_F\|W_s^3\|_F\|x\|_2\\
&\quad+\|\mathrm{diag}[\sigma''(f_{s,1}^{org}(\widetilde{W}))]\|_F\|(\widetilde{W})_s^{1,j}\|_F\|W_s^2-(\widetilde{W})_s^2\|_F\\
&\qquad\ \|\mathrm{diag}[\sigma'(f_{s,2}^{org}(W))]\|_F\|W_s^3\|_F\|x\|_2\\
&\quad+\|\mathrm{diag}[\sigma''(f_{s,1}^{org}(\widetilde{W}))]\|_F\|(\widetilde{W})_s^{1,j}\|_F\|(\widetilde{W})_s^2\|_F\\
&\qquad\ \|\mathrm{diag}[\sigma'(f_{s,2}^{org}(W))]-\mathrm{diag}[\sigma'(f_{s,2}^{org}(\widetilde{W}))]\|_F\|W_s^3\|_F\|x\|_2\\
&\quad+\|\mathrm{diag}[\sigma''(f_{s,1}^{org}(\widetilde{W}))]\|_F\|(\widetilde{W})_s^{1,j}\|_F\|(\widetilde{W})_s^2\|_F\\
&\qquad\ \|\mathrm{diag}[\sigma'(f_{s,2}^{org}(\widetilde{W}))]\|_F\|W_s^3-(\widetilde{W})_s^3\|_F\|x\|_2\\
&\quad+\|\mathrm{diag}[\sigma'(f_{s,1}^{org}(W))]-\mathrm{diag}[\sigma'(f_{s,1}^{org}(\widetilde{W}))]\|_F\|W_s^2\|_F\|W_s^2\|_F\\
&\qquad\ \|\mathrm{diag}[\sigma'(f_{s,1}^{org}(W))]\|_F\|W_s^{1,j}\|_F\|\mathrm{diag}[\sigma''(f_2^{org}(W))]\|_F\|W_s^3\|_F\|x\|_2\\
&\quad+\|\mathrm{diag}[\sigma'(f_{s,1}^{org}(\widetilde{W}))]\|_F\|W_s^2-(\widetilde{W})_s^2\|_F\|W_s^2\|_F\|\mathrm{diag}[\sigma'(f_{s,1}^{org}(W))]\|_F\\
&\qquad\ \|W_s^{1,j}\|_F\|\mathrm{diag}[\sigma''(f_2^{org}(W))]\|_F\|W_s^3\|_F\|x\|_2\\
&\quad+\|\mathrm{diag}[\sigma'(f_{s,1}^{org}(\widetilde{W}))]\|_F\|(\widetilde{W})_s^2\|_F\|W_s^2-(\widetilde{W})_s^2\|_F\|\mathrm{diag}[\sigma'(f_{s,1}^{org}(W))]\|_F\\
&\qquad\ \|W_s^{1,j}\|_F\|\mathrm{diag}[\sigma''(f_2^{org}(W))]\|_F\|W_s^3\|_F\|x\|_2\\
&\quad+\|\mathrm{diag}[\sigma'(f_{s,1}^{org}(\widetilde{W}))]\|_F\|(\widetilde{W})_s^2\|_F\|(\widetilde{W})_s^2\|_F\|\mathrm{diag}[\sigma'(f_{s,1}^{org}(W))]-\mathrm{diag}[\sigma'(f_{s,1}^{org}(\widetilde{W}))]\|_F\\
&\qquad\ \|W_s^{1,j}\|_F\|\mathrm{diag}[\sigma''(f_2^{org}(W))]\|_F\|W_s^3\|_F\|x\|_2\\
&\quad+\|\mathrm{diag}[\sigma'(f_{s,1}^{org}(\widetilde{W}))]\|_F\|(\widetilde{W})_s^2\|_F\|(\widetilde{W})_s^2\|_F\|\mathrm{diag}[\sigma'(f_{s,1}^{org}(\widetilde{W}))]\|_F\\
&\qquad\ \|W_s^{1,j}-(\widetilde{W})_s^{1,j}\|_F\|\mathrm{diag}[\sigma''(f_2^{org}(W))]\|_F\|W_s^3\|_F\|x\|_2\\
&\quad+\|\mathrm{diag}[\sigma'(f_{s,1}^{org}(\widetilde{W}))]\|_F\|(\widetilde{W})_s^2\|_F\|(\widetilde{W})_s^2\|_F\|\mathrm{diag}[\sigma'(f_{s,1}^{org}(\widetilde{W}))]\|_F\|(\widetilde{W})_s^{1,j}\|_F\\
&\qquad\ \|\mathrm{diag}[\sigma''(f_2^{org}(W))]-\mathrm{diag}[\sigma''(f_2^{org}(\widetilde{W}))]\|_F\|W_s^3\|_F\|x\|_2\\
&\quad+\|\mathrm{diag}[\sigma'(f_{s,1}^{org}(\widetilde{W}))]\|_F\|(\widetilde{W})_s^2\|_F\|(\widetilde{W})_s^2\|_F\|\mathrm{diag}[\sigma'(f_{s,1}^{org}(\widetilde{W}))]\|_F\|(\widetilde{W})_s^{1,j}\|_F\\
&\qquad\ \|\mathrm{diag}[\sigma''(f_2^{org}(\widetilde{W}))]\|_F\|W_s^3-(\widetilde{W})_s^3\|_F\|x\|_2\\
&\leq C(\sigma)(\sqrt{m_2}m_1\|(\widetilde{W})_s^2\|_F^2\|(\widetilde{W})_s^{1,j}\|_F\|W_s^3-(\widetilde{W})_s^3\|_F\\
&\quad+\sqrt{m_2}m_1^{3/2}\|W_s^3\|_F\max\{\|W_s^{1,j}\|_F,\|(\widetilde{W})_s^{1,j}\|_F\}\\
&\qquad\ \max\{\|\|W_s^2\|_F^2,\|(\widetilde{W})_s^2\|_F^2\}\|W_s^2-(\widetilde{W})_s^2\|_F\\
&\quad+{m_1}\|(\widetilde{W})_s^2\|_F^2\|W_s^3\|_F\|(\widetilde{W})_s^{1,j}\|_F\|b_s^2-(\widetilde{b})_s^2\|_2\\
&\quad+\sqrt{m_2}m_1\|W_s^3\|_F\max\{\|W_s^{1,j}\|_F,\|(\widetilde{W})_s^{1,j}\|_F\}\\
&\qquad\ \max\{\|W_s^2\|_F^3,\|(\widetilde{W})_s^2\|_F^3\}\|W_s^1-(\widetilde{W})_s^1\|_F\\
&\quad+\sqrt{m_2}m_1\|W_s^3\|_F\max\{\|W_s^2\|_F^2,\|(\widetilde{W})_s^2\|_F^2\}\|W_s^{1,j}-(\widetilde{W})_s^{1,j}\|_F\\
&\quad+\sqrt{m_2}m_1\|W_s^3\|_F\max\{\|W_s^{1,j}\|_F,\|(\widetilde{W})_s^{1,j}\|_F\}\\
&\qquad\ \max\{\|W_s^2\|_F^3,\|(\widetilde{W})_s^2\|_F^3\}\|b_s^1-(\widetilde{b})_s^1\|_F),\\
&\left\|\nabla_{b_s^1}\left(\frac{\partial f_W}{\partial x_j}\right)-\nabla_{b_s^1}\left(\frac{\partial f_{\widetilde{W}}}{\partial x_j}\right)\right\|_F\\
&\leq C(\sigma)(\sqrt{m_2}m_1\|(\widetilde{W})_s^2\|_F^2\|(\widetilde{W})_s^{1,j}\|_F\|W_s^3-(\widetilde{W})_s^3\|_F\\
&\quad+\sqrt{m_2}m_1^{3/2}\|W_s^3\|_F\max\{\|W_s^{1,j}\|_F,\|(\widetilde{W})_s^{1,j}\|_F\}\\
&\qquad\ \max\{\|\|W_s^2\|_F^2,\|(\widetilde{W})_s^2\|_F^2\}\|W_s^2-(\widetilde{W})_s^2\|_F\\
&\quad+{m_1}\|(\widetilde{W})_s^2\|_F^2\|W_s^3\|_F\|(\widetilde{W})_s^{1,j}\|_F\|b_s^2-(\widetilde{b})_s^2\|_2\\
&\quad+\sqrt{m_2}m_1\|W_s^3\|_F\max\{\|W_s^{1,j}\|_F,\|(\widetilde{W})_s^{1,j}\|_F\}\\
&\qquad\ \max\{\|W_s^2\|_F^3,\|(\widetilde{W})_s^2\|_F^3\}\|W_s^1-(\widetilde{W})_s^1\|_F\\
&\quad+\sqrt{m_2}m_1\|W_s^3\|_F\max\{\|W_s^2\|_F^2,\|(\widetilde{W})_s^2\|_F^2\}\|W_s^{1,j}-(\widetilde{W})_s^{1,j}\|_F\\
&\quad+\sqrt{m_2}m_1\|W_s^3\|_F\max\{\|W_s^{1,j}\|_F,\|(\widetilde{W})_s^{1,j}\|_F\}\\
&\qquad\ \max\{\|W_s^2\|_F^3,\|(\widetilde{W})_s^2\|_F^3\}\|b_s^1-(\widetilde{b})_s^1\|_F).
\end{align*}

In the final step, we deal with $\|\nabla_{W_s^l}\widehat{L}_R(f_W)-\nabla_{W_s^l}\widehat{L}_R(f_{\widetilde{W}})\|_F,\|\nabla_{b_s^l}\widehat{L}_R(f_W)-\nabla_{b_s^l}\widehat{L}_R(f_{\widetilde{W}})\|_2$ for $l=1,2,3$. To avoid excessive verbosity, we only present an estimate for $\|\nabla_{W_s^1}\widehat{L}_R(f_W)-\nabla_{W_s^1}\widehat{L}_R(f_{\widetilde{W}})\|_F$, and similar estimates can be obtained for the other terms. $\nabla_{W_s^1}\widehat{L}_R(f_W)-\nabla_{W_s^1}\widehat{L}_R(f_{\widetilde{W}})$ can be calculated by
\begin{align*}
&\nabla_{W_s^1}\widehat{L}_R(f_W)-\nabla_{W_s^1}\widehat{L}_R(f_{\widetilde{W}})\\
&=\frac{|\Omega|}{n}\sum_{i=1}^{n}\sum_{j=1}^{d}\left[\frac{\partial f_W(X_i)}{\partial x_j}\nabla_{W_s^1}\left(\frac{\partial f_W(X_i)}{\partial x_j}\right)-\frac{\partial f_{\widetilde{W}}(X_i)}{\partial x_j}\nabla_{W_s^1}\left(\frac{\partial f_{\widetilde{W}}(X_i)}{\partial x_j}\right)\right]\\
&\quad+\frac{|\Omega|}{n}\sum_{i=1}^{n}w(X_i)[f_W(X_i)\nabla_{W_s^1}f_W(X_i)-f_{\widetilde{W}}(X_i)\nabla_{W_s^1}f_{\widetilde{W}}(X_i)]\\
&\quad-\frac{|\Omega|}{n}\sum_{i=1}^{n}f(X_i)[\nabla_{W_s^1}f_W(X_i)-\nabla_{W_s^1}f_{\widetilde{W}}(X_i)]\\
&\quad+\frac{|\partial\Omega|}{\beta m}\sum_{i=1}^{m}[f_W(Y_i)\nabla_{W_s^1}f_W(Y_i)-f_{\widetilde{W}}(Y_i)\nabla_{W_s^1}f_{\widetilde{W}}(Y_i)]\\
&\quad-\frac{|\partial\Omega|}{\beta m}\sum_{i=1}^{m}g(Y_i)[\nabla_{W_s^1}f_W(Y_i)-\nabla_{W_s^1}f_{\widetilde{W}}(Y_i)].
\end{align*}
The first term on the right-hand side is bounded by
\begin{align*}
&\sum_{j=1}^d\left\|\frac{\partial f_W(X_i)}{\partial x_j}\nabla_{W_s^1}\left(\frac{\partial f_W(X_i)}{\partial x_j}\right)-\frac{\partial f_{\widetilde{W}}(X_i)}{\partial x_j}\nabla_{W_s^1}\left(\frac{\partial f_{\widetilde{W}}(X_i)}{\partial x_j}\right)\right\|_F\\
&\leq\sum_{j=1}^d\left|\frac{\partial f_W(X_i)}{\partial x_j}-\frac{\partial f_{\widetilde{W}}(X_i)}{\partial x_j}\right|\left\|\nabla_{W_s^1}\left(\frac{\partial f_{W}(X_i)}{\partial x_j}\right)\right\|_F\\
&\quad+\sum_{j=1}^d\left|\frac{\partial f_{\widetilde{W}}(X_i)}{\partial x_j}\right|\left\|\nabla_{W_s^1}\left(\frac{\partial f_W(X_i)}{\partial x_j}\right)-\nabla_{W_s^1}\left(\frac{\partial f_{\widetilde{W}}(X_i)}{\partial x_j}\right)\right\|_F\\
&\leq\|\nabla_xf_W(X_i)-\nabla_xf_{\widetilde{W}}(X_i)\|_2\left(\sum_{j=1}^d\left\|\nabla_{W_s^1}\left(\frac{\partial f_{W}(X_i)}{\partial x_j}\right)\right\|_F^2\right)^{1/2}\\
&\quad+\|\nabla_xf_{\widetilde{W}}(X_i)\|_2\left(\sum_{j=1}^d\left\|\nabla_{W_s^1}\left(\frac{\partial f_W(X_i)}{\partial x_j}\right)-\nabla_{W_s^1}\left(\frac{\partial f_{\widetilde{W}}(X_i)}{\partial x_j}\right)\right\|_F^2\right)^{1/2}.
\end{align*}
The fourth term on the right-hand side is bounded by(and there is a similar bound for the second term)
\begin{align*}
\|f_W\nabla_{W_s^1}f_W-f_{\widetilde{W}}\nabla_{W_s^1}f_{\widetilde{W}}\|_F\leq|f_W-f_{\widetilde{W}}|\|\nabla_{W_s^1}f_{W}\|_F
+|f_{\widetilde{W}}|\|\nabla_{W_s^1}f_W-\nabla_{W_s^3}f_{\widetilde{W}}\|_F
\end{align*}
Now, making use of the estimates of 
\begin{align*}
&|f_W|,|f_W-f_{\widetilde{W}}|,\|\nabla_xf_W\|_2,\|\nabla_xf_W-\nabla_xf_{\widetilde{W}}\|_2, \left\|\nabla_{W_s^l}f_{W}\right\|_F, \left\|\nabla_{b_s^l}f_{W}\right\|_2,\\
&\left\|\nabla_{W_s^l}f_{W}-\nabla_{W_s^l}f_{\widetilde{W}}\right\|_F,\left\|\nabla_{b_s^l}f_{W}-\nabla_{b_s^l}f_{\widetilde{W}}\right\|_2,\left\|\nabla_{W_s^l}\left(\frac{\partial f_W}{\partial x_j}\right)\right\|_F,\left\|\nabla_{b_s^l}\left(\frac{\partial f_W}{\partial x_j}\right)\right\|_2,\\
&\left\|\nabla_{W_s^l}\left(\frac{\partial f_W}{\partial x_j}\right)-\nabla_{W_s^l}\left(\frac{\partial f_{\widetilde{W}}}{\partial x_j}\right)\right\|_F,\left\|\nabla_{b_s^l}\left(\frac{\partial f_W}{\partial x_j}\right)-\nabla_{b_s^l}\left(\frac{\partial f_{\widetilde{W}}}{\partial x_j}\right)\right\|_F(l=1,2,3) 
\end{align*}
just obtained in the first four steps, we derive that 
\begin{align*}
&\|\nabla_{W_s^1}\widehat{L}_R(f_W)-\nabla_{W_s^1}\widehat{L}_R(f_{\widetilde{W}})\|_F\\
&\leq C(\sigma,f,w,\Omega)\sqrt{m_2}m_1\|W_s^{1}\|_F\|W_s^2\|_F^2\|W_s^3\|_F\\
&\quad\sum_{s'=1}^A[\sqrt{m_2m_1}\|(\widetilde{W})_{s'}^1\|_F\|(\widetilde{W})_{s'}^2\|_F\|W_{s'}^3-(\widetilde{W})_{s'}^3\|_F+|b_{s'}^3-(\widetilde{b})_{s'}^3|\\
&\quad\qquad+\max\{\sqrt{m_2},\sqrt{m_1}\}\sqrt{m_1}\|(\widetilde{W})_{s'}^1\|_F\|(\widetilde{W})_{s'}^2\|_F\\
&\qquad\qquad\max\{\|W_{s'}^3\|_F,\|(\widetilde{W})_{s'}^3\|_F\}\|{W}_{s'}^2-(\widetilde{W})_{s'}^2\|_{F}\\
&\quad\qquad+\sqrt{m_1}\|(\widetilde{W})_{s'}^1\|_F\|(\widetilde{W})_{s'}^2\|_F\max\{\|W_{s'}^3\|_F,\|(\widetilde{W})_{s'}^3\|_F\}\|{b}_{s'}^2-(\widetilde{b})_{s'}^2\|_{2}\\
&\quad\qquad+\sqrt{m_2m_1}\max\{\|W_{s'}^3\|_F,\|(\widetilde{W})_{s'}^3\|_F\}\|(\widetilde{W})_{s'}^1\|_F\\
&\qquad\qquad\max\{\|W_{s'}^2\|_F^2,\|(\widetilde{W})_{s'}^2\|_F^2\}\|{W}_{s'}^1-(\widetilde{W})_{s'}^1\|_{F}\\
&\quad\qquad+\max\{\sqrt{m_2},\sqrt{m_1}\}\|(\widetilde{W})_{s'}^1\|_F\max\{\|W_{s'}^3\|_F,\|(\widetilde{W})_{s'}^3\|_F\}\\
&\qquad\qquad\max\{\|W_{s'}^2\|_F^2,\|(\widetilde{W})_{s'}^2\|_F^2\}\|{b}_{s'}^1-(\widetilde{b})_{s'}^1\|_{2}]\\
&\quad+C(\sigma,f,w,\Omega)\left(\sum_{{s'}=1}^A\max\{\sqrt{m_2m_1}\|(\widetilde{W})_{s'}^1\|_2\|(\widetilde{W})_{s'}^2\|_2\|(\widetilde{W})_{s'}^3\|_2,|(\widetilde{b})_{s'}^3|\}\right)\\
&\qquad m_1\left(\sum_{j=1}^{d}\max\{\|W_s^{1,j}\|_F^2,\|(\widetilde{W})_s^{1,j}\|_F^2\}\right)^{1/2}\\
&\qquad(\sqrt{m_2}\|(\widetilde{W})_s^2\|_F^2\|W_s^3-(\widetilde{W})_s^3\|_F\\
&\qquad\ +\sqrt{m_2}m_1^{1/2}\|W_s^3\|_F\max\{\|\|W_s^2\|_F^2,\|(\widetilde{W})_s^2\|_F^2\}\|W_s^2-(\widetilde{W})_s^2\|_F\\
&\qquad\ +\|(\widetilde{W})_s^2\|_F^2\|W_s^3\|_F\|b_s^2-(\widetilde{b})_s^2\|_2\\
&\qquad\ +\sqrt{m_2}\|W_s^3\|_F\max\{\|W_s^2\|_F^3,\|(\widetilde{W})_s^2\|_F^3\}\|W_s^1-(\widetilde{W})_s^1\|_F\\
&\qquad\ +\sqrt{m_2}\|W_s^3\|_F\max\{\|W_s^2\|_F^3,\|(\widetilde{W})_s^2\|_F^3\}\|b_s^1-(\widetilde{b})_s^1\|_F)\\
&\quad+C(\sigma,g,\partial\Omega)\frac{1}{\beta}\sqrt{m_1m_2}\|W_s^2\|_F\|W_s^3\|_F\\
&\qquad\sum_{s'=1}^A({\sqrt{m_2}}\|{W}_{s'}^3-(\widetilde{W})_{s'}^3\|_{2}+|{b}_{s'}^3-(\widetilde{b})_{s'}^3|+{\sqrt{m_1}}\|(\widetilde{W})_{s'}^3\|_{2}\|{W}_{s'}^2-(\widetilde{W})_{s'}^2\|_{2}\\
&\qquad\qquad+\|(\widetilde{W})_{s'}^3\|_{2}\|{b}_{s'}^2-(\widetilde{b})_{s'}^2\|_{2}
+\|(\widetilde{W})_{s'}^3\|_{2}\|(\widetilde{W})_{s'}^2\|_{2}\|{W}_{s'}^1-(\widetilde{W})_{s'}^1\|_{2}\\
&\qquad\qquad+L_{\sigma}^2\|(\widetilde{W})_{s'}^3\|_{2}\|(\widetilde{W})_{s'}^2\|_{2}\|{b}_{s'}^1-(\widetilde{b})_{s'}^1\|_{2})\\
&\quad+C(\sigma,g,\partial\Omega)\frac{1}{\beta}\sum_{s'=1}^A(\sqrt{m_2}\|(\widetilde{W})_{s'}^3\|_F+|(\widetilde{b})_{s'}^3|)(\sqrt{m_2m_1}\|(\widetilde{W})_s^2\|_F\|W_s^3-(\widetilde{W})_s^3\|_F\\
&\qquad\qquad\qquad\qquad\qquad+\max\{\sqrt{m_2},\sqrt{m_1}\}\sqrt{m_1}\|W_s^3\|_F\|(\widetilde{W})_s^2\|_F\|{W}_s^2-(\widetilde{W})_s^2\|_{F}\\
&\qquad\qquad\qquad\qquad\qquad+\sqrt{m_1}\|(\widetilde{W})_s^2\|_F\|W_s^3\|_F\|{b}_s^2-(\widetilde{b})_s^2\|_{2}
&\qquad\qquad\qquad\qquad\qquad+\max\{\sqrt{m_2},\sqrt{m_1}\}\max\{\|W_s^2\|_F^2,\|(\widetilde{W})_s^2\|_F^2\}\|W_s^3\|_F\|{W}_s^1-(\widetilde{W})_s^1\|_{2}\\
&\qquad\qquad\qquad\qquad\qquad+\max\{\sqrt{m_2},\sqrt{m_1}\}\max\{\|W_s^2\|_F^2,\|(\widetilde{W})_s^2\|_F^2\}\\
&\qquad\qquad\qquad\qquad\qquad\quad\|W_s^3\|_F\|{b}_s^1-(\widetilde{b})_s^1\|_{2}).
\end{align*}
Combining the estimates of $\|\nabla_{W_s^l}\widehat{L}_R(f_W)-\nabla_{W_s^l}\widehat{L}_R(f_{\widetilde{W}})\|_F,\|\nabla_{b_s^l}\widehat{L}_R(f_W)-\nabla_{b_s^l}\widehat{L}_R(f_{\widetilde{W}})\|_2$ for $l=1,2,3$, we arrive at
\begin{align*}
&\|\nabla_{W}\widehat{L}_R(f_W)-\nabla_{W}\widehat{L}_R(f_{\widetilde{W}})\|^2\\
&\leq C(d,\sigma,f,g,w,\Omega)A^2\{1/\beta^2,1\}
\max_{1\leq s\leq A}[\max\{\|W_s^{1}\|_F^6,\|(\widetilde{W})_s^{1}\|_F^6\}\\
&\quad\max\{\|W_s^2\|_F^8,\|(\widetilde{W})_s^2\|_F^8\}\max\{\|W_s^3\|_F^6,\|(\widetilde{W})_s^3\|_F^6,|(\widetilde{b})_{s'}^3|^6\}]
\|W-\widetilde{W}\|^2.
\end{align*}
\end{proof}

\bibliographystyle{plain}

\bibliography{ref}

\end{document}